\documentclass[11pt]{article}

\usepackage[T1]{fontenc}
\usepackage{lmodern}
\usepackage[margin=1in]{geometry}
\usepackage{amsmath,amssymb,amsthm,mathtools}
\usepackage{enumitem}
\usepackage{booktabs,array}
\usepackage{xcolor}
\usepackage[protrusion=true,expansion=false]{microtype}
\usepackage{fvextra}
\usepackage{xurl}
\usepackage{authblk}

\usepackage{tikz}
\usetikzlibrary{arrows.meta,fit,positioning}

\usepackage[colorlinks=true,linkcolor=blue,citecolor=blue,urlcolor=blue]{hyperref}
\hypersetup{
  pdftitle={A Lean Formalization of Hamilton's Three-Manifold Theorem},
  pdfauthor={Bennett Chow, Yuan Liao, and Ziyang Qin}
}

\theoremstyle{plain}
\newtheorem{theorem}{Theorem}[section]
\newtheorem{lemma}[theorem]{Lemma}
\newtheorem{proposition}[theorem]{Proposition}
\newtheorem{corollary}[theorem]{Corollary}
\theoremstyle{definition}
\newtheorem{definition}[theorem]{Definition}
\theoremstyle{remark}
\newtheorem{remark}[theorem]{Remark}

\newcommand{\Ric}{\operatorname{Ric}}
\newcommand{\Rm}{\operatorname{Rm}}
\newcommand{\tr}{\operatorname{tr}}
\newcommand{\Div}{\operatorname{div}}
\newcommand{\Vol}{\operatorname{Vol}}

\newcommand{\Lie}{\mathcal L}
\newcommand{\ip}[2]{\left\langle #1,#2\right\rangle}
\newcommand{\abs}[1]{\left|#1\right|}
\newcommand{\norm}[1]{\left|#1\right|}

\newcommand{\eps}{\varepsilon}
\newcommand{\del}{\delta}
\newcommand{\releasecommit}{8bd406e35c33a200e9b88895cf11ee8429194e15}
\newcommand{\releaseblobbase}{https://github.com/qinz1yang/differential-geometry/blob/\releasecommit/DifferentialGeometry/}
\newcommand{\releasetreebase}{https://github.com/qinz1yang/differential-geometry/tree/\releasecommit/DifferentialGeometry/}
\newcommand{\declareleanfile}[3]{%
  \expandafter\def\csname leanfile@#1\endcsname{#2}%
  \expandafter\def\csname leanline@#1\endcsname{#3}}
\newcommand{\leanref}[1]{%
  \ifcsname leanfile@#1\endcsname
    \href{\releaseblobbase\csname leanfile@#1\endcsname\#L\csname leanline@#1\endcsname}{\nolinkurl{#1}}%
  \else
    \nolinkurl{#1}%
  \fi}
\newcommand{\leanrefas}[2]{%
  \ifcsname leanfile@#1\endcsname
    \href{\releaseblobbase\csname leanfile@#1\endcsname\#L\csname leanline@#1\endcsname}{\nolinkurl{#2}}%
  \else
    \nolinkurl{#2}%
  \fi}
\newcommand{\leanfile}[1]{\nolinkurl{#1}}
\newcommand{\rffile}[1]{\href{\releaseblobbase#1}{\leanfile{#1}}}
\newcommand{\rfdir}[1]{\href{\releasetreebase#1}{\leanfile{#1/}}}
\declareleanfile{BaseInjBound}{Geometry/Compactness/CheegerGromov/Pointed/InjectivityRadius.lean}{160}
\declareleanfile{CanonicalMetricCompactness}{Geometry/Flow/RicciFlow/Compactness/Metric/CanonicalConstruction.lean}{856}
\declareleanfile{CompleteInput}{Geometry/Flow/RicciFlow/Compactness/Foundations/Defs.lean}{114}
\declareleanfile{CurvBoundInput}{Geometry/Flow/RicciFlow/Compactness/Foundations/Defs.lean}{130}
\declareleanfile{ExtendsPastEndpoint}{Geometry/Flow/RicciFlow/Extension/MaximalTime.lean}{45}
\declareleanfile{FlowMetricBall.IsKappaNoncollapsed}{Geometry/Flow/RicciFlow/Perelman/Noncollapsing/Defs.lean}{60}
\declareleanfile{FlowMetricBall.IsRmControlled}{Geometry/Flow/RicciFlow/Perelman/Noncollapsing/Defs.lean}{55}
\declareleanfile{FlowTo}{Geometry/Flow/RicciFlow/Extension/MaximalFlow.lean}{25}
\declareleanfile{FlowTo.joint}{Geometry/Flow/RicciFlow/Extension/MaximalFlow.lean}{30}
\declareleanfile{FlowCover}{Geometry/Flow/RicciFlow/Extension/MaximalFlow.lean}{39}
\declareleanfile{FlowUpgrade}{Geometry/Flow/RicciFlow/Compactness/Limits/Upgrade.lean}{119}
\declareleanfile{FlowerScaleInjBound}{Geometry/Flow/RicciFlow/Compactness/Foundations/InjectivityRadius.lean}{18}
\declareleanfile{FlowerScaleVolData}{Geometry/Flow/RicciFlow/Compactness/Foundations/NoncollapseInjectivity.lean}{72}
\declareleanfile{Geometry.RoundQuotientData}{Geometry/Metric/Sphere/QuotientDescent.lean}{42}
\declareleanfile{HamiltonBlowup}{Geometry/Flow/RicciFlow/DimensionThree/PositiveRicci/Defs.lean}{95}
\declareleanfile{HamiltonFiniteTimeFlow}{Geometry/Flow/RicciFlow/DimensionThree/PositiveRicci/Defs.lean}{55}
\declareleanfile{HamiltonFiniteTimeFlow.curvUnbounded}{Geometry/Flow/RicciFlow/DimensionThree/PositiveRicci/Defs.lean}{60}
\declareleanfile{HasInjRadiusAt}{Geometry/Compactness/CheegerGromov/Pointed/InjectivityRadius.lean}{59}
\declareleanfile{HasPrincipalSymbol}{Analysis/Parabolic/DeTurckRicci/PrincipalSymbol.lean}{197}
\declareleanfile{IsFlowerScaleVolBound}{Geometry/Flow/RicciFlow/Compactness/Foundations/NoncollapseInjectivity.lean}{80}
\declareleanfile{IsFlowerScaleVolBound.curvature}{Geometry/Flow/RicciFlow/Compactness/Foundations/NoncollapseInjectivity.lean}{83}
\declareleanfile{IsFlowerScaleVolBound.noncollapsed}{Geometry/Flow/RicciFlow/Compactness/Foundations/NoncollapseInjectivity.lean}{97}
\declareleanfile{IsMaximalAtEndpoint}{Geometry/Flow/RicciFlow/Extension/MaximalTime.lean}{58}
\declareleanfile{IsMetricCompatible}{Geometry/Connection/LeviCivita/MetricCompatible.lean}{31}
\declareleanfile{IsSolutionOn}{Geometry/Flow/RicciFlow/Solution/Basic.lean}{381}
\declareleanfile{IsStrictlyParabolicMetricRHS}{Analysis/Parabolic/QuasiLinear/QuasilinearMetricShortTimeExistence.lean}{35}
\declareleanfile{LeviCivita}{Geometry/Connection/LeviCivita/Defs.lean}{180}
\declareleanfile{LeviCivita_chart_apply}{Geometry/Connection/LeviCivita/Defs.lean}{192}
\declareleanfile{LeviCivita_eq_leviCivitaConnectionOfMetric}{Geometry/Connection/LeviCivita/Defs.lean}{186}
\declareleanfile{LeviCivita_isMetricCompatible}{Geometry/Connection/LeviCivita/Defs.lean}{252}
\declareleanfile{LeviCivita_torsion_eq_zero}{Geometry/Connection/LeviCivita/Defs.lean}{239}
\declareleanfile{LeviCivita_unique}{Geometry/Connection/LeviCivita/Defs.lean}{311}
\declareleanfile{MetricCompactSeed.higherRegularityCanonicalMetricCompactness}{Geometry/Flow/RicciFlow/Compactness/Metric/HigherRegularityEndpoint.lean}{24}
\declareleanfile{MetricCompactnessConclusion}{Geometry/Compactness/CheegerGromov/Pointed/Compactness.lean}{1303}
\declareleanfile{PinchEigen3.q_sub_nonneg}{Geometry/Curvature/DimensionThree/PinchingAlgebra.lean}{323}
\declareleanfile{PinchEigen3Unordered.q_sub_nonneg}{Geometry/Curvature/DimensionThree/PinchingAlgebra.lean}{434}
\declareleanfile{PinchFlowWMPData.ofShiftNClosed}{Geometry/Flow/RicciFlow/Preservation/RicciPinching.lean}{991}
\declareleanfile{PinchInitLt}{Geometry/Flow/RicciFlow/Preservation/PositiveRicci.lean}{138}
\declareleanfile{PointedFlowData}{Geometry/Flow/RicciFlow/Compactness/Foundations/Defs.lean}{24}
\declareleanfile{PointedFlowSeq}{Geometry/Flow/RicciFlow/Compactness/Foundations/Defs.lean}{92}
\declareleanfile{RealTimeInterval.openInterval}{Analysis/TimeInterval.lean}{273}
\declareleanfile{RicciNormHeatEquationOn}{Geometry/Flow/RicciFlow/Solution/RicciNorm.lean}{220}
\declareleanfile{RiemannFromRicci3DTraceDataAt}{Geometry/Curvature/DimensionThree/RiemannFromRicci.lean}{515}
\declareleanfile{SmoothCGHConverges}{Geometry/Flow/RicciFlow/Compactness/Foundations/PointedMaps.lean}{1316}
\declareleanfile{SmoothRiemannianMetric}{Geometry/Metric/Basic.lean}{12}
\declareleanfile{SolutionAgreesOn}{Geometry/Flow/RicciFlow/Extension/MaximalTime.lean}{35}
\declareleanfile{SolutionOn}{Geometry/Flow/RicciFlow/Solution/Basic.lean}{168}
\declareleanfile{SphericalSpaceFormQuotientModel}{Geometry/Metric/Sphere/SpaceForm.lean}{28}
\declareleanfile{Tensor0SModel}{Tensor/RSTensor/Defs.lean}{33}
\declareleanfile{Tensor0SSpace}{Tensor/RSTensor/Defs.lean}{43}
\declareleanfile{TensorRSModel}{Tensor/RSTensor/Defs.lean}{38}
\declareleanfile{TensorRSSpace}{Tensor/RSTensor/Defs.lean}{196}
\declareleanfile{admitsConstantPositiveSectionalCurvature}{Geometry/Curvature/MetricConditions.lean}{33}
\declareleanfile{admitsPositiveRicci}{Geometry/Curvature/MetricConditions.lean}{23}
\declareleanfile{ccTensor02Symm}{Analysis/Spectral/Tensor/CovGrad/RicciDeTurckSectionDifferenceSymmetrizedReindexedCoeff.lean}{49}
\declareleanfile{chartChristoffel}{Geometry/Operator/Hessian.lean}{144}
\declareleanfile{chartLeviCivita}{Geometry/Connection/LeviCivita/LeviCivitaChartLocal.lean}{243}
\declareleanfile{chartLeviCivitaGoodSet}{Geometry/Connection/LeviCivita/LeviCivitaChartLocal.lean}{26}
\declareleanfile{chartLeviCivitaInnerCLM}{Geometry/Connection/LeviCivita/LeviCivitaChartLocal.lean}{223}
\declareleanfile{chartLeviCivita_chart_overlap}{Geometry/Connection/LeviCivita/Defs.lean}{53}
\declareleanfile{christoffelAlongInFrame}{Geometry/Connection/Chart/Christoffel.lean}{32}
\declareleanfile{christoffelCorrection}{Geometry/Connection/LeviCivita/LeviCivitaChartLocal.lean}{127}
\declareleanfile{compactnessSol}{Geometry/Flow/RicciFlow/Compactness/Limits/Hamilton.lean}{44}
\declareleanfile{component0S}{Tensor/RSTensor/Coordinates/CoordinateBasis.lean}{131}
\declareleanfile{componentRS}{Tensor/RSTensor/Coordinates/Components.lean}{62}
\declareleanfile{conjugating_diffeo_family_jointsmooth}{Geometry/Flow/RicciFlow/ShortTime/GaugeRecovery/ConjugatingDiffeoFamily.lean}{66}
\declareleanfile{connCoeff_eq_christoffelAlong_coord}{Geometry/Connection/Chart/NablaComponents/Basic.lean}{49}
\declareleanfile{connectionEndomorphismInChart}{Tensor/RSTensor/Derivation/NablaOnTensors.lean}{452}
\declareleanfile{constantPositiveSectionalCurvatureMetric}{Geometry/Curvature/MetricConditions.lean}{26}
\declareleanfile{constant_positive_sectional_curvature_iff_spherical_space_form}{Geometry/Metric/Sphere/SpaceForm.lean}{78}
\declareleanfile{constant_positive_sectional_curvature_implies_spherical_space_form}{Geometry/Metric/Sphere/SpaceForm.lean}{43}
\declareleanfile{constant_positive_sectional_curvature_of_smooth_cgh}{Geometry/Flow/RicciFlow/DimensionThree/PositiveRicci/Compactness/Limit.lean}{371}
\declareleanfile{coordComponent0SAt}{Geometry/Connection/Chart/CoordinateFrame.lean}{162}
\declareleanfile{coordComponentRSAt}{Geometry/Connection/Chart/CoordinateFrame.lean}{175}
\declareleanfile{coordInvEvol}{Geometry/Flow/RicciFlow/Evolution/Metric/Evolution.lean}{111}
\declareleanfile{coordRicciEvol}{Geometry/Flow/RicciFlow/Evolution/Ricci/CoordinateIdentities.lean}{828}
\declareleanfile{covariantDeriv_tensor0SModelWithin_apply_basis_slots}{Tensor/RSTensor/Derivation/NablaOnTensors.lean}{207}
\declareleanfile{cubicQ_pinchOn}{Geometry/Flow/RicciFlow/Preservation/Pinching/Definitions.lean}{151}
\declareleanfile{deTurckRicciRHS_hasPrincipalSymbol_at_self}{Analysis/Parabolic/DeTurckRicci/RHSStrictParabolic.lean}{417}
\declareleanfile{deTurckRicci_forcingBootstrap_symm}{Analysis/Spectral/Intrinsic/HeatSemigroup/MaxRegSolutionRegularity.lean}{4274}
\declareleanfile{deTurckSobolevNonlinearitySymm}{Analysis/Spectral/Intrinsic/DeTurck/SobolevNonlinearityExistence.lean}{1138}
\declareleanfile{deTurckSymbolCoeff}{Analysis/Parabolic/PrincipalSymbol.lean}{123}
\declareleanfile{deTurckSymbolCoeff_apply}{Analysis/Parabolic/PrincipalSymbol.lean}{127}
\declareleanfile{displayedRiemannFromRicci3D_of_algebraic_curvature_symmetries}{Geometry/Curvature/DimensionThree/CurvatureAlgebra.lean}{288}
\declareleanfile{evol_christoffel_inFrame}{Geometry/Flow/RicciFlow/Evolution/Connection/Evolution.lean}{345}
\declareleanfile{evol_inverse_metric_inFrame}{Geometry/Flow/RicciFlow/Evolution/Metric/Evolution.lean}{256}
\declareleanfile{exists_bounded_geometry_normal_data}{Geometry/Compactness/CheegerGromov/BoundedGeometry/NormalData.lean}{1002}
\declareleanfile{exists_hamilton_vol}{Geometry/Flow/RicciFlow/DimensionThree/PositiveRicci/Compactness/FlowUpgrade.lean}{1824}
\declareleanfile{exists_max_flow}{Geometry/Flow/RicciFlow/Extension/MaximalFlow.lean}{165}
\declareleanfile{exists_pinching_estimate_of_smooth_solution}{Geometry/Flow/RicciFlow/Preservation/Pinching/Estimate.lean}{989}
\declareleanfile{exists_uniform_jointly_smooth_ricciDeTurck_metric_solution}{Geometry/Flow/RicciFlow/ShortTime/LowRegularity/Background/Bootstrap.lean}{267}
\declareleanfile{extendInputs_of_soln}{Geometry/Flow/RicciFlow/Extension/ShiInputs.lean}{505}
\declareleanfile{extend_construction_of_restart}{Geometry/Flow/RicciFlow/Extension/Construction.lean}{158}
\declareleanfile{extends_of_rmBounded}{Geometry/Flow/RicciFlow/Extension/MaximalTime.lean}{144}
\declareleanfile{finiteTime3D}{Geometry/Flow/RicciFlow/Estimates/FiniteTime/Scalar.lean}{265}
\declareleanfile{flowInj_of_vol}{Geometry/Flow/RicciFlow/Compactness/Foundations/NoncollapseInjectivity.lean}{130}
\declareleanfile{flow_end_le}{Geometry/Flow/RicciFlow/Estimates/FiniteTime/Solution.lean}{99}
\declareleanfile{flow_to_agree}{Geometry/Flow/RicciFlow/Extension/MaximalFlow.lean}{60}
\declareleanfile{flow_to_eq}{Geometry/Flow/RicciFlow/Extension/MaximalFlow.lean}{87}
\declareleanfile{flow_to_extend}{Geometry/Flow/RicciFlow/Extension/MaximalFlow.lean}{96}
\declareleanfile{flow_to_seed}{Geometry/Flow/RicciFlow/Extension/MaximalFlow.lean}{46}
\declareleanfile{gluedFamily}{Geometry/Flow/RicciFlow/Extension/SmoothLimit.lean}{149}
\declareleanfile{hamiltonBlowupPointSelection}{Geometry/Flow/RicciFlow/DimensionThree/PositiveRicci/Blowup.lean}{47}
\declareleanfile{hamiltonCubicQ3}{Geometry/Curvature/DimensionThree/PinchingAlgebra.lean}{38}
\declareleanfile{hamiltonCubicQ3_factorized}{Geometry/Curvature/DimensionThree/PinchingAlgebra.lean}{51}
\declareleanfile{hamiltonCubicQ3_lower_bound_ordered_nonnegative_eigenvalues}{Geometry/Curvature/DimensionThree/PinchingAlgebra.lean}{225}
\declareleanfile{hamiltonCubicQFactorized3}{Geometry/Curvature/DimensionThree/PinchingAlgebra.lean}{45}
\declareleanfile{hamiltonPinchingEstimate}{Geometry/Flow/RicciFlow/DimensionThree/PositiveRicci/Blowup.lean}{152}
\declareleanfile{hamiltonRiemannCurvatureBound}{Geometry/Flow/RicciFlow/DimensionThree/PositiveRicci/Blowup.lean}{289}
\declareleanfile{hamiltonSourceSequence}{Geometry/Flow/RicciFlow/DimensionThree/PositiveRicci/Compactness/FlowUpgrade.lean}{232}
\declareleanfile{hamiltonWindow}{Geometry/Flow/RicciFlow/DimensionThree/PositiveRicci/Blowup.lean}{281}
\declareleanfile{hamilton_admits_constant_positive_sectional_curvature}{Geometry/Flow/RicciFlow/DimensionThree/PositiveRicci/Classification.lean}{167}
\declareleanfile{hamilton_constant_positive_sectional_curvature_of_injectivity_radius_bound}{Geometry/Flow/RicciFlow/DimensionThree/PositiveRicci/Classification.lean}{32}
\declareleanfile{hamilton_constant_positive_sectional_curvature_of_pinching}{Geometry/Flow/RicciFlow/DimensionThree/PositiveRicci/Classification.lean}{144}
\declareleanfile{hamilton_exists_blowup_point_sequence}{Geometry/Flow/RicciFlow/DimensionThree/PositiveRicci/Blowup.lean}{505}
\declareleanfile{hamilton_extinction_time_bound}{Geometry/Flow/RicciFlow/DimensionThree/PositiveRicci/Flow.lean}{537}
\declareleanfile{hamilton_finite_time_flow_exists_on_closed_open}{Geometry/Flow/RicciFlow/DimensionThree/PositiveRicci/Flow.lean}{34}
\declareleanfile{hamilton_flow_upgrade_of_metric_compactness}{Geometry/Flow/RicciFlow/DimensionThree/PositiveRicci/Compactness/FlowUpgrade.lean}{1392}
\declareleanfile{hamilton_pinching_implies_pinch_estimate}{Geometry/Flow/RicciFlow/DimensionThree/PositiveRicci/Blowup.lean}{977}
\declareleanfile{hamilton_positive_ricci}{Geometry/Flow/RicciFlow/DimensionThree/PositiveRicci/Hamilton.lean}{29}
\declareleanfile{hamilton_reference_radius}{Geometry/Flow/RicciFlow/DimensionThree/PositiveRicci/Defs.lean}{118}
\declareleanfile{hamilton_reference_radius_window}{Geometry/Flow/RicciFlow/DimensionThree/PositiveRicci/Blowup.lean}{1110}
\declareleanfile{hamilton_rescaled_curvature_bound}{Geometry/Flow/RicciFlow/DimensionThree/PositiveRicci/Blowup.lean}{1012}
\declareleanfile{hamilton_rescaled_ricci_nonnegative}{Geometry/Flow/RicciFlow/DimensionThree/PositiveRicci/Blowup.lean}{808}
\declareleanfile{hamilton_ricci_nonnegative}{Geometry/Flow/RicciFlow/DimensionThree/PositiveRicci/Blowup.lean}{756}
\declareleanfile{hamilton_scalar_blowup}{Geometry/Flow/RicciFlow/DimensionThree/PositiveRicci/Blowup.lean}{470}
\declareleanfile{hamilton_scalar_weak_maximum_principle_regularity_on_interval}{Geometry/Flow/RicciFlow/DimensionThree/PositiveRicci/Flow.lean}{89}
\declareleanfile{hamilton_tensor_wmp_section}{Analysis/Parabolic/MaximumPrinciple/Tensor/Weak.lean}{367}
\declareleanfile{hell_of_soln}{Geometry/Flow/RicciFlow/Extension/ShiInputs.lean}{109}
\declareleanfile{isClosedThreeManifold}{Topology/ThreeManifold/Closed.lean}{16}
\declareleanfile{isSolutionOn_of_extendData}{Geometry/Flow/RicciFlow/Extension/Regularity.lean}{1139}
\declareleanfile{isSphericalSpaceForm}{Geometry/Metric/Sphere/SpaceForm.lean}{39}
\declareleanfile{iteratedCovGrad}{Analysis/Sobolev/Embedding/SobolevEmbeddingCm.lean}{34}
\declareleanfile{koszulCovectorField}{Geometry/Connection/LeviCivita/KoszulFormula.lean}{447}
\declareleanfile{koszulCovectorField_apply_of_mdiff}{Geometry/Connection/LeviCivita/KoszulFormula.lean}{713}
\declareleanfile{koszulNablaAt}{Geometry/Connection/LeviCivita/KoszulFormula.lean}{698}
\declareleanfile{koszulNablaField}{Geometry/Connection/LeviCivita/KoszulFormula.lean}{459}
\declareleanfile{koszulScalar}{Geometry/Connection/LeviCivita/KoszulFormula.lean}{145}
\declareleanfile{koszul_identity}{Geometry/Connection/LeviCivita/Koszul.lean}{118}
\declareleanfile{koszul_local_uniqueness}{Geometry/Connection/LeviCivita/Koszul.lean}{139}
\declareleanfile{leviCivitaConnectionCandidateAt}{Geometry/Connection/LeviCivita/KoszulFormula.lean}{894}
\declareleanfile{leviCivitaConnectionCandidateAt_agreesWithDescended}{Geometry/Connection/LeviCivita/KoszulFormula.lean}{974}
\declareleanfile{leviCivitaConnectionOfMetric}{Geometry/Connection/LeviCivita/KoszulFormula.lean}{988}
\declareleanfile{leviCivitaConnectionOfMetric_contMDiffCovariantDerivativeLocally}{Geometry/Connection/LeviCivita/Smooth/Connection.lean}{63}
\declareleanfile{leviCivitaConnectionOfMetric_isMetricCompatible}{Geometry/Connection/LeviCivita/KoszulFormula.lean}{1075}
\declareleanfile{leviCivitaConnectionOfMetric_isTorsionFree}{Geometry/Connection/LeviCivita/Torsion.lean}{372}
\declareleanfile{leviCivitaStitched}{Geometry/Connection/LeviCivita/Defs.lean}{98}
\declareleanfile{lichnerowiczRHSInFrame}{Geometry/Flow/RicciFlow/Evolution/Ricci/Lichnerowicz.lean}{195}
\declareleanfile{limit_to_orig}{Geometry/Flow/RicciFlow/DimensionThree/PositiveRicci/LimitRoundness.lean}{345}
\declareleanfile{metricCompactness}{Geometry/Flow/RicciFlow/Compactness/Metric/Endpoint.lean}{48}
\declareleanfile{metricCov}{Geometry/Curvature/Metric.lean}{43}
\declareleanfile{metricCovectorNormSq_pos}{Analysis/Parabolic/PrincipalSymbol.lean}{88}
\declareleanfile{metricEquiv_of_ricBound}{Geometry/Flow/RicciFlow/Extension/Construction.lean}{354}
\declareleanfile{metricRicciAt}{Geometry/Curvature/Metric.lean}{84}
\declareleanfile{metricRm04StdAt}{Geometry/Curvature/Metric.lean}{123}
\declareleanfile{metricSeedOfBG}{Geometry/Flow/RicciFlow/Compactness/Metric/BoundedGeometryCompactness.lean}{24}
\declareleanfile{modelDeriv_eq_coordDeriv0SAt}{Geometry/Connection/Chart/NablaComponents/Basic.lean}{207}
\declareleanfile{nabla0SFun}{Tensor/RSTensor/Derivation/NablaOnTensors.lean}{855}
\declareleanfile{nablaRSFun}{Tensor/RSTensor/Derivation/NablaOnTensors.lean}{865}
\declareleanfile{Nabla0SRegular}{Tensor/RSTensor/Derivation/NablaOnTensors.lean}{899}
\declareleanfile{NablaRSRegular}{Tensor/RSTensor/Derivation/NablaOnTensors.lean}{911}
\declareleanfile{nabla0S_coordFrame_slots}{Geometry/Connection/Chart/NablaComponents/Basic.lean}{324}
\declareleanfile{nabla0S_coordFrame_slots_of_smooth}{Geometry/Connection/Chart/NablaComponents/Basic.lean}{392}
\declareleanfile{nablaRS_coordFrame_slots_of_smooth}{Geometry/Coordinates/NablaComponents/TensorRS/Formula.lean}{175}
\declareleanfile{negative_region_parabolic_lower_bound}{Analysis/Parabolic/MaximumPrinciple/Scalar/Weak.lean}{496}
\declareleanfile{no_local_open}{Geometry/Flow/RicciFlow/Perelman/Noncollapsing/EarlyTime.lean}{90}
\declareleanfile{normSqLeOfFirstTrace}{Geometry/Curvature/DimensionThree/RicciControlsRm.lean}{805}
\declareleanfile{parabolic_const_sub}{Analysis/Parabolic/MaximumPrinciple/Scalar/Weak.lean}{47}
\declareleanfile{pinchBarrierReg}{Geometry/Flow/RicciFlow/Preservation/RicciPinching.lean}{615}
\declareleanfile{pinchParabolic}{Geometry/Flow/RicciFlow/Preservation/RicciPinching.lean}{82}
\declareleanfile{pinchQuot_slab_bound}{Geometry/Flow/RicciFlow/Preservation/Pinching/Estimate.lean}{763}
\declareleanfile{pinchQuotient_grad_pos}{Geometry/Flow/RicciFlow/Preservation/Pinching/Estimate.lean}{688}
\declareleanfile{pinchQuotient_initial_bound}{Geometry/Flow/RicciFlow/Preservation/Pinching/Estimate.lean}{416}
\declareleanfile{pinchQuotient_parabolic_nonpos}{Geometry/Flow/RicciFlow/Preservation/Pinching/Estimate.lean}{289}
\declareleanfile{pinchQuotient_space_pos}{Geometry/Flow/RicciFlow/Preservation/Pinching/Estimate.lean}{652}
\declareleanfile{pinchShiftNull_ge}{Geometry/Flow/RicciFlow/Preservation/PositiveRicciReaction.lean}{189}
\declareleanfile{pinch_init_sol_lt}{Geometry/Flow/RicciFlow/Preservation/RicciPinching.lean}{1327}
\declareleanfile{pinch_init_wmp_lt}{Geometry/Flow/RicciFlow/Preservation/RicciPinching.lean}{1304}
\declareleanfile{pinch_quotient_evolution_of_heat_equations}{Geometry/Flow/RicciFlow/Preservation/Pinching/HamiltonReaction.lean}{48}
\declareleanfile{pinch_quotient_evolution_of_solution_data}{Geometry/Flow/RicciFlow/Preservation/Pinching/SolutionEvolution.lean}{793}
\declareleanfile{pinch_quotient_evolution_of_tensor_sections}{Geometry/Flow/RicciFlow/Preservation/Pinching/HamiltonReaction.lean}{945}
\declareleanfile{pinch_sol_closed}{Geometry/Flow/RicciFlow/Preservation/RicciPinching.lean}{1197}
\declareleanfile{pinch_sol_closed_nonneg}{Geometry/Flow/RicciFlow/Preservation/RicciPinching.lean}{1222}
\declareleanfile{positiveRicciMetric}{Geometry/Curvature/MetricConditions.lean}{19}
\declareleanfile{quasilinear_maxreg_solution_of_nemytskii}{Analysis/Spectral/Intrinsic/DeTurck/DeTurckQuasilinearExistence.lean}{543}
\declareleanfile{quasilinear_metric_short_time_existence_of_nemytskii_data}{Geometry/Flow/RicciFlow/ShortTime/Construction/QuasilinearExistence.lean}{31}
\declareleanfile{quotHeat}{Geometry/Flow/RicciFlow/Preservation/Pinching/QuotientEvolution.lean}{231}
\declareleanfile{quotHeat1_of_nonneg}{Geometry/Flow/RicciFlow/Preservation/Pinching/QuotientEvolution.lean}{367}
\declareleanfile{quotHeatDiv}{Geometry/Flow/RicciFlow/Preservation/Pinching/QuotientEvolution.lean}{591}
\declareleanfile{reaction_difference_lower_bound_on_negative_region}{Analysis/Parabolic/MaximumPrinciple/Scalar/Weak.lean}{482}
\declareleanfile{ric_quad_le_of_soln}{Geometry/Flow/RicciFlow/Extension/ShiInputs.lean}{372}
\declareleanfile{ricciGradCoupleSq}{Geometry/Flow/RicciFlow/Preservation/Pinching/HamiltonReaction.lean}{125}
\declareleanfile{ricciGradSq}{Geometry/Flow/RicciFlow/Solution/Basic.lean}{334}
\declareleanfile{ricciHeatSmooth}{Geometry/Flow/RicciFlow/Evolution/Ricci/NormEvolution.lean}{598}
\declareleanfile{ricciMixed_eq_gradNorm}{Geometry/Flow/RicciFlow/Preservation/Pinching/HamiltonReaction.lean}{200}
\declareleanfile{ricciMixed_eq_tfGrad}{Geometry/Flow/RicciFlow/Preservation/Pinching/HamiltonReaction.lean}{317}
\declareleanfile{ricciNorm}{Geometry/Flow/RicciFlow/Solution/Basic.lean}{327}
\declareleanfile{ricciNormLap}{Geometry/Flow/RicciFlow/Solution/Basic.lean}{356}
\declareleanfile{ricciNorm_slabCont}{Geometry/Flow/RicciFlow/Preservation/Pinching/Estimate.lean}{566}
\declareleanfile{ricciReact}{Geometry/Flow/RicciFlow/Solution/Basic.lean}{373}
\declareleanfile{ricci_flow_forward_unique}{Geometry/Flow/RicciFlow/Extension/Construction.lean}{126}
\declareleanfile{ricci_flow_interior_restart}{Geometry/Flow/RicciFlow/Extension/Construction.lean}{61}
\declareleanfile{ricci_flow_pde_at_zero}{Geometry/Flow/RicciFlow/ShortTime/GaugeRecovery/RicciFlowPdeAtZero.lean}{61}
\declareleanfile{ricci_flow_short_time_existence}{Geometry/Flow/RicciFlow/ShortTime/Existence.lean}{35}
\declareleanfile{ricci_flow_uniform_existence}{Geometry/Flow/RicciFlow/Extension/Construction.lean}{30}
\declareleanfile{ricci_gauge_of_dt}{Geometry/Flow/RicciFlow/ShortTime/LowRegularity/GaugeRemoval.lean}{141}
\declareleanfile{ricci_nonnegative_of_closed_solution_wmp_data}{Geometry/Flow/RicciFlow/Preservation/RicciPinching.lean}{1249}
\declareleanfile{rm04Comp_displayedRiemannFromRicci3D_at}{Geometry/Curvature/DimensionThree/RiemannFromRicci.lean}{624}
\declareleanfile{rm04Realizes_metric}{Geometry/Flow/RicciFlow/Extension/MaximalTime.lean}{77}
\declareleanfile{rmBounded_of_not_unbounded}{Geometry/Flow/RicciFlow/Extension/MaximalTime.lean}{129}
\declareleanfile{rmUnbounded_of_maximal}{Geometry/Flow/RicciFlow/Extension/MaximalTime.lean}{277}
\declareleanfile{round_at_zero_of_smooth_cgh}{Geometry/Flow/RicciFlow/DimensionThree/PositiveRicci/Compactness/Limit.lean}{265}
\declareleanfile{scalarLowerBarrier}{Geometry/Flow/RicciFlow/Preservation/ScalarLowerBound.lean}{30}
\declareleanfile{scalarLowerBarrier_hasDerivWithinAt}{Geometry/Flow/RicciFlow/Preservation/ScalarLowerBound.lean}{44}
\declareleanfile{scalarRegOfSmooth}{Geometry/Flow/RicciFlow/Preservation/ScalarLowerBound.lean}{271}
\declareleanfile{scalar_curvature_lower_bound_of_parabolic_inequality}{Geometry/Flow/RicciFlow/Preservation/ScalarLowerBound.lean}{117}
\declareleanfile{scalar_curvature_lower_bound_of_scalarEvolution}{Geometry/Flow/RicciFlow/Preservation/ScalarLowerBound.lean}{525}
\declareleanfile{scalar_curvature_lower_bound_of_scalarEvolution_initialMinimum}{Geometry/Flow/RicciFlow/Preservation/ScalarLowerBound.lean}{670}
\declareleanfile{scalar_endpoint_le_blowupTime_of_lower_barrier_bound}{Geometry/Flow/RicciFlow/Estimates/FiniteTime/Scalar.lean}{101}
\declareleanfile{scalar_evolution_of_smooth_solution}{Geometry/Flow/RicciFlow/Evolution/Scalar/Basic.lean}{91}
\declareleanfile{scalar_wmp_sub_theorem_7_2}{Analysis/Parabolic/MaximumPrinciple/Scalar/Weak.lean}{2258}
\declareleanfile{scalar_wmp_super_theorem_7_1}{Analysis/Parabolic/MaximumPrinciple/Scalar/Weak.lean}{2213}
\declareleanfile{scalar_wmp_supersolutions_of_lipschitz_on_value_set_of_regular_positive_time}{Analysis/Parabolic/MaximumPrinciple/Scalar/Weak.lean}{2167}
\declareleanfile{shiCovBound_of_soln}{Geometry/Flow/RicciFlow/Extension/ShiInputs.lean}{405}
\declareleanfile{shiftNAt}{Geometry/Flow/RicciFlow/Preservation/PositiveRicciReaction.lean}{1050}
\declareleanfile{shiftNRaw}{Geometry/Flow/RicciFlow/Preservation/PositiveRicciReaction.lean}{1060}
\declareleanfile{shiftNRaw_null_symm}{Geometry/Flow/RicciFlow/Preservation/PositiveRicciReaction.lean}{1581}
\declareleanfile{shiftNull3}{Geometry/Flow/RicciFlow/Preservation/PositiveRicciReaction.lean}{180}
\declareleanfile{shiftReact3_nonneg}{Geometry/Flow/RicciFlow/Preservation/PositiveRicciReaction.lean}{211}
\declareleanfile{shiftScal3_eq}{Geometry/Flow/RicciFlow/Preservation/PositiveRicciReaction.lean}{165}
\declareleanfile{solutionOn_of_joint}{Geometry/Flow/RicciFlow/Extension/Regularity.lean}{1064}
\declareleanfile{spherical_space_form_admits_constant_positive_sectional_curvature}{Geometry/Metric/Sphere/SpaceForm.lean}{55}
\declareleanfile{strictCert_sec}{Analysis/Parabolic/MaximumPrinciple/Tensor/Certification.lean}{286}
\declareleanfile{strict_barrier_nonnegative_of_positive_time}{Analysis/Parabolic/MaximumPrinciple/Scalar/Weak.lean}{668}
\declareleanfile{strict_pinch_sol_lt}{Geometry/Flow/RicciFlow/Preservation/RicciPinching.lean}{1562}
\declareleanfile{tangentConstAt}{Geometry/Connection/LeviCivita/KoszulFormula.lean}{32}
\declareleanfile{tensor0SModelAt_coordComponent0SAt}{Geometry/Connection/Chart/NablaComponents/Basic.lean}{142}
\declareleanfile{tensorRSBundle_fiber}{Tensor/RSTensor/Defs.lean}{647}
\declareleanfile{tensorRSBundle_smooth}{Tensor/RSTensor/Defs.lean}{670}
\declareleanfile{tensorRSBundle_topology}{Tensor/RSTensor/Defs.lean}{638}
\declareleanfile{tensorRSBundle_vector}{Tensor/RSTensor/Defs.lean}{658}
\declareleanfile{tensorSectionRealizeMetric_inner}{Analysis/Spectral/Intrinsic/MetricRealization/TensorHsRealize.lean}{372}
\declareleanfile{tensorHs}{Analysis/Spectral/Tensor/SobolevScale/Defs.lean}{143}
\declareleanfile{covGrad}{Analysis/Spectral/Tensor/CovGrad/Defs.lean}{164}
\declareleanfile{tensor_wmp}{Analysis/Parabolic/MaximumPrinciple/Tensor/Weak.lean}{321}
\declareleanfile{traceFreeRicciNormSq}{Geometry/Flow/RicciFlow/Preservation/Pinching/Definitions.lean}{85}
\declareleanfile{trace_free_ricci_norm_sq_heat_equation_of_smooth_solution}{Geometry/Flow/RicciFlow/Preservation/Pinching/SolutionEvolution.lean}{488}
\declareleanfile{trace_free_ricci_norm_sq_heat_equation_of_solution}{Geometry/Flow/RicciFlow/Preservation/Pinching/SolutionEvolution.lean}{564}
\declareleanfile{vec2}{Geometry/Curvature/Tensor.lean}{131}
\declareleanfile{wmp_section_sec}{Analysis/Parabolic/MaximumPrinciple/Tensor/Weak.lean}{251}
\makeatletter
\g@addto@macro{\UrlBreaks}{\do\_\do\.\do\/}
\makeatother
\title{A Lean Formalization of Hamilton's Three-Manifold Theorem}
\author[1]{Bennett Chow}
\author[1]{Yuan Liao}
\author[2]{Ziyang Qin}
\affil[1]{Department of Mathematics, University of California San Diego}
\affil[2]{Department of Mathematics, Cornell University}
\date{\today}

\begin{document}
\maketitle

\begin{abstract}

We describe a Lean formalization of Hamilton's 1982 theorem on
closed, connected three-manifolds with positive Ricci curvature.  The
development contains a short-time existence theorem for Ricci flow and
substantial geometric-analysis infrastructure: Riemannian tensor calculus,
the Levi--Civita connection, Ricci-flow evolution equations, scalar and tensor
maximum principles, three-dimensional curvature algebra, preservation of
Ricci pinching, and Hamilton's improved pinching estimate.  The formalization follows an alternative blow-up route, rather than Hamilton's
original normalized-flow proof.  Its time-uniform short-time existence,
maximal continuation, no-local-collapsing, and
Cheeger--Gromov--Hamilton compactness pipelines have been formalized and are
included in the artifact, while we give only a brief account of
these companion developments and record the interfaces and
consequences used by the Hamilton argument; a detailed exposition of their full
constructions is deferred to the second author's forthcoming thesis.    We interweave
representative Lean declarations with their mathematical meaning and record
the status and provenance of every major component.  All source-level status
claims are tied to the source release identified below.
\end{abstract}

\noindent\textbf{Keywords.}
Ricci flow; positive Ricci curvature; three-manifolds; geometric analysis;
formalized mathematics; Lean.

\medskip
\noindent\textbf{Mathematics Subject Classification (2020).}
Primary 53E20, 68V20; Secondary 35K59, 53C44.

\tableofcontents

\section{Introduction}

The Ricci flow is the equation
\begin{equation}
    \frac{\partial}{\partial t} g = - 2 \Ric
\end{equation}
for one-parameter families of Riemannian metrics $g(t)$ on smooth manifolds, where $\Ric$ denotes the Ricci tensor of $g(t)$.
It was introduced by Richard Hamilton in his seminal 1982 paper
\cite{MR664497}, where he used it to prove the following geometric-topological classification
result.
\begin{theorem}[Hamilton 1982]
\label{thm:main-hamilton-3d}
   If $M^3$ is a closed, connected smooth three-manifold that admits a metric
   $g_0$ with positive Ricci curvature, then $M^3$ admits a metric $g_\infty$
   with constant positive sectional curvature.
\end{theorem}
Equivalently, the conclusion is that such a manifold is diffeomorphic to a spherical space form
\cite{MR2742530}, \cite[Corollary 12.5]{LeeRiemannianManifolds}.  This
theorem is the starting point of the Ricci-flow approach to
three-dimensional topology.

Table~\ref{tab:status-at-a-glance} separates the mathematical target from the
major producer and consumer layers in its checked Lean proof.  The
Hamilton assembly has the surface statement of
Theorem~\ref{thm:main-hamilton-3d}, and its fresh transitive axiom audit contains
only Lean/\texttt{mathlib}'s standard classical and quotient axioms.
\begin{table}[htbp]
\centering
\scriptsize
\renewcommand{\arraystretch}{1.12}
\begin{tabular}{@{}>{\raggedright\arraybackslash}p{0.32\linewidth}
                    >{\raggedright\arraybackslash}p{0.63\linewidth}@{}}
\toprule
\textbf{Component and role} & \textbf{Provenance and representative declarations} \\
\midrule
Hamilton's 3-dimensional positive Ricci theorem; mathematical target & Classical theorem \cite{MR664497}; Lean endpoint \leanref{hamilton_positive_ricci}. \\
\addlinespace[0.4em]
Short-time existence from a smooth initial metric & Project-local over Lean/\texttt{mathlib}; \leanref{ricci_flow_short_time_existence}. \\
\addlinespace[0.4em]
Time-uniform short-time existence & Project-local class-first fixed-background Ricci--DeTurck construction; \leanref{ricci_flow_uniform_existence}. \\
\addlinespace[0.4em]
Local geometry, evolution, maximum principles, and pinching & Project-local layers; representative declarations are cited beside their mathematical explanations. \\
\addlinespace[0.4em]
No-local-collapsing and injectivity radius control & Companion development included in the release; \leanref{no_local_open}, \leanref{exists_hamilton_vol}, and \leanref{flowInj_of_vol}. \\
\addlinespace[0.4em]
Riemannian metric and Ricci-flow compactness theorems & Companion developments included in the release; \leanref{metricCompactness}, \leanref{compactnessSol}, and \leanref{hamilton_constant_positive_sectional_curvature_of_pinching}. \\
\addlinespace[0.4em]
Topological endpoint: constant positive curvature implies quotient of sphere & Final theorems; \leanref{hamilton_admits_constant_positive_sectional_curvature} and \leanref{hamilton_positive_ricci}. \\
\bottomrule
\end{tabular}
\caption{Major components of the checked proof and their representative
declarations.  The release identifies the exact source revision for these
names and provenance claims.}
\label{tab:status-at-a-glance}
\end{table}

The checked route culminates in two declarations.  The theorem
\leanref{hamilton_admits_constant_positive_sectional_curvature} has the metric conclusion of
Theorem~\ref{thm:main-hamilton-3d}: existence of a metric of constant positive
sectional curvature.  The theorem \leanref{hamilton_positive_ricci} adds the
spherical-space-form conclusion.  Their proof assemblies, including uniform
existence, maximal continuation, noncollapsing, and compactness, are checked with exactly the transitive axioms
\texttt{propext}, \texttt{Classical.choice}, and \texttt{Quot.sound}.

During the 1980s and 1990s, Hamilton transformed the Ricci flow introduced in
his 1982 paper into a far-reaching program for proving Thurston's
geometrization conjecture, and hence the Poincar\'e conjecture.  In a sequence
of works, Hamilton developed tensor maximum principles and curvature-pinching
estimates, differential Harnack inequalities,
surface entropy,
compactness and blow-up methods
for analyzing singularities and their ancient or eternal models, a surgery
procedure in the four-dimensional positive-isotropic-curvature setting that
served as a model for the required three-dimensional theory, and an analysis
of nonsingular long-time three-dimensional flows, thereby formulating the
strategy of evolving an arbitrary metric on a closed three-manifold, cutting
along necklike regions to continue through singularities, and identifying the
long-time thick and thin pieces predicted by geometrization
\cite{MR664497,MR862046,MR954419,MR1198607,MR1231700,HamiltonCompactness,
Hamilton1995SDG,Ham97,MR1714939,CCCY}.
During this period, other mathematicians that contributed to this program include Bryant \cite{Bryant}, Ivey \cite{Ivey1993}, and Shi \cite{Shi1989a}.

In 2002--2003, Perelman introduced entropy and reduced-volume
monotonicity, proved no-local-collapsing and pseudolocality estimates,
a spacetime geometry,
the qualitative classification of singularity models,
developed canonical-neighborhood analysis,
a delicate modification of Hamilton's Ricci flow with surgery procedure,
established finite-time extinction for closed three-manifolds whose prime
decomposition has no aspherical factors, and stated without proof a local collapsing theorem asserting that sufficiently
locally collapsed three-manifolds with local lower sectional-curvature
bounds are graph manifolds
\cite{Perelman1,Perelman2,Perelman3}.
Perelman's advances supplied the essential
missing ingredients in Hamilton's program to form the Hamilton--Perelman proof of the geometrization and the
Poincar\'e conjectures.
Detailed expositions were subsequently given by Kleiner and Lott
\cite{MR2460872}, Morgan and Tian \cite{MR2334563,MR3186136}, and Cao and Zhu
\cite{MR2233789}.
A detailed exposition of Perelman's collapsing result was given by Shioya and
Yamaguchi \cite{MR2169831}.
An alternative proof of finite-time extinction was given by Colding and Minicozzi \cite{MR2460871}.
The Hamilton--Perelman theory in turn drew on a broad body of work in
geometric PDE, comparison geometry, Riemannian convergence and collapse,
and three-manifold topology.
Among its direct antecedents were Eells--Sampson's harmonic-map heat flow,
DeTurck's gauge-fixing argument, Li--Yau differential Harnack estimates,
Cheeger--Gromov convergence and collapse theory, and Thurston's
geometrization program.

Ricci flow has since become a central tool in geometric
analysis, with applications including the differentiable sphere theorem of
Brendle and Schoen \cite{BrendleSchoen}, the spherical space form theorem of
B\"ohm and Wilking \cite{BohmWilking}, and later work of Brendle and Bamler on
problems arising from Perelman's program \cite{Brendle2018U,Bamler0}.
Using Ricci flow through singularities, Bamler and Kleiner proved the
generalized Smale conjecture for spherical space forms and obtained further
results on diffeomorphism groups of three-manifolds
\cite{BamlerKleinerContractibility,BamlerKleinerDiffGroups}.  Bamler's work has
also developed a high-dimensional Ricci-flow theory
\cite{Bamler2020A,MR4623543,Bamler2020C}; his 2021 survey gives an overview of
these and related developments \cite{BamlerNotices2021}.
The Ricci flow continues to be an active field at the present time with Bamler's 4-manifold program being a major area of study.

The formalization of mathematics is the project of writing mathematical
definitions, statements, and proofs in languages whose correctness can be
checked by computer.  Early milestones include de Bruijn's Automath,
Trybulec's Mizar project, and Milner's LCF approach to interactive theorem
proving \cite{deBruijnAutomath,TrybulecMizar,GordonMilnerWadsworthLCF}.
Later systems such as HOL, Coq (now the Rocq Prover), Isabelle, and Lean developed powerful
foundational frameworks and extensive libraries of formalized mathematics
\cite{GordonHOL,Coq,PaulsonIsabelle,deMouraLean}.  Major formalization
achievements include Gonthier's formal proof of the four-color theorem, the Coq
formalization of the Feit--Thompson odd order theorem, and Hales's Flyspeck
formal proof of the Kepler conjecture
\cite{GonthierFourColor,GonthierOddOrder,HalesFlyspeck}.  More recently,
Lean's mathematical library \texttt{mathlib} and projects such as the Liquid
Tensor Experiment have shown that proof assistants can engage directly with
contemporary research mathematics \cite{mathlib,CommelinLiquidTensor}.
In mathematical physics, Douglas, Hoback, Mei, and Nissim formalized in
Lean~4 the construction of the free massive bosonic field in
four-dimensional Euclidean spacetime and proved that it satisfies the
Glimm--Jaffe formulation of the Osterwalder--Schrader axioms
\cite{douglas2026formalization}.
At a still larger scale, the FLT project led by Kevin Buzzard, following a
route developed with Richard Taylor, is formalizing a modern proof of Fermat's
Last Theorem in Lean \cite{BuzzardTaylorFLT}.

Geometric analysis is a natural and demanding frontier for formalization.
Modern proofs in the subject combine long chains of local tensor calculations,
analytic estimates, compactness arguments, limiting procedures, and global
geometric or topological inputs.  Hamilton's theorem is thus a useful
test case.  His original proof uses curvature evolution, scalar and tensor
maximum principles, three-dimensional curvature algebra, pinching estimates,
and convergence of the normalized flow.  Our final proof endpoint instead assembles a later-developed blow-up route, which also uses Perelman's
no-local-collapsing theorem and Cheeger--Gromov--Hamilton compactness.  A Lean
formalization forces the hypotheses and interfaces in either route to be made
explicit, and it separates reusable geometric-analysis infrastructure from the
argument specific to the three-manifold theorem.

Recent Lean formalizations have begun to reach geometry, analysis, and
geometric analysis more directly.  The higher-order Fr\'echet calculus
underlying \texttt{mathlib}'s smoothness predicates is described by
Gou\"ezel \cite{GouezelHigherOrderCalculus}; its design supports, among other
applications, manifolds with boundary and corners, manifolds over general
normed fields, and infinite-dimensional Banach manifolds.  Bordg and Cavalleri
give an early account of the formalization of Lie groups, vector bundles, and
Lie algebras in Lean \cite{BordgCavalleriDGLean}.

The manifold and Riemannian-geometry infrastructure at the current time is a community
achievement.  Rothgang's survey \cite{RothgangDGMathlib} records major
contributions by Floris van Doorn, S\'ebastien Gou\"ezel, Yury Kudryashov,
Heather Macbeth, Patrick Massot, and Winston Yin, together with substantial
review work by Oliver Nash, and describes ongoing joint work of Macbeth,
Massot, and Rothgang on the Levi--Civita connection, curvature, and geodesics.
The sphere-eversion project of van Doorn, Massot, and Nash formalized a
substantial part of Gromov's \(h\)-principle and deduced Smale's sphere
eversion theorem \cite{MassotVanDoornNashSphereEversion}.

In analysis, the Carleson project
formalized a generalized Carleson theorem on doubling metric measure spaces and
recovered the classical theorem as a corollary \cite{CarlesonLean}.  Armstrong
and Kempe formalized the core interior De Giorgi--Nash--Moser theory for weak
solutions of rough-coefficient elliptic equations \cite{ArmstrongKempeDGNM}.
Together, these developments suggest that the infrastructure needed for Ricci
flow is beginning to come within reach of Lean.

In \texttt{mathlib}, the
generalized Poincar\'e conjecture and the three-dimensional topological and
smooth Poincar\'e conjectures have already been expressed as Lean statements
\cite{XuPoincareMathlib}.  The Lean Millennium Prize Problems project likewise
gives Clay-aligned formal statements of the Millennium problems, including the
Poincar\'e conjecture \cite{LeanMillenniumPrizeProblems}.  Those projects are
important reference points, and the current project could serve as the first milestone in formalizing the proof of the Poincar\'e conjecture.  We focus on Hamilton's
specific Ricci-flow theorem at the beginning of the Hamilton--Perelman program,
and on the geometric-analytic proof infrastructure needed to support its proof.

\begin{samepage}
\medskip
\noindent
The Lean source is available at
\begin{center}
  \url{https://github.com/qinz1yang/differential-geometry}
\end{center}
The version described in this paper is
\href{https://github.com/qinz1yang/differential-geometry/releases/tag/arxiv-v1-preview}{release \texttt{0.1.0}},
corresponding to commit
\href{https://github.com/qinz1yang/differential-geometry/commit/8bd406e35c33a200e9b88895cf11ee8429194e15}{\nolinkurl{8bd406e35}};
it uses Lean~4 v4.29.0 and \texttt{mathlib} v4.29.0.  Full release details and
source links are given in Subsection~\ref{subsec:source-release}.
\medskip
\end{samepage}

At the methodological level, Kontorovich has proposed a human-supervised
multi-agent framework for ``quasi-autoformalization,'' organized around
decomposer, translator, solver, and conductor roles
\cite{KontorovichShape}.  Section~\ref{sec:auto-formalization-workflow}
describes the related top-down workflow used in the present project and its
project-specific dependency-management and verification mechanisms.

The source release contains $1,945,081$ lines of tracked Lean source, including
$57{,}317$ lines in the vendored \rfdir{External} tree; dependency packages and
build products are excluded.  This paper describes the current
state of that formalization. Our goals
are the following.
\begin{itemize}[leftmargin=2em]
  \item Give an exact release-bound account of the proof architecture and
  dependency closure, distinguishing project-local work from separately
  organized companion developments.
  \item Explain the project-local differential-geometric infrastructure,
  including tensor bundles, the Levi-Civita connection, induced tensor
  covariant derivatives, and the passage from invariant definitions to
  coordinate formulas.
  \item Record the closed one-metric short-time Ricci-flow theorem and the
    Ricci--DeTurck, spectral-Sobolev,
    regularity, and
    gauge-removal architecture behind it, including the quantitative upgrade
    to one common lifetime for a $C^3$-bounded metric class.
  \item Present the evolution equations, maximum-principle layers,
  three-dimensional curvature algebra, Ricci-preservation, and
  improved-pinching layers implemented by the project, then explain how the
    noncollapsing and compactness companion developments feed the
  Hamilton blow-up assembly.
\end{itemize}
The purpose is therefore not merely to display a final Lean declaration.  It
is to explain the geometric argument in a Lean-facing way and to state which
portions of the proof are internal to this development and which are organized
as companion formalizations, while documenting that the selected endpoint now
has no \texttt{sorryAx} or project-specific axiom in its transitive axiom report.

The paper is written for two overlapping audiences.  Differential geometers
may read each displayed Lean signature through the accompanying mathematical
explanation; Lean users may read the surrounding argument as the intended
meaning and scope of the interface.  Section~2 motivates the formalization and
fixes the target, conventions, and proof architecture.  We then develop the
differential-geometric foundation, the spectral proof of short-time existence,
the evolution and maximum-principle layers, and the checked blow-up
assembly.  The final sections delimit future library generalization work and describe
the development and verification methodology.  A directory-like catalogue of
declaration names is deliberately omitted: representative declarations are
cited where their types and mathematical roles are actually discussed.

\section{Motivation, formal target, and proof architecture}

\subsection{Motivation: reusable infrastructure and scalable verification}
\label{subsec:formalization-motivation}

Formalization serves here not only to produce a checked endpoint for Hamilton's
theorem, but also to prepare for a verification problem likely to become acute as
AI-assisted proof production scales.  AI systems are likely soon to generate
formal proofs in volumes far beyond the capacity of humans to review them line by
line.  In that setting, exhaustive human review of every generated proof script
cannot serve as a realistic acceptance mechanism.  Assurance must instead rest
on a small trusted kernel, explicit dependency and axiom audits, and human
scrutiny of the definitions and theorem statements that connect the formal
artifact to its intended mathematics.

At present, the Lean/\texttt{mathlib} ecosystem does not yet provide mature,
reusable interfaces for substantial parts of geometric analysis.  Replacing a
missing geometric-analytic argument by a project-specific axiom, a
\texttt{sorry}-based placeholder, or an opaque external proof oracle does not
eliminate the verification obligation; it moves the unverified mathematics to
that boundary.  Accordingly, a primary objective of this project is to make its
differential-geometric and analytic infrastructure as mathematically natural,
general, and reusable as practical.  During top-down proof development, minimal
interfaces first expose the exact needs of the Hamilton argument; stable results
can then be promoted to general-purpose library statements, allowing later
formalizations to reuse checked mathematics rather than introduce new
unverified assumptions.  We intend to continue maintaining and updating the
repository as Lean and \texttt{mathlib} evolve, improving these interfaces,
generalizing results where mathematically natural, and making the infrastructure
easier for other formalization projects to use.

In the longer term, this suggests a model of AI-assisted mathematics in which
agents produce Lean proofs whose transitive dependency closures contain no
project-specific unproved axioms or external proof oracles, Lean's kernel checks
the resulting proof terms, and natural-language accounts are generated from the
checked artifact for mathematical communication.  The natural-language
rendering is an expository layer rather than a source of formal validity.  Human
responsibility remains concentrated at the semantic boundary: choosing the
intended definitions and statements, judging the mathematical significance and
generality of the interfaces, and checking that the formal claims correspond to
the mathematics being advertised.

\subsection{Formal endpoint and global conventions}
\label{subsec:main-lean-endpoint}

All declaration names, source paths, and completion claims in this subsection
refer to the exact source release identified in
Subsection~\ref{subsec:source-release}.  We call a displayed endpoint
\emph{closed} only when its exact signature elaborates, its proof compiles, and
its transitive axiom report contains no \texttt{sorryAx} or project-specific
axiom.  Section~\ref{sec:artifact-verification} records the corresponding
evidence; this theorem-specific claim does not assert that every unrelated or
imported declaration in the repository is placeholder-free.

The mathematical target is Theorem~\ref{thm:main-hamilton-3d}.  We track two Lean conclusions.  The declaration
\leanref{hamilton_admits_constant_positive_sectional_curvature}
states that the manifold admits a metric of constant positive sectional
curvature.  The augmented declaration \leanref{hamilton_positive_ricci}, in
\rffile{Geometry/Flow/RicciFlow/DimensionThree/PositiveRicci/Hamilton.lean},
adds the spherical-space-form conclusion.  Its signature uses
\leanref{isClosedThreeManifold} and \leanref{admitsPositiveRicci}; its two
conclusions are \leanref{admitsConstantPositiveSectionalCurvature} and
\leanref{isSphericalSpaceForm}, with the final implication supplied by
\leanref{constant_positive_sectional_curvature_iff_spherical_space_form}:

\begin{Verbatim}[breaklines=true,breakanywhere=true,fontsize=\small]
theorem hamilton_positive_ricci
    (hM : isClosedThreeManifold (I := I) (M := M))
    (hpos : admitsPositiveRicci (I := I) (M := M)) :
    admitsConstantPositiveSectionalCurvature (I := I) (M := M) /\
      isSphericalSpaceForm (I := I) (M := M) := by
  haveI : I.Boundaryless := hM.2.2.1
  haveI : NeZero (Module.finrank Real E) := ⟨by
    rw [hM.2.2.2]
    norm_num⟩
  have hconst : admitsConstantPositiveSectionalCurvature (I := I) (M := M) :=
    hamilton_admits_constant_positive_sectional_curvature (I := I) (M := M) hM hpos
  exact ⟨hconst,
    (constant_positive_sectional_curvature_iff_spherical_space_form
      (I := I) (M := M) hM).1 hconst⟩
\end{Verbatim}

The surface signature contains only the manifold package and the existence of
a positive-Ricci metric.  The first expands to compactness, connectedness,
boundarylessness, and real model-space dimension three.  The second uses the
canonical Ricci tensor of the chosen smooth metric; Ricci curvature is not an
independent input.  The two conclusions are respectively the metric conclusion
and its global topological corollary.

The spherical-space-form conclusion is represented by an explicit quotient
model rather than by an unstructured classification predicate. A value of
\leanref{SphericalSpaceFormQuotientModel} consists of
\leanref{Geometry.RoundQuotientData} for the unit round three-sphere together
with a diffeomorphism from (M) to the resulting quotient manifold. The
quotient data records a finite group $\Gamma$, an orthogonal action of
$\Gamma$ on $\mathbb{R}^4$, a smooth projection
\[
S^3 \longrightarrow Q,
\]
whose fibers are precisely the $\Gamma$-orbits, and smooth local sections of
this projection. Thus \leanref{isSphericalSpaceForm} packages the global
conclusion that (M) is diffeomorphic to a finite round quotient of $S^3$.
For a closed connected three-manifold this formulation is formally equivalent
to the metric conclusion: \leanref{constant_positive_sectional_curvature_implies_spherical_space_form} constructs the quotient
model from a constant-positive-sectional-curvature metric, while
\leanref{spherical_space_form_admits_constant_positive_sectional_curvature} pulls the descended round metric back along
the identifying diffeomorphism. These two directions are assembled in
\leanref{constant_positive_sectional_curvature_iff_spherical_space_form}, which is the final step used by
\leanref{hamilton_positive_ricci}.

\begin{Verbatim}[
  commandchars=\\\{\},
  breaklines=true,
  breakanywhere=true,
  fontsize=\small
]
structure SphericalSpaceFormQuotientModel
    (I : ModelWithCorners Real E H) (N : Type u)
    [TopologicalSpace N] [ChartedSpace H N] : Type _ where
  data :
    Geometry.RoundQuotientData.{0, u, u}
      (EuclideanSpace Real (Fin 4)) 3
  equiv : \ensuremath{N \simeq{m} {\langle I,\mathcal{R} \;{3}\rangle}\mathsf{data.Q}}
def isSphericalSpaceFormQuotient
    (I : ModelWithCorners Real E H) (N : Type u)
    [TopologicalSpace N] [ChartedSpace H N] : Prop :=
  Nonempty (SphericalSpaceFormQuotientModel I N)

def isSphericalSpaceForm : Prop :=
  isSphericalSpaceFormQuotient I M
\end{Verbatim}
The displayed proof body is the final packaging and thus short.  Its
dependency chain contains the completed uniform-existence, maximal-flow,
noncollapsing, compactness, transfer, and global-topology developments.

For the Hamilton argument, $M$ is closed and connected.  The short-time theorem
of Section~\ref{sec:short-time-existence} is dimension-independent and does not
assume connectedness.  Our sign and scaling conventions are
\[
  \partial_tg=-2\Ric,
  \qquad
  \Delta T=g^{ij}\nabla_i\nabla_jT,
  \qquad
  \Rm_{04}(X,Y,Z,W)=\langle R(X,Y)Z,W\rangle.
\]
Indices are raised and lowered by the evolving metric.  Parabolic rescaling by
a factor $Q>0$ sends $g(t)$ to
$\widetilde g(s)=Q g(t_0+s/Q)$; scalar curvature then scales by $Q^{-1}$.
All pullbacks use the same curvature convention, and derivatives at an initial
time are understood relative to the appropriate one-sided time set.

\subsection{Choice of the global proof route}
\label{sec:proof-variants}

Hamilton's original proof evolves the volume-normalized Ricci flow for all
time.  Preservation and improvement of Ricci pinching are combined with a
gradient estimate of the form
\[
  |\nabla R|^2\le \eta R^3+C_\eta
\]
and higher derivative estimates to show that the normalized metrics converge
smoothly to constant positive sectional curvature
\cite{MR664497}.  This route gives a stronger asymptotic conclusion than the
bare existence of a constant-curvature metric, but its formalization would
require a long-time normalized-flow theory, the gradient estimate, and a
smooth convergence argument.

The formal proof instead follows a finite-time blow-up
route, which is useful for singularity analysis in general.
A scalar lower bound forces finite maximal time; an extension criterion
and three-dimensional curvature control produce scalar blow-up; point
selection and parabolic rescaling produce a normalized sequence; and
noncollapsing plus Cheeger--Gromov--Hamilton compactness produce a complete
limit.  Preserved Ricci nonnegativity and improved pinching are proved to pass to the limit,
forcing it to be Einstein, and hence of constant positive sectional curvature since we are in dimension 3.
Myers compactness and the convergence embeddings then transfer the constant curvature metric to
the original manifold.
\subsection{Proof and dependency map}
\label{subsec:proof-dependency-map}

For readers who want the global argument before the implementation details,
the proof assembled in the paper has the following structure.

\begin{enumerate}[leftmargin=2em]
  \item The direct metric target \leanref{hamilton_admits_constant_positive_sectional_curvature} states the
  constant-positive-curvature conclusion for a closed, connected
  three-manifold admitting a positive-Ricci metric.  The augmented endpoint
  \leanref{hamilton_positive_ricci} also states the spherical-space-form conclusion.

  \item The one-metric short-time theorem constructs a Ricci flow from any
  fixed smooth initial metric and is closed in the theorem-specific sense
  audited in Section~\ref{sec:artifact-verification}.  The stronger theorem
  \leanref{ricci_flow_uniform_existence} selects a common lifetime before
  the initial metric varies in a controlled class.
  \item The tensor-calculus, curvature-evolution, maximum-principle,
  three-dimensional algebra, Ricci-preservation, improved-pinching,
  continuation, finite-time, and point-selection layers are built on our tensor definitions.
  \item Scalar blow-up and point selection produce basepoint-scalar-normalized
  parabolic rescalings with scale-independent curvature control.
  \item The noncollapsing and Cheeger--Gromov--Hamilton projects
  supply volume control, injectivity control, metric compactness, the smooth
  Ricci-flow upgrade, and the complete pointed blow-up limit.  The selected
  route uses \leanref{exists_hamilton_vol}, \leanref{flowInj_of_vol},
  \leanref{metricCompactness}, and
  \leanref{hamilton_flow_upgrade_of_metric_compactness}.
  \item Improved pinching passes to the limit and forces the terminal metric to
  be Einstein.  The scalar normalization, contracted Bianchi identity, and
  three-dimensional curvature identity give constant positive sectional
  curvature.  Compactness and the convergence embeddings transfer this metric
  to the original manifold, after which the spherical-space-form conclusion
  follows.
\end{enumerate}
Sections~\ref{sec:conditional-assembly} give the
detailed assembly and compare it with Hamilton's original normalized-flow
proof.  The intervening sections explain the checked project-local and
companion layers in dependency order.

For the dependency graph of the current checked blow-up route, see Figure \ref{fig:hamilton-dependency-graph}.

\begin{figure}[htbp]
\centering
\begin{tikzpicture}[
  node distance=5mm and 8mm,
  depbox/.style={
    draw,
    rounded corners,
    align=center,
    font=\scriptsize,
    text width=43mm,
    inner sep=4pt
  },
  hypothesis/.style={depbox, fill=white},
  spine/.style={depbox, fill=orange!12, thick},
  support/.style={depbox, fill=blue!10},
  endpoint/.style={
    depbox,
    fill=orange!12,
    thick,
    double,
    double distance=0.7pt
  },
  dependency/.style={-{Stealth[length=2mm]}, semithick}
]

\node[hypothesis] (initial)
  {Closed three-manifold with a positive-Ricci initial metric};

\node[spine, below=of initial] (short)
  {Short-time Ricci-flow existence};

\node[spine, below=of short] (maximal)
  {Forward uniqueness, and supremum
   maximal-flow construction};

\node[support, right=of maximal] (maxneed)
  {Shi derivative of curvature bounds, time-uniform existence, interior restart, and
   time-extension gluing};

\node[spine, below=of maximal] (blowup)
  {Maximal solution curvature unboundedness and scalar-curvature blow-up};

\node[spine, below=of blowup] (selection)
  {Point selection and basepoint-scalar-normalized parabolic rescalings};

\node[spine, below=of selection] (limit)
  {Complete pointed blow-up limit};

\node[spine, below=of limit] (einstein)
  {3D positive Einstein blow-up limit, hence constant positive sectional
   curvature};

\node[endpoint, below=of einstein] (metric)
  {Constant-positive-sectional-curvature metric on the original
   manifold\\
   \leanref{hamilton_admits_constant_positive_sectional_curvature}};

\node[endpoint, below=of metric] (spherical)
  {Spherical-space-form conclusion\\
   \leanref{hamilton_positive_ricci}};

\node[support, left=of blowup] (local)
  {Evolution equations, weak maximum principles, three-dimensional
   algebra, Ricci preservation, scalar estimates, and improved pinching};

\node[support, right=of limit] (compactness)
  {No-local-collapsing, Shi estimates, injectivity-radius
   control, and Cheeger--Gromov--Hamilton compactness};

\node[support, left=of einstein] (transfer)
  {Ricci and pinching estimates transfer to the limit, limit scalar positivity};

\node[support, right=of metric] (globalize)
  {Bonnet--Myers compactness of the limit, globalization of the
   convergence embedding, and transfer of the limit metric to the
   original manifold};

\node[support, right=of spherical] (topology)
  {Final topological handoff\\
   \leanref{constant_positive_sectional_curvature_iff_spherical_space_form}};

\draw[dependency] (initial) -- (short);
\draw[dependency] (maxneed) -- (blowup);
\draw[dependency] (short) -- (maximal);
\draw[dependency] (maximal) -- (blowup);
\draw[dependency] (blowup) -- (selection);
\draw[dependency] (selection) -- (limit);
\draw[dependency] (limit) -- (einstein);
\draw[dependency] (einstein) -- (metric);
\draw[dependency] (metric) -- (spherical);

\draw[dependency] (local) -- (blowup);
\draw[dependency] (local) -- (transfer);
\draw[dependency] (compactness) -- (limit);
\draw[dependency] (transfer) -- (einstein);

\draw[dependency] (globalize) -- (metric);
\draw[dependency] (topology) -- (spherical);

\end{tikzpicture}

\caption{High-level dependency graph of the checked blow-up route
assembled by the project.  An arrow indicates that the target uses the
source.  }
\label{fig:hamilton-dependency-graph}
\end{figure}
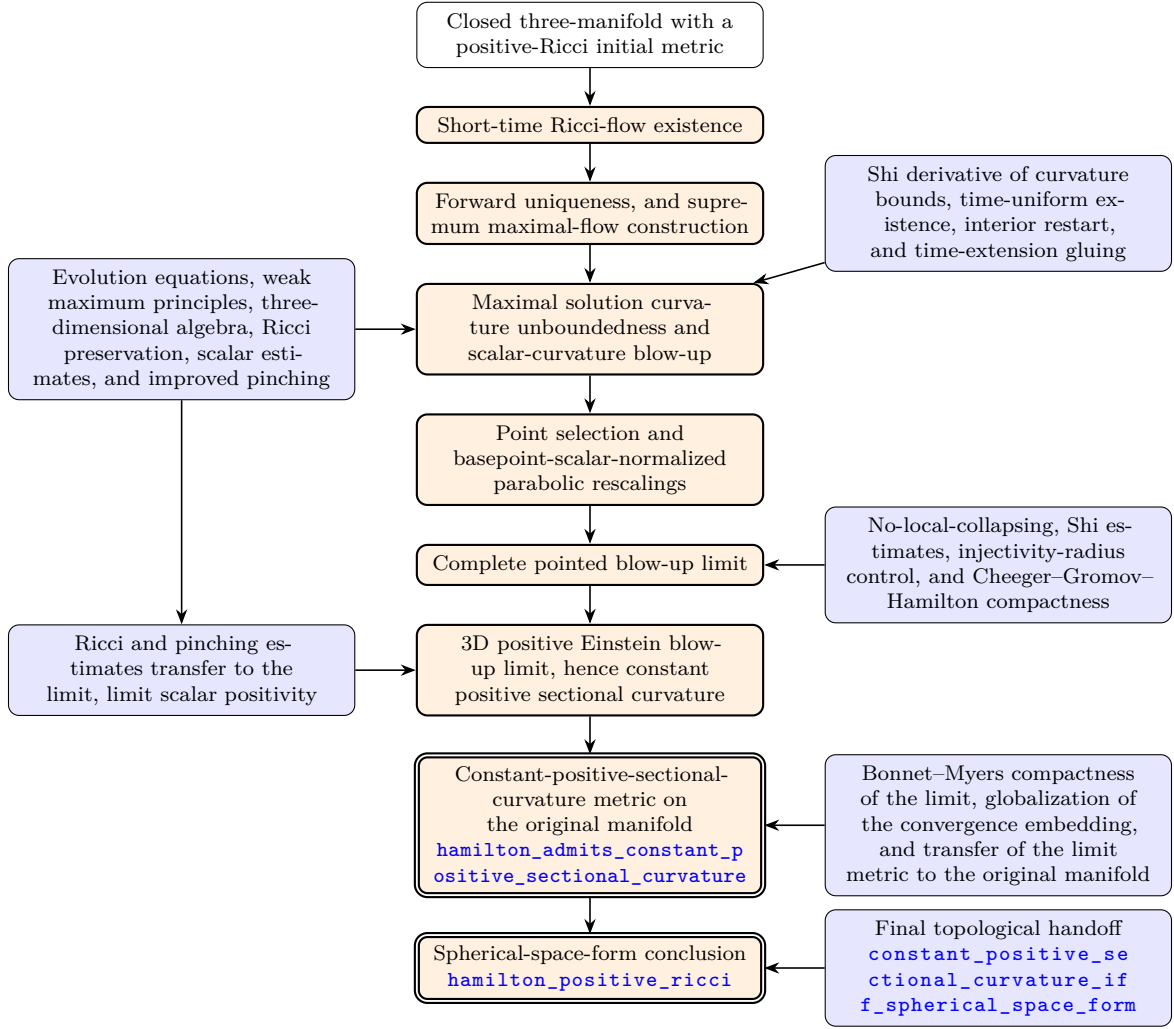

\subsection{Release}
\label{subsec:source-release}

The formal artifact described in this paper is identified by the following
release:
\begin{center}
\begin{tabular}{@{}p{0.20\linewidth}p{0.74\linewidth}@{}}
Repository: & \url{https://github.com/qinz1yang/differential-geometry} \\
Release: & \url{https://github.com/qinz1yang/differential-geometry/releases/tag/arxiv-v1-preview} \\
Version: & \texttt{0.1.0} \\
Release tag: & \href{https://github.com/qinz1yang/differential-geometry/tree/arxiv-v1-preview}{\nolinkurl{arxiv-v1-preview}} \\
Commit: & \href{https://github.com/qinz1yang/differential-geometry/commit/8bd406e35c33a200e9b88895cf11ee8429194e15}{\nolinkurl{8bd406e35c33a200e9b88895cf11ee8429194e15}} \\
Source tree: & \url{https://github.com/qinz1yang/differential-geometry/tree/arxiv-v1-preview/DifferentialGeometry} \\
Lean toolchain: & \href{https://github.com/qinz1yang/differential-geometry/blob/arxiv-v1-preview/lean-toolchain}{\nolinkurl{leanprover/lean4:v4.29.0}} \\
\texttt{mathlib}: & \href{https://github.com/qinz1yang/differential-geometry/blob/arxiv-v1-preview/lakefile.toml}{\nolinkurl{v4.29.0}}
\end{tabular}
\end{center}
The full commit SHA (Secure Hash Algorithm) records the source revision represented by the release.
Every displayed project signature and source path in this manuscript was
checked against that revision.  Model-space assumptions were checked
declaration by declaration; in particular, the normed-model-space signature
of the headline short-time theorem is not used to infer a stronger
project-wide generality claim.

\section{Differential-geometric foundations in Lean}
\label{sec:formal-differential-geometry}

This section describes the project's Lean differential-geometric infrastructure,
taking the Levi-Civita connection and the induced tensor covariant derivatives
as a detailed running example.

Our construction builds on \texttt{mathlib}'s differential-calculus,
smooth-manifold, vector-bundle, and local-flow infrastructure
\cite{mathlib,GouezelHigherOrderCalculus,BordgCavalleriDGLean,
RothgangDGMathlib}.  These citations provide historical and mathematical
context; the \texttt{mathlib} version used by the project is listed in
Subsection~\ref{subsec:source-release}.

The project-local contribution is not a replacement for this broader library
work.  It is a Ricci-flow-oriented layer supplying the tensor model fibers and
tensor-bundle constructions, pointwise tensor fibers, and the particular
connection, tensor-nabla, coordinate-component, curvature, metric, and
regularity interfaces needed by Hamilton's proof and by the final theorem
statement itself.  As in the short-time chapter, we display only signatures
whose types reveal a stable mathematical interface.  Each is followed by an
account of its parameters, its content, and its downstream role; long
implementation-facing declarations are cited without being reproduced.
The displays use notation-normalized ASCII font for a few Lean symbols and
suppress only repetitive namespace parameters; the declaration links point to
the exact definitions in the release.

\subsection{Tensor bundles and theorem-facing predicates}
\label{subsec:theorem-facing-predicates}

We describe the structures over
which the final theorem is stated.
The selected theorem \leanref{hamilton_positive_ricci} is
formulated for a charted smooth manifold modeled on a normed space, smooth
Riemannian metrics, tensor fields on the tangent bundle, and pointwise
curvature tensors constructed from these data.  The manifold, model-space,
and metric layers are built from \texttt{mathlib} structures
\href{https://leanprover-community.github.io/mathlib4_docs/Mathlib/Geometry/Manifold/IsManifold/Basic.html}{Mathlib Manifold},
while the tensor and curvature layers contain central project-local
constructions.  Mathlib's manifold framework accommodates infinite-dimensional
Banach manifolds, whereas the tensor-bundle construction used here is
deliberately finite-dimensional.  Covariant tensors can still be modeled in
infinite dimensions as continuous multilinear forms.  For mixed tensors,
however, the familiar finite-dimensional identifications are no longer
canonical: algebraic tensor products must generally be completed using a
chosen cross norm, and the projective and injective tensor norms can produce
different Banach spaces.  Likewise, the finite-dimensional identification
\[
V^*\otimes W \cong \operatorname{Hom}(V,W)
\]
does not extend to an identification with all bounded linear maps in general.
Since Hamilton's theorem is finite-dimensional, our construction uses the
continuous-linear-map realization of mixed tensors, where these distinctions
cause no ambiguity.

The corresponding pointwise tensor fibers \leanref{Tensor0SSpace} and
\leanref{TensorRSSpace} are defined by
\begin{Verbatim}[breaklines=true,breakanywhere=true,fontsize=\small]
def Tensor0SSpace (s : Nat) (I : ModelWithCorners K E H)
    [IsManifold I 1 M] (x : M) : Type _ :=
  Bundle.continuousMultilinearMap K s E
    (TangentSpace I) x

def TensorRSSpace (r s : Nat) (I : ModelWithCorners K E H)
    [IsManifold I 1 M] (x : M) : Type _ :=
  Tensor0SSpace r I x ->L[K] Tensor0SSpace s I x
\end{Verbatim}
Thus \leanref{Tensor0SSpace} replaces the model space \(E\) by the
tangent space \(T_xM\), while \leanref{TensorRSSpace} realizes the
mixed \((r,s)\)-tensor fiber at \(x\) as a continuous linear map between
the corresponding covariant tensor fibers.

The underlying total space of the mixed tensor bundle uses
\leanref{TensorRSModel}:
\begin{Verbatim}[breaklines=true,breakanywhere=true,fontsize=\small]
TotalSpace (TensorRSModel r s K E)
  (fun x : M => TensorRSSpace r s I x)
\end{Verbatim}
and, analogously, the covariant tensor bundle uses \leanref{Tensor0SModel}:
\begin{Verbatim}[breaklines=true,breakanywhere=true,fontsize=\small]
TotalSpace (Tensor0SModel s K E)
  (fun x : M => Tensor0SSpace s I x)
\end{Verbatim}
An older version of the project had an additional wrapper
\texttt{tensorRSBundle}; it was later removed as the inference API improved.
The current project equips the total spaces directly with the instances
\leanref{tensorRSBundle_topology},
\leanref{tensorRSBundle_fiber},
\leanref{tensorRSBundle_vector}, and
\leanref{tensorRSBundle_smooth} directly.  These provide, respectively, the
topology, local fiber-bundle trivializations, vector-bundle operations,
and smooth compatibility.

The distinction between the model fiber and the bundle is essential.
The fixed finite-dimensional space
\leanref{TensorRSModel}\,r\,s\,K\,E is the analytic model used in
local trivializations, whereas
\leanref{TensorRSSpace}\,r\,s\,I\,x is the geometric fiber that
varies with \(x\).  The bundle structure identifies these varying fibers
locally with the fixed model fiber, thereby making it possible to define
a smooth mixed tensor field as a smooth section over \(M\).

% AUTHOR-CHECK: Ask Yuan to identify the precise tensor-bundle type-checking
% problem referenced in the earlier draft, the relevant declarations, and the
% corresponding Zulip discussion.  Confirm whether the issue concerned
% point-dependent fibers, the mixed-tensor bundle construction, typeclass
% inference, valence-changing operations, or some combination of these before
% describing it as part of the project's development history.

\begin{remark} \textit{Why tensor bundles are formally demanding?}
The principal formal difficulty is that a tensor field, as used in calculus on manifold, has to be treated as a section of tensor
bundle, rather than an ordinary function taking values in one fixed vector
space.  Schematically, if \(\alpha\) is a covariant \(s\)-tensor field, then considering two different points
\[
  \alpha(x) : \texttt{Tensor0SSpace s I x},
  \qquad
  \alpha(y) : \texttt{Tensor0SSpace s I y},
\]
these values inhabit different
Lean types and disturb the inference process, as the fibers are allowed to depend on the base points.  Ordinary differential-geometric notation suppresses this
distinction and silently uses a local trivialization when tensors at nearby
points are compared or differentiated.  In Lean, the required identification
must be made explicit: locally, the tensor field is transported through a
bundle trivialization to the fixed model fiber
\(\texttt{Tensor0SModel s \(\mathbb{R}\) E}\), the calculation is performed
there, and compatibility with the geometric bundle must be established.
Attempting to combine values from different fibers without such a transport
shall produce a type mismatch.
\end{remark}

Our treatment of tensor bundles gives a unified account of tensors with different valences.  The
directional derivative \(\nabla_X\alpha\) has the same valence as \(\alpha\),
whereas the total covariant derivative \(\nabla\alpha\) has one additional
covariant slot.  Contraction, raising or lowering an index, permuting slots,
and passing between an invariant tensor and its coordinate representation
also involve type-level changes.  Moreover, identifications that are
mathematically canonical are often implemented as explicit isomorphisms rather
than being definitionally equal in Lean.  Expressions treated as
interchangeable on paper may therefore require transports, coercions,
reindexing maps, or bridge lemmas before they type-check.  This additional
bookkeeping is substantial, but it also prevents the formal development from
silently identifying different fibers or confusing operations having
genuinely different domains and codomains.

The base manifold data are carried by a model with corners \(I\) on a model
space \(E\) ($\mathbb{R}^n$ for finite-dimensional real manifolds), a
topological space \(M\), a charted-space structure, and manifold instances
encoding the regularity of chart transitions.  The theorem-facing closed
three-manifold hypothesis is the
predicate \leanref{isClosedThreeManifold}.
In the current Lean endpoint this augments the \texttt{mathlib} manifold
hypotheses with compactness, connectedness, boundarylessness, and the dimension
condition
\[
  \operatorname{finrank}_{\mathbb R} E = 3 .
\]

These pointwise fibers carry the topology, vector-space, and normed-space
structures used by the later estimates.  Curvature, Ricci tensors, trace-free Ricci tensors, and tensor
norms are defined on these actual fiber objects.

Tensor fields are then functions or sections over these fibers.  The metric is
a \leanref{SmoothRiemannianMetric}, displayed with its parameters as
\[
  \mathtt{SmoothRiemannianMetric}\ I\ M,
\]
that is, a smoothly varying inner product on the tangent spaces.  The metric
also gives a \((0,2)\)-tensor field through the tensor-bundle layer.  Curvature
objects used by the Hamilton endpoint are built from the Levi-Civita connection
of such a metric.  In particular, the pointwise Ricci tensor is evaluated by
\leanref{metricRicciAt}:
\[
  \mathtt{metricRicciAt}\ g\ x
  :
  \mathtt{Tensor0SSpace}\ 2\ I\ x,
\]
and the standard lowered Riemann tensor evaluator is
\leanref{metricRm04StdAt}:
\[
  \mathtt{metricRm04StdAt}\ g\ x\ X\ Y\ Z\ W
  =
  \Rm_{04}(X,Y,Z,W).
\]
The standard slot convention is
\[
  \Rm_{04}(X,Y,Z,W)=\langle R(X,Y)Z,W\rangle .
\]

With these objects in place, the initial positive-Ricci hypothesis becomes a
precise predicate.  For a smooth metric \(g\), the predicate
\leanref{positiveRicciMetric} is written
\[
  \mathtt{positiveRicciMetric}\ g
\]
and means
\[
  \forall x:M,\ \forall v\in T_xM,\ v\ne0
  \Longrightarrow
  0 <
  \mathtt{metricRicciAt}\ g\ x(v,v).
\]
In Lean the pair \((v,v)\) is supplied by the local tensor argument constructor
\leanref{vec2}.  The theorem hypothesis
\leanref{admitsPositiveRicci} is then the existence of a smooth Riemannian metric satisfying
\leanref{positiveRicciMetric}.  This is why the tensor-bundle and Ricci-tensor
construction is part of the theorem statement itself, not only part of the
evolution proof.

The metric conclusion is similarly theorem-facing.  For a smooth metric \(g\),
the predicate \leanref{constantPositiveSectionalCurvatureMetric} is written
\[
  \mathtt{constantPositiveSectionalCurvatureMetric}\ g
\]
and asserts the existence of a constant \(c>0\) such that, for every point and every
pair of tangent vectors \(X,Y\),
\[
  \mathtt{metricRm04StdAt}\ g\ x\ X\ Y\ Y\ X
  =
  c\bigl(
      g(X,X)g(Y,Y)-g(X,Y)^2
    \bigr).
\]
The predicate \leanref{admitsConstantPositiveSectionalCurvature} is the existence of such a smooth
metric.  The topological conclusion \leanref{isSphericalSpaceForm} is a separate
global-geometry predicate: it is represented by a finite free isometric
quotient model of the round three-sphere and a smooth equivalence with \(M\).

Thus the final statement \leanref{hamilton_positive_ricci} rests on several infrastructure
layers before any Ricci-flow argument begins:
the manifold package \leanref{isClosedThreeManifold}, the tensor model fibers and
bundles \leanref{Tensor0SModel}, \leanref{TensorRSModel},
\leanref{Tensor0SSpace}, and \leanref{TensorRSSpace}, smooth metrics,
Levi-Civita curvature producers, and the theorem-facing predicates
\leanref{positiveRicciMetric}, \leanref{admitsPositiveRicci},
\leanref{constantPositiveSectionalCurvatureMetric}, \leanref{admitsConstantPositiveSectionalCurvature}, and
\leanref{isSphericalSpaceForm}.  The rest of this section explains how the
connection, tensor-nabla, and coordinate-component layers supply the local
geometric calculations consumed by the Ricci-flow proof.

\subsection{Construction of the Levi-Civita connection}

We present the Levi-Civita connection as a detailed example of our general
methodology.  It is a mathematically familiar object, but it exposes one of the
central issues in any Lean Ricci-flow development: in ordinary mathematical
writing one suppresses the point, the metric, the domain, and the regularity
hypotheses whenever they are clear from context, whereas in Lean these
parameters must be supplied in a form from which elaboration can recover them.
The same infrastructure is then reused in the later Ricci-flow evolution
equations, where the project must move reliably between invariant tensor
notation, coordinate-frame components, and local analytic hypotheses.  We begin
with the several distinct operations that the single symbol $\nabla$ denotes.

\paragraph{Lean source.}
The project-local construction of the Levi--Civita connection is contained
principally in
\rffile{Geometry/Connection/LeviCivita/LeviCivitaChartLocal.lean},
\rffile{Geometry/Connection/LeviCivita/Defs.lean}, and
\rffile{Geometry/Connection/LeviCivita/Koszul.lean}.  These relative paths
link to the corresponding files in the release.
The principal declarations are
\leanref{chartLeviCivita}, \leanref{LeviCivita},
\leanref{LeviCivita_torsion_eq_zero},
\leanref{LeviCivita_isMetricCompatible},
\leanref{LeviCivita_unique}, and \leanref{koszul_identity}.

\subsubsection{Linear connections on \texorpdfstring{$TM$}{TM} and their actions on tensors}

We begin by separating several uses of the notation $\nabla$.  In differential
geometry, a linear connection on the tangent bundle is often written as a
bilinear operation on vector fields
\[
  \nabla : \mathfrak X(M) \times \mathfrak X(M) \to \mathfrak X(M),
  \qquad
  (X,Y) \mapsto \nabla_XY .
\]
It is $C^\infty(M)$-linear in the first argument, $\mathbb R$-linear in the
second argument, and satisfies the Leibniz rule
\[
  \nabla_X(fY) = X(f)Y + f\nabla_XY .
\]
The same connection induces covariant derivatives on tensor fields.  For a
covariant tensor field $\alpha \in \Gamma((T^*M)^{\otimes s})$, the directional
covariant derivative is characterized by
\[
  (\nabla_X\alpha)(Y_1,\ldots,Y_s)
  =
  X\bigl(\alpha(Y_1,\ldots,Y_s)\bigr)
  -
  \sum_{a=1}^s
  \alpha(Y_1,\ldots,\nabla_XY_a,\ldots,Y_s).
\]
Equivalently, the total covariant derivative is the $(0,s+1)$-tensor field
\[
  (\nabla\alpha)(X,Y_1,\ldots,Y_s)
  =
  (\nabla_X\alpha)(Y_1,\ldots,Y_s).
\]
Thus the same mathematical symbol $\nabla$ denotes three related
operations: the tangent-bundle connection, the induced directional covariant
derivative on tensors, and the total covariant derivative, which adds one lower
index.  In Lean these operations have different types and therefore have to be
implemented as separate, connected APIs.

\subsubsection{Linear connections in \texttt{mathlib} and the Levi-Civita connection}

The base \href{https://leanprover-community.github.io/mathlib4_docs/Mathlib/Geometry/Manifold/VectorBundle/CovariantDerivative/Basic.html#CovariantDerivative}{interface} is \texttt{mathlib}'s bundled tangent-bundle covariant
derivative
\[
  \texttt{CovariantDerivative I E (TangentSpace I : M -> Type \_)} .
\]
Its application order is
\[
  \texttt{cov Y x v} = (\nabla_vY)(x),
\]
where $Y$ is a vector field, $x\in M$, and $v\in T_xM$.  In other words, the
point $x$ and the tangent direction $v$ are explicit arguments of the API.  This
object is a connection on tangent-vector fields: it contains the additivity and
Leibniz rules required of a covariant derivative, but it is not by itself a
complete tensor-calculus package.  Its $C^\infty$-linearity in the direction
variable is expressed pointwise, through the value at $x$ and the tangent vector
$v$, rather than through a globally chosen vector field $X$.

For a smooth Riemannian metric $g$, the project defines the Levi-Civita
connection \leanref{LeviCivita} from
\leanref{leviCivitaConnectionOfMetric} as a particular instance of this
\texttt{mathlib} structure:
\begin{Verbatim}[breaklines=true,breakanywhere=true,fontsize=\small]
def LeviCivita (g : SmoothRiemannianMetric I M) :
    CovariantDerivative I E (TangentSpace I : M -> Type _) :=
  leviCivitaConnectionOfMetric (I := I) g
\end{Verbatim}
The input is the metric itself; the output is a bundled covariant derivative
on tangent-vector fields.  No torsion-free or compatibility hypothesis is
passed into this definition.  Those properties must be proved for the
chart-stitched construction and are later used to characterize it uniquely.
The classical characterization of this connection is the Koszul formula
\[
\begin{aligned}
2\langle \nabla_XY,Z\rangle
&=
X\langle Y,Z\rangle
+
Y\langle Z,X\rangle
-
Z\langle X,Y\rangle                                      \\
&\quad
-\langle X,[Y,Z]\rangle
+
\langle Y,[Z,X]\rangle
+
\langle Z,[X,Y]\rangle .
\end{aligned}
\]
In the formal development, the canonical Levi--Civita connection is constructed
directly from the Koszul formula.  Given a smooth Riemannian metric \(g\), the
project first defines the scalar expression
\leanref{koszulScalar} by
\[
\begin{aligned}
  K_g(X,Y,Z)
  ={}&
  X\langle Y,Z\rangle_g
  +Y\langle Z,X\rangle_g
  -Z\langle X,Y\rangle_g  \\
  &-\langle X,[Y,Z]\rangle_g
  +\langle Y,[Z,X]\rangle_g
  +\langle Z,[X,Y]\rangle_g .
\end{aligned}
\]
For the Levi--Civita connection this expression satisfies
\[
  2\langle \nabla_XY,Z\rangle_g=K_g(X,Y,Z).
\]
At a point \(x\), the construction turns
\[
  Z_x\longmapsto \frac12 K_g(X,Y,Z)(x)
\]
into a continuous covector
\leanref{koszulCovectorField}.  A noncomputable finite basis of \(T_xM\) is
used only to package this covector as a continuous linear map; the theorem
\leanref{koszulCovectorField_apply_of_mdiff} proves that the packaged map
agrees with the intrinsic Koszul scalar on every differentiable test field.
Applying the metric musical isomorphism produces the vector
\leanref{koszulNablaField}.  The descended value
\leanref{koszulNablaAt} interprets a tangent vector \(v\in T_xM\) by extending
it to the chart-constant field \leanref{tangentConstAt} and then evaluating
the field-level Koszul construction at \(x\).

The pointwise map in the differentiating direction is packaged as the
continuous linear map
\leanref{leviCivitaConnectionCandidateAt}.  Its agreement with the descended
Koszul value is recorded by
\leanref{leviCivitaConnectionCandidateAt_agreesWithDescended}.  Finally,
\leanref{leviCivitaConnectionOfMetric} assembles these pointwise maps into
\texttt{mathlib}'s bundled
\texttt{CovariantDerivative}; its additivity and Leibniz laws are proved
directly from the corresponding identities for the Koszul expression.  Thus
the connection is not introduced by choosing an object satisfying torsion
freeness and metric compatibility: it is an explicit metric-derived
connection whose defining action is the Koszul action.

The public name \leanref{LeviCivita} is definitionally this canonical
connection:
\begin{Verbatim}[breaklines=true,breakanywhere=true,fontsize=\small]
def LeviCivita (g : SmoothRiemannianMetric I M) :
    CovariantDerivative I E (TangentSpace I : M -> Type _) :=
  leviCivitaConnectionOfMetric (I := I) g
\end{Verbatim}
The identity
\leanref{LeviCivita_eq_leviCivitaConnectionOfMetric} is therefore proved by
reflexivity.  The name \leanref{metricCov}, used throughout the curvature and
Ricci-flow layers, is another compatibility wrapper around the same
Koszul-constructed connection.  Smoothness is not built into the pointwise
definition; it is established separately by
\leanref{leviCivitaConnectionOfMetric_contMDiffCovariantDerivativeLocally}.

There is also a coordinate realization of this construction.  In the chart
centered at a point, the metric is represented by its Gram matrix and inverse
Gram matrix.\footnote{If \(e_1,\ldots,e_n\) is the coordinate frame, the Gram
matrix is \(G_{ij}=g(e_i,e_j)\).  Its inverse \(G^{ij}\) represents the inverse
metric in the dual coordinate frame; it is not an additional metric.}
The project defines the chart Christoffel symbols
\leanref{chartChristoffel} by the classical half-inverse-Gram formula displayed
in the coordinate subsection below.  The chart-local candidate
\leanref{chartLeviCivita} differentiates the coordinate representative of a
tangent-vector field using the chart's Fr\'echet derivative and adds the
Christoffel correction term
(\leanref{chartLeviCivitaInnerCLM},
\leanref{christoffelCorrection}).  It is valid on the good chart region
\leanref{chartLeviCivitaGoodSet} around the chart center.

The overlap theorem
\leanref{chartLeviCivita_chart_overlap} proves that chart-local candidates
agree wherever their good regions overlap.  The auxiliary assignment
\leanref{leviCivitaStitched}, which evaluates at \(x\) the candidate associated
with the chart centered at \(x\), remains in the coordinate infrastructure.
It is no longer the definition of the public Levi--Civita connection.
Instead, the theorem
\leanref{LeviCivita_chart_apply} identifies the directly Koszul-constructed
global connection with the chart-local Christoffel candidate on every good
region.  Its proof uses the local uniqueness theorem
\leanref{koszul_local_uniqueness}: both connections are torsion-free and
metric-compatible there, and hence have the same action on differentiable
sections.  Consequently, Christoffel symbols are a derived coordinate
realization of the canonical connection rather than the data from which the
canonical connection is assembled.

The formal development proves the two characterizing properties of the
constructed connection.  Metric compatibility is the identity
\[
  X\langle Y,Z\rangle_g
  =
  \langle \nabla_XY,Z\rangle_g
  +
  \langle Y,\nabla_XZ\rangle_g,
\]
while torsion-freeness is
\[
  T^\nabla(X,Y)
  =
  \nabla_XY-\nabla_YX-[X,Y]
  =
  0.
\]
For the canonical object these are first established as
\leanref{leviCivitaConnectionOfMetric_isMetricCompatible} and
\leanref{leviCivitaConnectionOfMetric_isTorsionFree}.  In the compatibility
API, metric compatibility is expressed by the project predicate
\leanref{IsMetricCompatible}, while torsion-freeness is expressed through
\texttt{mathlib}'s torsion field by the equation
\(\texttt{cov.torsion}=0\).  The theorem-facing wrappers are
\leanref{LeviCivita_torsion_eq_zero} and
\leanref{LeviCivita_isMetricCompatible}:
\begin{Verbatim}[breaklines=true,breakanywhere=true,fontsize=\small]
theorem LeviCivita_torsion_eq_zero (g : SmoothRiemannianMetric I M) :
    (LeviCivita (I := I) g).torsion = 0

theorem LeviCivita_isMetricCompatible (g : SmoothRiemannianMetric I M) :
    IsMetricCompatible (LeviCivita (I := I) g) g
\end{Verbatim}
Both statements concern the same explicitly constructed connection and the
same input metric.  The first is an equality of torsion fields; the second
unfolds to the metric Leibniz identity for differentiable vector fields.
Their downstream role is not merely classificatory: together with uniqueness,
they identify the connection under pullback and thereby support the
naturality of the invariant curvature constructions.

The uniqueness theorem retains exactly the regularity required to compare two
connections on a section at one point:
\begin{Verbatim}[breaklines=true,breakanywhere=true,fontsize=\small]
theorem LeviCivita_unique (g : SmoothRiemannianMetric I M)
    (cov : CovariantDerivative I E (TangentSpace I : M -> Type _))
    (htor : cov.torsion = 0) (hmc : IsMetricCompatible cov g)
    {sigma : forall x : M, TangentSpace I x} {x : M}
    (hsigma : MDiffAt (T% sigma) x) (v : TangentSpace I x) :
    cov.toFun sigma x v = (LeviCivita (I := I) g).toFun sigma x v
\end{Verbatim}
Here \texttt{cov} is an arbitrary competing connection;
\texttt{htor} and \texttt{hmc} are the two characterizing hypotheses;
\texttt{hsigma} supplies differentiability only at the comparison point; and
\(v\) is the differentiating direction.  The conclusion identifies the two
covariant derivatives on that section and direction.  This pointwise form is
the strongest appropriate statement because a bundled
\texttt{CovariantDerivative} is not constrained on nondifferentiable section
inputs, and it is exactly the form later naturality arguments consume.

Finally, \leanref{koszul_identity} proves the Koszul formula for every
torsion-free, metric-compatible competing connection, and the project derives
\leanref{koszul_local_uniqueness} and \leanref{LeviCivita_unique} from that
identity.  The formal order is therefore: the Koszul scalar drives the direct
construction; metric compatibility, torsion-freeness, smoothness, and
uniqueness are proved afterward; and the Christoffel formula supplies the
coordinate representation of the already constructed canonical connection.
The reader might find this unnecessary since we prove the uniqueness via recovering the Koszul formula, but explicit definition makes the later calculation-unfolding much easier.
\subsection{Induced covariant derivatives on tensor fields in Lean}

Once a tangent-bundle connection \texttt{cov} has been constructed, the project
defines its induced action on tensor fields in a separate layer.  For covariant
tensor fields, the main pointwise directional definition is
\leanref{nabla0SFun}, evaluated as
\[
  \texttt{nabla0SFun s cov X alpha x}.
\]
This represents the value of $\nabla_X\alpha$ at $x$, where $\alpha$ is a
$(0,s)$-tensor field.  For mixed tensor fields, the corresponding definition is
\leanref{nablaRSFun}, evaluated as
\[
  \texttt{nablaRSFun r s cov X T x},
\]
which represents the value of $\nabla_XT$ at $x$ for a mixed $(r,s)$-tensor
field $T$.

These definitions are directional covariant derivatives: they preserve tensor
valence.  The total covariant derivative is a separate, typed construction:
for smooth compactly carried tensor fields the project uses \leanref{covGrad},
whose type is
\[
  \texttt{covGrad g r s : SmoothCcTensor g r s -> SmoothCcTensor g r (s+1)},
\]
which adds one covariant slot and instantiates the connection with
\leanref{LeviCivita}; its iterates \leanref{iteratedCovGrad} produce the
higher covariant derivatives consumed by the Sobolev-comparison layer.
Evaluating the added slot at a vector $X_x$ recovers the directional
derivative
\[
  (\nabla\alpha)_x(X_x,\cdots) = (\nabla_X\alpha)_x(\cdots).
\]

Smoothness is also kept separate from the raw pointwise definition.  The raw
functions \leanref{nabla0SFun} and \leanref{nablaRSFun} define pointwise
values.  The bundled tensor-field versions require the explicit smoothness
assumptions \leanref{Nabla0SRegular} and \leanref{NablaRSRegular}, written
\[
  \texttt{Nabla0SRegular}, \qquad
  \texttt{NablaRSRegular},
\]
while the typed total derivative \leanref{covGrad} instead carries smoothness
in its domain and codomain types.
This design avoids hiding analytic regularity work inside the definition of the
operator.  For the main Ricci-flow applications in this project, the relevant
regularity hypotheses are produced from smoothness of the metric and the
Levi-Civita connection, ultimately through the chart-local Christoffel
construction and its coordinate consequences.

For mixed tensors, the induced connection has the usual sign convention: upper
indices contribute with a plus sign and lower indices contribute with a minus
sign.  For a mixed tensor $T^a{}_{bc}$, the coordinate expression is
\[
  (\nabla_i T)^a{}_{bc}
  =
  \partial_i T^a{}_{bc}
  +
  \Gamma^a_{im} T^m{}_{bc}
  -
  \Gamma^m_{ib} T^a{}_{mc}
  -
  \Gamma^m_{ic} T^a{}_{bm}.
\]
The project proves
both the covariant coordinate projection theorem discussed below and the mixed
slotwise expansion \leanref{nablaRS_coordFrame_slots_of_smooth} in
\rffile{Geometry/Coordinates/NablaComponents/TensorRS/Formula.lean}.

\subsection{Coordinate Christoffel symbols}
The development first defines tensors, the Levi--Civita connection, and
covariant derivatives invariantly.  Components are evaluations of those same
objects, not parallel coordinate definitions: a covariant tensor uses
\leanref{component0S}, a mixed tensor uses \leanref{componentRS}, and a local
frame is evaluated through
\href{https://leanprover-community.github.io/mathlib4_docs/Mathlib/Geometry/Manifold/VectorBundle/LocalFrame.html#IsLocalFrameOn.toBasisAt}{\nolinkurl{IsLocalFrameOn.toBasisAt}}
together with the domain
witness.  Coordinate-chart wrappers \leanref{coordComponent0SAt} and
\leanref{coordComponentRSAt} specialize this pointwise basis construction.
Christoffel symbols are therefore the coordinate-frame components of the
already constructed connection, and coordinate identities are realization
theorems for invariant objects.

Only after the invariant connection and tensor-nabla layers are in place does
the development pass to local coordinates.  If $\{\partial_i\}$ is a
coordinate frame, the Christoffel symbols of a connection are defined by
\[
  \nabla_{\partial_i}\partial_j
  =
  \Gamma^k_{ij}\partial_k.
\]
For the Levi-Civita connection, the chart-level Christoffel symbols are the
classical expression
\[
  \Gamma^k_{ij}
  =
  \frac12
  \sum_\ell g^{k\ell}
  \left(
    \partial_i g_{j\ell}
    +
    \partial_j g_{i\ell}
    -
    \partial_\ell g_{ij}
  \right)
\]
by definition: this is the formula \leanref{chartChristoffel} that drives the
chart-local construction described in the preceding subsection.  The proved
content of the coordinate layer is the identification of invariant connection
coefficients with these chart quantities.  The frame-level Christoffel symbol
of an arbitrary connection is \leanref{christoffelAlongInFrame}, the $k$-th
frame coefficient of $\nabla_X\partial_j$; the theorem
\leanref{connCoeff_eq_christoffelAlong_coord} identifies the entries of the
chart connection endomorphism with these coefficients in the coordinate
frame, and \leanref{LeviCivita_chart_apply} identifies the values of
\leanref{LeviCivita} with the chart-local Christoffel-corrected derivative.
\subsection{A concrete checked example: the covariant-tensor coordinate formula}

As a representative worked example we take the general coordinate identity for
the covariant derivative of a covariant tensor of arbitrary valence,
formalized as \leanref{nabla0S_coordFrame_slots_of_smooth} in
\rffile{Geometry/Connection/Chart/NablaComponents/Basic.lean}.

\begin{Verbatim}[breaklines=true,breakanywhere=true,fontsize=\small]
theorem nabla0S_coordFrame_slots_of_smooth {s : Nat}
    (cov : CovariantDerivative I E (TangentSpace I : M -> Type _))
    (X : ContMDiffSection I E infinity (TangentSpace I : M -> Type _))
    (alpha : Tensor0SField (I := I) (M := M) (n := infinity) s)
    (x0 : M) (slots : Fin s -> CoordinateIdx E) :
    coordComponent0SAt (I := I) (nabla0SFun s cov X alpha x0) slots =
      coordDeriv0SAt (I := I) (fun x => X x) x0 (fun x => alpha x) slots -
        sum a : Fin s, sum k : CoordinateIdx E,
          christoffelAlongInFrame cov (coordinateFrameAt (I := I) x0)
            (coordinateFrameAt_isLocalFrame_one (I := I) x0)
            x0 (X x0) (slots a) k *
          coordComponent0SAt (I := I) (alpha x0)
            (Function.update slots a k)
\end{Verbatim}

The signature shows the entire mathematical scope.  It is valid for an
arbitrary connection \texttt{cov}, a smooth vector field $X$, and a smooth
covariant $s$-tensor $\alpha$ at a chosen chart center $x_0$.  The index map
\texttt{slots} selects one component.  The conclusion is the derivative of
that component minus one Christoffel correction for each covariant slot.  It
does not assume the Levi--Civita connection and it does not state a mixed-tensor
formula; those are, respectively, a later specialization and a broader
interface not claimed by this declaration.

Fix a coordinate chart on an $n$-dimensional manifold and let $\alpha$ be a
smooth covariant tensor field of type $(0,s)$.  A choice of coordinate indices
\[
  J : \operatorname{Fin}(s) \to \{1,\ldots,n\}
\]
selects a scalar component $\alpha_J$.  We prove the classical component
formula
\[
  (\nabla_X \alpha)_J
  =
  X\bigl(\alpha_J\bigr)
  -
  \sum_{a=1}^{s}\sum_{k}
  \Gamma^k(X,J_a)\,\alpha_{J[a\mapsto k]},
\]
where $J[a\mapsto k]$ replaces the $a$-th index of $J$ by $k$.  The sign is
the usual one: each covariant slot contributes a negative Christoffel term.
For a mixed $(r,s)$-tensor the contravariant slots contribute the corresponding
positive terms; the project proves that full formula separately in
\leanref{nablaRS_coordFrame_slots_of_smooth}.

Although this identity is classical, the formalization does not adopt it as the
definition of the covariant derivative: $\nabla$ is defined invariantly, and the
formula is \emph{proved} by projecting that invariant operator into a chart.
Fix the chart and identify its domain with an open set $U\subseteq\mathbb R^n$.
Over $U$ the tensor field $\alpha$ is the same data as a single smooth map
\[
  \check \alpha : U \longrightarrow \mathcal T^{(0,s)},
  \qquad
  \check \alpha(x)=\bigl(\alpha_J(x)\bigr)_{J},
\]
valued in the fixed finite-dimensional model fiber $\mathcal T^{(0,s)}$ of
covariant $s$-tensors on $\mathbb R^n$ (the space \leanref{Tensor0SModel}).  If
$e_1,\ldots,e_n$ is the coordinate basis, then
\[
  \alpha_J(x)
  =\check \alpha(x)
    \bigl(e_{J_1},\ldots,e_{J_s}\bigr);
\]
at the chart center this reading is the bridge lemma
\leanref{tensor0SModelAt_coordComponent0SAt}.  The connection, read in the
same chart, becomes matrix-valued: contracting the connection against a
direction $X$ gives an endomorphism $\Gamma(X)$ of $\mathbb R^n$, with entries
$\Gamma^p(X,q)$, the \emph{connection endomorphism in the chart}
(\leanref{connectionEndomorphismInChart}).  That the invariant
$\nabla_X \alpha$, restricted to the chart, \emph{is} the model-space
covariant derivative of $\check \alpha$ is the definition unfolded by
\leanref{nabla0SFun}.

On the model side the covariant derivative is transparent, because
$\mathcal T^{(0,s)}$ is a fixed vector space.  It is the sum of the ordinary
(Fr\'echet) derivative of the map $\check \alpha$ and the induced action of
$\Gamma(X)$ on the fiber,
\[
  \nabla_X\check \alpha = D\check \alpha(X) + \Gamma(X)\cdot\check \alpha ,
\]
and, since the fiber is assembled from $s$ covariant copies of
$(\mathbb R^n)^*$, the endomorphism $\Gamma(X)$ acts on it as a
derivation---one term per slot---by its dual action (the transpose, carrying
a negative sign) on each covariant slot.
Taking the $J$ coefficient of $\nabla_X\check \alpha$ therefore produces
exactly two kinds of term: the $J$ coefficient of $D\check \alpha(X)$, and,
for each tensor factor $a$, the dual-action term
$-\sum_k\Gamma^k(X,J_a)\,\alpha_{J[a\mapsto k]}$.
This slotwise product rule is the algebraic heart of the proof
(\leanref{covariantDeriv_tensor0SModelWithin_apply_basis_slots}), and it already
has the shape of the displayed formula.

It remains to read the two kinds of term as the coordinate quantities in the
statement.  The entries $\Gamma^p(X,q)$ of the connection endomorphism are the
frame-level Christoffel symbols contracted with $X$
(\leanref{connCoeff_eq_christoffelAlong_coord}, through
\leanref{christoffelAlongInFrame}).  Each basis coefficient
of $\check \alpha$ is, by definition, the coordinate component $\alpha_L$
(\leanref{tensor0SModelAt_coordComponent0SAt}), and replacing the relevant
tensor factor is exactly the index substitution
$J[a\mapsto k]$---in Lean this replacement is the combinator
\texttt{Function.update} acting on the index tuple.  The only
step that uses genuine analysis is the derivative term:
$D\check \alpha(X)$ is the Fr\'echet derivative of a vector-valued map on
$U\subseteq\mathbb R^n$, whereas $X(\alpha_J)$ is a directional derivative of a
scalar function on the manifold, and the bridge
\leanref{modelDeriv_eq_coordDeriv0SAt} proves that the two agree componentwise;
every remaining step is pointwise linear algebra on the fixed fiber
$\mathcal T^{(0,s)}$.  With both kinds of term matched, the two sides are
finite sums over the same index set $J$ and ranges of $k$, and the
identity follows by term-by-term comparison.  The smooth-input theorem
\leanref{nabla0S_coordFrame_slots_of_smooth} packages this argument under a
smoothness hypothesis on $\alpha$; its explicit form
\leanref{nabla0S_coordFrame_slots} carries out the same computation with the
derivative identification supplied as a separate hypothesis.

\section{Spectral PDE foundations and Ricci-flow short-time existence}
\label{sec:short-time-existence}
\label{subsec:ShortTimeExistence}% legacy label retained for old links

The gauge-reduction part of the construction follows DeTurck's method
\cite{MR697987,DeTurck2003}.  The fixed-background spectral-resolvent,
maximal-regularity, and forcing-space fixed-point implementation is
project-specific rather than a line-by-line transcription of one source.
Standard analytic context for closely related resolvent and quasilinear
parabolic methods is provided by Amann's abstract linear theory,
Lunardi's analytic-semigroup treatment, and the maximal-regularity framework
of Pr\"uss and Simonett
\cite{Amann1995,Lunardi1995,PrussSimonett2016}.
These references motivate the
analytic technology; they are not cited as proofs of the project-specific Lean
endpoint.
The short-time theorem is the most complete analytic result in the present
development.  It is also where the formalization makes its most substantial
choice of analytic technology.  Rather than treating existence as a black-box
consequence of a generic parabolic theorem, the proof constructs a compact
self-adjoint tensor resolvent, the associated spectral Sobolev scale, a
maximal-regularity solution operator, a nonlinear fixed point for the
Ricci--DeTurck equation, all-order space--time regularity, and the
time-dependent diffeomorphisms that remove the gauge.  The endpoint is a
theorem about Ricci flow from arbitrary smooth initial data on a closed
positive-dimensional manifold; it assumes neither dimension three nor any
curvature sign.

This section deliberately interleaves selected Lean signatures with their
mathematical interpretation.  A signature is displayed only when it exposes a
stable mathematical interface.  Several indispensable internal declarations
have much larger implementation-facing types; those are described in prose
rather than presented as if their current signatures were suitable public
interfaces.

The headline below retains the complete source signature, with mathematical
Unicode typeset in LaTeX.  In the shorter internal displays we suppress only
repetitive namespace qualifiers and use readable ASCII spellings such as
\texttt{Nat}, \texttt{Real}, \texttt{forall}, and \texttt{exists}.  Those
displays preserve the binders and conclusion relevant to the explanation but
are not intended as paste-ready source excerpts; the declaration links point
to the exact definitions in the release.

\subsection{The exact public endpoint}
\label{subsec:short-time-public-endpoint}

Mathematically, the result is the following classical theorem.

\begin{theorem}[Short-time existence for Ricci flow]
\label{thm:rf-short-time-existence}
Let $M$ be a closed smooth manifold of positive dimension and let $g_0$ be a
smooth Riemannian metric on $M$.  There are $T>0$ and a Ricci flow $g(t)$ on
$[0,T)$ such that $g(0)=g_0$.  The metric is jointly smooth in space and time,
including at the initial time in the one-sided sense.
\end{theorem}

The formal endpoint is
\leanref{ricci_flow_short_time_existence}, in
\rffile{Geometry/Flow/RicciFlow/ShortTime/Existence.lean}.  With only
typographical substitutions for Lean's mathematical Unicode, its signature is
as follows.

\begin{Verbatim}[commandchars=\\\{\},breaklines=true,breakanywhere=true,fontsize=\small]
theorem ricci_flow_short_time_existence
    {E : Type*} [NormedAddCommGroup E] [NormedSpace \(\mathbb R\) E]
      [FiniteDimensional \(\mathbb R\) E] [NeZero (Module.finrank \(\mathbb R\) E)]
    {H : Type*} [TopologicalSpace H] {I : ModelWithCorners \(\mathbb R\) E H}
    {M : Type*} [TopologicalSpace M] [ChartedSpace H M]
      [IsManifold I \(\infty\) M] [CompactSpace M]
      [I.Boundaryless] [T2Space M]
    (g\(_0\) : SmoothRiemannianMetric I M) :
    \(\exists\) T : \(\mathbb R\), 0 < T \(\wedge\)
      \(\exists\) g_fam : \(\mathbb R\) \(\to\) SmoothRiemannianMetric I M,
      g_fam 0 = g\(_0\) \(\wedge\)
      (\(\forall\) (x\(_0\) : M) (i j : Fin (Module.finrank \(\mathbb R\) E)),
        ContMDiffOn (\(\mathcal I\)(\(\mathbb R\), \(\mathbb R\)).prod I) \(\mathcal I\)(\(\mathbb R\)) \(\infty\)
          (fun p : \(\mathbb R\) \(\times\) M \(\Rightarrow\)
            Integral.Measure.chartGramMatrix
              (I := I) (g_fam p.1) x\(_0\) p.2 i j)
          (Set.Ico (0 : \(\mathbb R\)) T \(\times^{\mathrm s}\)
            (trivializationAt E (TangentSpace I) x\(_0\)).baseSet)) \(\wedge\)
      (\(\forall\) t \(\in\) Set.Ico (0 : \(\mathbb R\)) T, \(\forall\) x : M,
        \(\forall\) v w : TangentSpace I x,
        HasDerivWithinAt (fun s : \(\mathbb R\) \(\Rightarrow\)
          (g_fam s).inner x v w)
          ((-2 : \(\mathbb R\)) *
            DifferentialGeometry.Geometry.Curvature.ricciTensor
              (I := I) (g_fam t) x v w) (Set.Ici 0) t)
\end{Verbatim}

The hypotheses deserve to be read literally.  The model space $E$ is only a
finite-dimensional real normed space; the public theorem does not assume a
chosen inner product on $E$.  The instance
\texttt{NeZero (Module.finrank $\mathbb R$ E)} excludes dimension zero.  The
metric $g_0$, rather than the model space, supplies the fiberwise inner
products used by the tensor $L^2$ theory and the connection Laplacian.
Compactness and boundarylessness express the closed-manifold hypothesis, while
\texttt{T2Space} meets a library interface used by the construction.
Connectedness is absent, as it should be.

The typeclass \texttt{I.Boundaryless} says that the model with corners has full
range and supplies the boundaryless-manifold interface used by the proof.  In
particular, the theorem is a theorem for manifolds without boundary or
corners.

The conclusion has three mathematical clauses.  It gives $T>0$ and a family
of smooth Riemannian metrics with $g(0)=g_0$.  It then states joint smoothness
of every chart Gram-matrix entry on $[0,T)$ times the base set of the chosen
tangent-bundle trivialization; these base sets cover $M$.  Finally, for every
$t\in[0,T)$, $x\in M$, and $v,w\in T_xM$, it states
\[
  \frac{d}{dt}g(t)_x(v,w)=-2\Ric_{g(t)}(v,w).
\]
At $t=0$ the derivative is taken relative to \texttt{Set.Ici 0}, so it is the
right derivative; at an interior time it agrees with the ordinary derivative.
The Ricci tensor is constructed from the Levi--Civita connection of $g(t)$,
not supplied as separate data.  Although \texttt{g\_fam} has the total type
$\mathbb R\to\mathtt{SmoothRiemannianMetric}$, the theorem asserts no
regularity, initial-value, or evolution property outside $[0,T)$.

This is an existence theorem, not yet a general local theory.  It does not
assert uniqueness, continuous dependence, a quantitative lifetime, or maximal
continuation.  Of these, only the class-uniform quantitative lifetime and the
resulting maximal-continuation construction are required by the selected
Hamilton route: continuation restarts the equation from metrics that vary in a
controlled class, so the lifetime must be chosen before the restart metric.
General continuous dependence is not a separate dependency of the endpoint,
and none of these distinctions weakens the one-metric short-time conclusion.

\subsection{The class-uniform existence strengthening}
\label{subsec:uniform-existence}

The theorem above is metricwise: its lifetime is selected after the initial
metric.  The compactness and continuation development requires a stronger
quantifier order.  After fixing a background metric and the bounds defining a
controlled class of metrics, one must choose a common lifetime before the
initial metric varies.  Schematically, the required statement is
\[
  \exists \tau_0>0\ \ \forall g_0\in\mathcal G_3(\bar g,\Lambda),
  \qquad
  \text{the Ricci flow from }g_0\text{ exists on }[0,\tau_0).
\]
Here \(\mathcal G_3(\bar g,\Lambda)\) controls uniform metric equivalence and
the background-covariant derivatives through order three.

\begin{theorem}[Uniform short-time existence on a controlled metric class]
\label{bb:uniform-existence}
Let \(M^3\) be closed, let \(\bar g\) be a smooth background metric, and let
\(\Lambda\ge1\).  There is a number \(\tau_0>0\), depending only on the
background and the class bounds, such that every smooth metric \(g_0\) that is
\(\Lambda\)-uniformly equivalent to \(\bar g\) and whose prescribed
background-covariant derivatives through order three are bounded by
\(\Lambda\) admits a Ricci flow on \([0,\tau_0)\), jointly smooth up to the
initial time in the one-sided sense.
\end{theorem}

This change of quantifiers is substantive and does not follow formally from
Theorem~\ref{thm:rf-short-time-existence}.  It is proved by
\leanref{ricci_flow_uniform_existence} in
\rffile{Geometry/Flow/RicciFlow/Extension/Construction.lean}.  The
fixed-background producer \leanref{exists_uniform_jointly_smooth_ricciDeTurck_metric_solution} selects the common lifetime
and constructs the class-uniform Ricci--DeTurck solutions; the gauge-removal
theorem \leanref{ricci_gauge_of_dt} converts each of them to Ricci flow while
preserving the common interval.  A fresh transitive axiom audit of
\leanref{ricci_flow_uniform_existence} reports only \texttt{propext},
\texttt{Classical.choice}, and \texttt{Quot.sound}.

\subsection{Why the DeTurck--spectral route?}
\label{subsec:short-time-route-choice}

Hamilton's original proof dealt directly with the weakly parabolic Ricci-flow
operator by applying the Nash--Moser inverse function theorem
\cite{MR664497}.  Reproducing that route would require a reusable
formal theory of tame Fr\'echet spaces, smoothing operators, tame maps, and a
Nash--Moser inverse theorem, followed by verification that the Ricci operator
fits that framework.  Such a foundation would be valuable, especially for
geometric equations where gauge fixing is unavailable, but it would be a major
independent formalization project.

The present development instead follows DeTurck's gauge reduction
\cite{MR697987,DeTurck2003}.  With a background metric $\bar g=g_0$, put
\[
  W(h)^k=h^{pq}\bigl(\Gamma(h)^k_{pq}-\Gamma(\bar g)^k_{pq}\bigr)
\]
and solve
\begin{equation}
  \partial_t h=-2\Ric(h)+\Lie_{W(h)}h,
  \qquad h(0)=g_0.
  \label{eq:ricci-deturck-short-time}
\end{equation}
The connection difference is tensorial, hence $W(h)$ is a global vector
field.  The gauge term cancels the non-parabolic part of the Ricci
linearization, leaving the principal symbol
$-\lvert\xi\rvert_{g_0}^2\operatorname{id}$ at the initial metric.

There remains a genuine choice of parabolic technology.  A classical
H\"older-space implementation would lead naturally to Schauder estimates.
The project chose a spectral Hilbert-space route because compact
self-adjoint-operator theory, Bochner integration, and Hilbert-space fixed
points offered a firm library substrate, while the required manifold-valued
Schauder theory was not yet available in the project.  The connection
Laplacian also diagonalizes the linear evolution, turning its main identities
into scalar convolution equations.  This choice is therefore both
mathematical and formal: it gives a direct path from compact resolvent to
maximal regularity and makes the modewise estimates auditable.

The choice has costs.  Much of the current PDE layer is specialized to tensor
fields over a compact Riemannian manifold, and its highest-level capstone is
Ricci-shaped rather than a polished theorem for arbitrary quasilinear systems.
The project will address that limitation by separating and generalizing the
reusable PDE interfaces and by developing Schauder estimates and related
H\"older-space machinery.  Those are future extensions, not prerequisites
silently assumed by the present theorem.  For this stage report, the honest
claim is a specialized but closed Ricci-flow endpoint supported by a
substantial spectral parabolic foundation.

\subsection{DeTurck's gauge reduction to a strictly parabolic equation}

Relative to the fixed background connection $\bar{\nabla}$, the Ricci--DeTurck flow system \eqref{eq:ricci-deturck-short-time} has
the local form
(see e.g.~\cite[(2.51)]{MR2274812})
\[
\partial_t h_{ij}
=
h^{pq}\bar{\nabla}_p\bar{\nabla}_q h_{ij}
+
F_{ij}(x,h,\bar{\nabla}h),
\]
where $F_{ij}$ is smooth in its arguments and contains no second derivatives of
$h$. Hence \eqref{eq:ricci-deturck-short-time} is a quasilinear strictly parabolic
system. Since $M$ is closed, standard parabolic theory yields $\varepsilon>0$ and a
unique smooth solution $h(t)$ on $[0,\varepsilon)$.
We formalize this standard parabolic theory in Lean.\smallskip

We now recall why the Ricci--DeTurck flow is geometrically equivalent to
the Ricci flow.
Now let $\varphi_t:M\to M$ be the family of diffeomorphisms solving
\begin{equation}\label{eq: dphi dt -W phi}
\partial_t\varphi_t=-W(h(t))\circ\varphi_t,
\qquad
\varphi_0=\operatorname{id}_M.
\end{equation}
Because $M$ is compact, this ODE has a smooth solution for all
$t\in[0,\varepsilon)$, and each $\varphi_t$ is a diffeomorphism. Define
\[
g(t) :=  \varphi_t^*h(t).
\]
Using the general identity
\[
\frac{d}{dt}\,\varphi_t^*T
=
\varphi_t^*\!\left(\partial_t T+\mathcal L_{X_t}T\right)
\quad\text{whenever}\quad
\partial_t\varphi_t=X_t\circ\varphi_t,
\]
with $X_t=-W(h(t))$, we obtain
\begin{align*}
\partial_t g(t)
&=
\varphi_t^*\!\left(\partial_t h+\mathcal L_{-W(h(t))}h\right) \\
&=
\varphi_t^*\!\left(-2\operatorname{Ric}(h)+\mathcal L_{W(h)}h-\mathcal L_{W(h)}h\right) \\
&=
-2\operatorname{Ric}(\varphi_t^*h)
=
-2\operatorname{Ric}(g(t)).
\end{align*}
Also $g(0)=g_0$. Therefore $g(t)$ is a Ricci flow with initial metric $g_0$,
which proves short-time existence.
Thus, starting from a solution of the Ricci--DeTurck flow, one obtains a solution of the Ricci flow by pulling back along the diffeomorphisms generated by \eqref{eq: dphi dt -W phi}.

Note that \eqref{eq: dphi dt -W phi} is equivalent to the harmonic map heat flow equation
\[
\partial_t \varphi_t = \Delta_{g(t),g_0}\varphi_t,
\qquad\text{where } g(t):=\varphi_t^*h(t)
\]
and where $\Delta_{g(t),g_0}$ is the map-Laplacian from $(M,g(t))$ to $(M,g_0)$
(see \cite[\S 4.3 of Chapter 3]{MR2061425} or \cite[p.~117]{MR2274812}).\smallskip

For uniqueness, let $g_1(t)$ and $g_2(t)$ be two Ricci flows on a common time
interval with $g_1(0)=g_2(0)=g_0$. For $i=1,2$, let
\[
\psi_i(t): (M,g_i(t))\to (M,g_0)
\]
solve the harmonic map heat flow
\[
\partial_t\psi_i=\Delta_{g_i(t),g_0}\psi_i,
\qquad
\psi_i(0)=\operatorname{id}_M.
\]
On a closed manifold,
by standard parabolic theory
these solutions exist for short time and remain
diffeomorphisms after shrinking the time interval if necessary. Define
\[
h_i(t) :=  (\psi_i(t)^{-1})^*g_i(t).
\]
By the converse direction of DeTurck's trick, each $h_i$ solves the same
Ricci--DeTurck flow \eqref{eq:ricci-deturck-short-time} with the same initial
metric $g_0$. By the standard uniqueness theorem for strictly parabolic systems, $h_1=h_2$.
Moreover, the harmonic map heat flow identity gives
\[
\partial_t\psi_i=-W(h_i(t))\circ\psi_i.
\]
Hence $\psi_1$ and $\psi_2$ satisfy the same ODE with the same initial
condition, so $\psi_1=\psi_2$, and therefore $g_1= \psi_1^* h_1 = \psi_2^* h_2 =  g_2$. This proves uniqueness.

Conversely, starting from a solution of the Ricci flow, one recovers a solution of the Ricci--DeTurck flow by solving the harmonic map heat flow with target $(M,g_0)$ and pushing the metric forward by the resulting diffeomorphisms.
The project does formalize the forward uniqueness statement used below, but the
checked implementation is not the harmonic-map heat-flow proof just sketched.
The theorem \leanref{ricci_flow_forward_unique} proves equality on a common
closed-open interval from joint chart-Gram smoothness and continuity, the
Ricci-flow equation, and equality at the initial time.  Its proof uses the
project's curvature--connection-difference energy route.  It is consumed both
by \leanref{flow_to_agree}, in the supremum construction, and by the
interior-restart extension argument.

\subsection{Compact resolvent and the tensor Sobolev scale}
\label{subsec:short-time-spectral-foundation}

Fix a smooth background metric $g$ and tensor type $(r,s)$.  The project first
constructs the variational resolvent of the connection Laplacian and then
composes with the intrinsic inclusion from the $H^1$ completion into tensor
$L^2$.  Rellich compactness makes this inclusion compact, while the
variational construction gives self-adjointness.  The public interface is
short enough to display in full:

\begin{Verbatim}[breaklines=true,breakanywhere=true,fontsize=\small]
theorem tensorResolventL2_isCompactOperator_isSelfAdjoint_intrinsic
    (g : SmoothRiemannianMetric I M) (r s : Nat) :
    IsCompactOperator (tensorResolventL2 (I := I) (M := M) g r s) /\
      IsSelfAdjoint (tensorResolventL2 (I := I) (M := M) g r s)
\end{Verbatim}

The two conclusions are exactly the hypotheses needed
to apply the compact self-adjoint spectral theorem.  They do not themselves
say that the resulting $L^2$ eigenvectors are smooth.  Smooth eigentensor
representatives are supplied later by a separate elliptic-regularity bridge;
keeping that step distinct avoids attributing regularity to the abstract
spectral theorem.

If $\lambda_i\ge0$ denotes the eigenvalue of $-\Delta_g$, the spectral weight
is $(1+\lambda_i)^\sigma$.  The resulting Sobolev object
\leanref{tensorHs} is deliberately transparent:

\begin{Verbatim}[breaklines=true,breakanywhere=true,fontsize=\small]
structure tensorHs (g : SmoothRiemannianMetric I M)
    (r s : Nat) (sigma : Real) where
  coeff : TensorEigenIdx (I := I) (M := M) g r s -> Real
  weighted_summable :
    Summable (fun i => tensorSobolevWeight (I := I) (M := M) i sigma *
      (coeff i) ^ 2)
\end{Verbatim}

A value of \texttt{tensorHs g r s $\sigma$} is thus not an unexplained
function-space token: it is a coefficient family together with the proof that
\(\sum_i(1+\lambda_i)^\sigma|u_i|^2<\infty\).  The exponent convention places
the entire Sobolev exponent on the squared coefficient.  Comparison theorems
connect these spectral norms to covariant-derivative Sobolev norms at the
orders used downstream, and Sobolev embedding then realizes sufficiently
regular coefficient families as tensor fields.  Polynomial eigenvalue-counting
bounds and smooth eigenrepresentatives are used in the all-order convergence
arguments; neither is hidden in the two displayed definitions.

\subsection{Modewise evolution and maximal regularity}
\label{subsec:short-time-maximal-regularity}

On the eigenbasis, the zero-initial-value heat equation
\[
  \partial_tu=\Delta_g u+f,
  \qquad u(0)=0,
\]
reduces to
\[
  \dot u_i=-\lambda_i u_i+f_i,
  \qquad
  u_i(t)=\int_0^t e^{-\lambda_i(t-s)}f_i(s)\,ds.
\]
The integral formula defines a \emph{mild solution}: it makes sense in the
ambient Hilbert spaces before a pointwise time derivative is known.  The
maximal-regularity argument proves that the resulting function has the stated
time and spatial derivatives, so the mild solution is a strong solution in
the maximal-regularity space.

The first formal identity exposes the scalar equation for every spectral
mode:

\begin{Verbatim}[breaklines=true,breakanywhere=true,fontsize=\small]
theorem maximalRegularityOp_solves_perMode {a : Real} (hT : 0 <= T)
    (h_compact : IsCompactOperator
      (tensorResolventL2 (I := I) (M := M) g r s))
    (f : timeL2 (tensorHs (I := I) (M := M) g r s a) T)
    (i : TensorEigenIdx (I := I) (M := M) g r s) :
    timeModeCoeff (maximalRegularityDerivField a hT f) i =
      (-TensorEigenIdx.lambda i) •
        timeModeCoeff (maximalRegularitySolField a hT f) i +
      timeModeCoeff f i
\end{Verbatim}

The parameters identify the Sobolev order $a$, the time interval, the compact
resolvent used to obtain the eigenbasis, the forcing $f$, and one eigenmode
$i$.  The conclusion is precisely the scalar ODE above.  Coefficient
injectivity then reassembles the modes into the Hilbert-space equation:

\begin{Verbatim}[breaklines=true,breakanywhere=true,fontsize=\small]
theorem maximalRegularityOp_solves {a : Real}
    (h_compact : IsCompactOperator
      (tensorResolventL2 (I := I) (M := M) g r s))
    (hT : 0 < T) (hT1 : T <= 1)
    (f : timeL2 (tensorHs (I := I) (M := M) g r s a) T) :
    TimeSobolev.timeH1.timeDeriv _ T (maximalRegularityOp a hT hT1 f) =
      timeScaleLaplacian a (maximalRegularitySolField a hT.le f) + f
\end{Verbatim}

This second theorem is the operator equation, not merely equality of formal
series.  The associated construction places the time derivative in
$L^2_tH^a$ and the solution in $L^2_tH^{a+2}$, with an $H^1_tH^a$ control and
zero initial trace.  Its trace estimate yields the intermediate control
\[
  H^1((0,T);H^a)\cap L^2((0,T);H^{a+2})
  \longrightarrow C^0([0,T];H^{a+1}),
\]
which is the norm needed for the quasilinear fixed point.

\subsection{The specialized nonlinear construction}
\label{subsec:short-time-nonlinear}

Writing $g=g_0+S$ and freezing the connection Laplacian of $g_0$, the
Ricci--DeTurck equation takes the form
\[
  \partial_tS=\Delta_{g_0}S+N(S),\qquad S(0)=0.
\]
The dependence of the principal coefficients on $S$ has not disappeared; it
is contained in $N$.  At the fixed exponent
$a=4\dim M+10$, the implementation proves a mixed two-scale estimate of the
schematic form
\[
 \|N(u)-N(v)\|_{H^a}
 \le C_1\max(\|u\|_{H^{a+1}},\|v\|_{H^{a+1}})
          \|u-v\|_{H^{a+2}}
     +C_2\|u-v\|_{H^{a+1}}.
\]
The coefficient of the two-derivative difference becomes small near the
initial metric, while shortening the time interval controls the remaining
term.

The actual proof chain is more specific than the phrase ``apply a locally
Lipschitz existence theorem'' would suggest.  The declaration
\leanref{deTurckSobolevNonlinearitySymm} first symmetrizes the perturbation,
radially scales it into the metric/Sobolev ball, and extends the resulting
map to the ambient Sobolev completion.  This produces a globally Lipschitz
extension of the geometric nonlinearity on the space used by the fixed point.
The theorem
\leanref{quasilinear_maxreg_solution_of_nemytskii} then performs the contraction
in the forcing space.  It uses a retraction onto a forcing ball, but proves
that the fixed point already lies in that ball; the retraction is therefore
inactive at the solution.  Subsequent smallness and representation lemmas
likewise prove that the radial metric scaling is inactive, identifying the
abstract fixed point with the genuine Ricci--DeTurck remainder.

This distinction matters for auditability.  A generic local-Lipschitz theorem
exists elsewhere in the development, but it is not the theorem used by the
final Ricci-flow chain.  The high-level declaration
\leanref{quasilinear_metric_short_time_existence_of_nemytskii_data} is also a
specialized capstone with a large Ricci-shaped signature and a strong
all-orders forcing-regularity input.  Its current statement is not displayed
here because it exposes internal data, contains hypotheses not consumed by
its proof body, and could misleadingly appear to be a general strict-parabolic
existence theorem.  The forcing-regularity input is not assumed by the public
Ricci-flow theorem: it is discharged for the symmetrized DeTurck nonlinearity
by \leanref{deTurckRicci_forcingBootstrap_symm}.  A cleaner public wrapper is a
priority for the planned PDE generalization.

\subsection{Symmetry and the metric cone}
\label{subsec:short-time-metric-realization}

The fixed point lives in a linear tensor space, whereas the equation is
defined on the open cone of positive-definite symmetric tensors.  The
implementation handles these two constraints explicitly.  Slot
symmetrization is built into \leanref{ccTensor02Symm}; compatibility lemmas
relate it to the connection Laplacian and to metric realization.  Positivity
is reduced to the following fiberwise bound, defined in
\rffile{Analysis/Spectral/Intrinsic/MetricRealization/PosDefPerturbation.lean}:

\begin{Verbatim}[breaklines=true,breakanywhere=true,fontsize=\small]
def metricCauchySchwarzBound (g : SmoothRiemannianMetric I M)
    (h : forall x : M,
      TangentSpace I x ->L[Real] TangentSpace I x ->L[Real] Real)
    (delta : Real) : Prop :=
  forall (x : M) (v w : TangentSpace I x),
    |h x v w| <= delta * sqrt (g.inner x v v) * sqrt (g.inner x w w)
\end{Verbatim}

The definition states a uniform relative operator bound, not merely component
smallness in one chart.  Its key consequence is

\begin{Verbatim}[breaklines=true,breakanywhere=true,fontsize=\small]
theorem perturbedInner_self_lower_bound
    (g : SmoothRiemannianMetric I M)
    (h : forall x : M,
      TangentSpace I x ->L[Real] TangentSpace I x ->L[Real] Real)
    {delta : Real} (hdelta : metricCauchySchwarzBound g h delta) (x : M)
    (v : TangentSpace I x) :
    (1 - delta) * g.inner x v v <= perturbedInner g h x v v
\end{Verbatim}

Thus $\delta<1$ makes $g+h$ positive on every nonzero vector.  A separate
Sobolev-to-pointwise estimate produces the required fiberwise bound from the
small spectral perturbation; positivity is not inferred directly from an
unexplained Sobolev norm.  Finally,
\leanref{tensorSectionRealizeMetric_inner} states that the constructed metric
has inner product exactly
\[
  g_x(v,w)+\operatorname{sym}(S)_x(v,w).
\]
Consequently the analytic perturbation, after the proved-inactive radial
scaling, represents the geometric metric used in the DeTurck equation.

\subsection{Principal symbol and the DeTurck solution}
\label{subsec:short-time-deturck}

The coordinate calculation that justifies the gauge choice is packaged by
\leanref{deTurckRicciRHS_hasPrincipalSymbol_at_self}; its conclusion uses
\leanref{HasPrincipalSymbol} with coefficient \leanref{deTurckSymbolCoeff}:

\begin{Verbatim}[breaklines=true,breakanywhere=true,fontsize=\small]
theorem deTurckRicciRHS_hasPrincipalSymbol_at_self [I.Boundaryless]
    (g0 g_bg : SmoothRiemannianMetric I M) :
    HasPrincipalSymbol (I := I) (deTurckRicciRHS (I := I) g_bg) g0
      (isotropicSymbol
        (fun x : M =>
          TangentSpace I x ->l[Real] TangentSpace I x ->l[Real] Real)
        (deTurckSymbolCoeff (I := I) g0))
\end{Verbatim}

The two metrics have different roles: $g_0$ is the point at which the operator
is linearized, and $g_{\mathrm{bg}}$ is the background metric in the DeTurck
vector field.  The coefficient satisfies
\[
  \mathtt{deTurckSymbolCoeff}(g_0,x,\xi)
       =-\lvert\xi\rvert_{g_0}^{2},
\]
and \leanref{metricCovectorNormSq_pos} proves
$\lvert\xi\rvert_{g_0}^2>0$ for $\xi\ne0$, hence strict negativity of the
coefficient.  The proof constructs the chart second-order part, evaluates it on
quadratic test perturbations whose value and first jet vanish at the base
point, and verifies the cancellation between the Ricci and Lie-derivative
terms.

One terminological limitation should be explicit.  The outer definition of
\leanref{IsStrictlyParabolicMetricRHS} is an existential wrapper around
\leanref{HasPrincipalSymbol}.  The latter's second-order specification does
include the pointwise identity
$\sigma(x,\xi)t=-\lvert\xi\rvert_{g_0}^2t$ and positivity of the squared norm
for $\xi\ne0$; these clauses are discharged using
\leanref{deTurckSymbolCoeff_apply} and
\leanref{metricCovectorNormSq_pos}.  What the wrapper does not provide is a
separate uniform-neighborhood parabolic estimate.  Moreover, the symbol
theorem and the nonlinear existence engine are logically distinct:
the latter freezes $\Delta_{g_0}$ and estimates the remaining nonlinear terms
explicitly rather than deriving a black-box existence result from the
predicate alone.

After the fixed point, metric realization, and all-order bootstrap, the
geometric output is concise:

\begin{Verbatim}[breaklines=true,breakanywhere=true,fontsize=\small]
theorem deTurckRicci_solution_with_jointReg
    (g0 g_bg : SmoothRiemannianMetric I M) :
    exists T : Real, exists g_DT : Real -> SmoothRiemannianMetric I M,
      IsQuasilinearMetricParabolicSolution
        (deTurckRicciRHS (I := I) g_bg) g0 T g_DT /\
      JointChartGramSmooth (I := I) T g_DT
\end{Verbatim}

The first conjunct includes $T>0$, the initial value, and the
Ricci--DeTurck equation on $[0,T)$.  The second gives joint chartwise
regularity on the closed time slab used internally by the gauge construction.
The final public Ricci-flow theorem retains only the assertion on the natural
half-open interval $[0,T)$; it does not claim an endpoint value at its terminal
time.

\subsection{All-order regularity and removal of the gauge}
\label{subsec:short-time-gauge-removal}

The maximal-regularity fixed point initially controls two spatial derivatives
in $L^2$ and one time derivative.  To obtain a smooth metric rather than only
a strong Sobolev solution, the proof bootstraps the forcing.  For every mode,
all time derivatives satisfy differentiated scalar convolution identities;
weighted summable majorants at arbitrarily high spectral orders then justify
termwise differentiation.  Sobolev realization converts these coefficient
series into a tensor field jointly smooth in $(t,x)$, including at $t=0$.
This is an internal all-order spectral argument, not an appeal to an unstated
boundary regularity theorem.

Let $g_{\mathrm{DT}}(t)$ be the resulting solution and solve
\[
  \partial_t\Phi_t=-W(g_{\mathrm{DT}}(t))\circ\Phi_t,
  \qquad \Phi_0=\operatorname{id}_M.
\]
Because the vector field is initially defined on a one-sided time interval,
the ODE layer uses a Seeley extension in the time variable before applying
the standard smooth flow theory.  This extends the auxiliary vector field to
an open interval; it does not extend the Ricci--DeTurck metric or a Ricci flow
to negative original time.  Forward and reverse evolution maps give mutually
inverse smooth maps, and
\leanref{conjugating_diffeo_family_jointsmooth} packages the jointly smooth
family of diffeomorphisms.  Its implementation-facing signature is large, so
the paper records its mathematical output rather than reproducing the entire
type.

Define $g(t)=\Phi_t^*g_{\mathrm{DT}}(t)$.  Pullback differentiation gives
\begin{align*}
  \partial_t(\Phi_t^*g_{\mathrm{DT}})
  &=\Phi_t^*\bigl(\partial_tg_{\mathrm{DT}}
       -\Lie_{W(g_{\mathrm{DT}})}g_{\mathrm{DT}}\bigr)\\
  &=-2\Phi_t^*\Ric(g_{\mathrm{DT}})
   =-2\Ric(\Phi_t^*g_{\mathrm{DT}}).
\end{align*}
The last equality is proved through the connection: the transported
connection is torsion-free and metric-compatible, Levi--Civita uniqueness
identifies it with the pullback metric's connection, and curvature and Ricci
naturality follow.  This is why the differential-geometric foundation of
Section~\ref{sec:formal-differential-geometry} is an actual dependency of the
analytic theorem rather than preliminary notation.

The displayed calculation first yields the equation for $t>0$.  The initial
time is closed by the separate one-sided calculus theorem
\leanref{ricci_flow_pde_at_zero}:

\begin{Verbatim}[breaklines=true,breakanywhere=true,fontsize=\small]
theorem ricci_flow_pde_at_zero
    (g_fam : Real -> SmoothRiemannianMetric I M) {T : Real} (hT : 0 < T)
    (x : M) (v w : TangentSpace I x)
    (h_cont : ContinuousOn (fun s => (g_fam s).inner x v w) (Set.Ico 0 T))
    (h_ric_cont : ContinuousWithinAt
      (fun s => (-2) * ricciTensor (I := I) (g_fam s) x v w)
      (Set.Ioi 0) 0)
    (h_interior : forall t in Set.Ioo (0 : Real) T,
      HasDerivWithinAt (fun s => (g_fam s).inner x v w)
        ((-2) * ricciTensor (I := I) (g_fam t) x v w) (Set.Ici 0) t) :
    HasDerivWithinAt (fun s => (g_fam s).inner x v w)
      ((-2) * ricciTensor (I := I) (g_fam 0) x v w) (Set.Ici 0) 0
\end{Verbatim}

Its assumptions are exactly the data needed for the limiting argument:
continuity of the scalar metric component, right-continuity of the Ricci
right-hand side, and the interior PDE.  It concludes the right derivative at
zero.  Together with $\Phi_0=\operatorname{id}$ and
$g_{\mathrm{DT}}(0)=g_0$, this supplies every clause of the public theorem.

\section{Evolution equations and three-dimensional curvature algebra}
\label{sec:local-tensor-geometry}

This section records the project-local geometric calculations: the standing
Riemannian conventions, the Ricci-flow evolution identities, and the
three-dimensional curvature algebra.  These are the tensor-calculus and
coordinate/local-frame components that feed the analytic consumers of the
next section.

\subsection{Foundational geometric conventions}

For basic Riemannian geometry, see \S~\ref{subsec:BasicRiemGeom}.  For basic
tensor calculus, see \S~\ref{subsec:TensorCalculus}.

The preceding differential-geometric infrastructure section describes the
project's local Lean APIs for the Levi-Civita connection and induced tensor
covariant derivatives.  The following list fixes the textbook conventions used
in the remaining proof; it is not an additional unproved assumption.  Many of
these local pieces are supplied by the project infrastructure above, the
appendix gives their textbook formulas, and each later proof paragraph records
the exact Lean scope of the result it cites.

We use the standard smooth Riemannian tensor calculus associated to a metric
$g$, with following conventions.
\begin{enumerate}[label=(\arabic*)]
  \item The Levi-Civita connection $\nabla$ is defined on vector fields and
  extends to arbitrary tensors by linearity, the Leibniz rule, and commuting
  with contractions.

  \item The connection is torsion-free and metric-compatible:
  \[
  \nabla g=0.
  \]

  \item The Riemann curvature tensor is defined by the commutator of covariant
  derivatives:
  \[
  \Rm(X,Y)Z
  =
  \nabla_X\nabla_YZ
  -\nabla_Y\nabla_XZ
  -\nabla_{[X,Y]}Z.
  \]

  \item The Ricci tensor and scalar curvature are contractions of $\Rm$:
  \[
  \Ric_{ij}=g^{k\ell}R_{kij\ell},
  \qquad
  R=g^{ij}\Ric_{ij}.
  \]

  \item We use the algebraic symmetries of $\Rm$, the first and second Bianchi
  identities, and the contracted Bianchi identity.

  \item We use the standard commutator identities for covariant derivatives
  acting on tensors.

  \item Norms, traces, divergences, contractions, and the rough Laplacian
  \[
  \Delta T=g^{ij}\nabla_i\nabla_jT
  \]
  are formed using $g$ and $\nabla$.
\end{enumerate}

\subsection{Evolution equations}

 In this section we display the usual geometric
formulas and consequences of the Ricci-flow equation.  We explain their
checked local-component representations only where the formalization requires an
explicit choice of frame, time domain, or regularity hypothesis.  The proof of
short-time existence was given in
Section~\ref{subsec:ShortTimeExistence}; the maximum-principle consumers and
maximal-flow interface appear in Section~\ref{sec:analytic-consumers}.

The basic definitions and structure of Ricci flow are collected in
\rffile{Geometry/Flow/RicciFlow/Solution/Basic.lean}, and the evolution identities below live in
the files under \rfdir{Geometry/Flow/RicciFlow/Evolution}.

\begin{definition}[Ricci-flow solution in the project]
\label{def:project-ricci-flow-solution}
Let \(D\) be a real time interval.  A candidate
\leanref{SolutionOn} \(D\) consists of a one-parameter family \(g(t)\) of
Riemannian metrics; its Levi--Civita connection, Riemann tensor, Ricci tensor,
and scalar curvature are derived from \(g(t)\), rather than supplied as
independent data.  The predicate \leanref{IsSolutionOn} holds when the metric
and its derived connection have the required interval-wise smoothness,
\[
  \frac{\partial}{\partial t}g(t)=-2\Ric(g(t))
  \qquad (t\in D),
\]
and the scalar curvature, Ricci tensor, Riemann tensor, and Ricci-norm fields
carry the continuity and differentiability used by the later evolution and
maximum-principle APIs.  Thus the mathematical equation is the usual Ricci
flow equation; the additional fields in \leanref{IsSolutionOn} are the
project's theorem-facing regularity package for that metric solution.
\end{definition}

\subsubsection{Metric and connection variation}

\begin{lemma}[Evolution of the inverse metric]
\label{lem:evol-inverse-metric}
Along Ricci flow,
\[
  \partial_t g^{ij}=2\Ric^{ij}.
\]
\end{lemma}

\begin{proof}
Differentiate the inverse identity
\[
  g^{ia}g_{aj}=\del^i_j.
\]
Since the frame is fixed in time,
\[
  (\partial_t g^{ia})g_{aj}
  +g^{ia}(\partial_t g_{aj})=0.
\]
Substituting the Ricci-flow equation \(\partial_tg_{aj}=-2\Ric_{aj}\) and
contracting with the inverse metric gives
\[
  \partial_tg^{ij}
  =
  2g^{ia}g^{jb}\Ric_{ab}
  =
  2\Ric^{ij}.
\]

The Lean development follows almost verbatim.  For a coordinate frame,
\leanref{coordInvEvol} differentiates the canonical coordinate inverse and
identifies its entries with the inverse-metric components.  The theorem
\leanref{evol_inverse_metric_inFrame} in
\rffile{Geometry/Flow/RicciFlow/Evolution/Metric/Evolution.lean} gives the corresponding
fixed-frame statement, producing a \texttt{HasDerivWithinAt} statement.
\end{proof}

\begin{lemma}[Evolution of the Levi-Civita connection]
\label{lem:evol-christoffel}
Along Ricci flow, the Christoffel symbols in a fixed local frame satisfy
\[
  \partial_t\Gamma^k_{ij}
  =
  -g^{k\ell}
  \left(
    \nabla_i\Ric_{j\ell}
    +
    \nabla_j\Ric_{i\ell}
    -
    \nabla_\ell\Ric_{ij}
  \right).
\]
\end{lemma}

\begin{proof}
Let
\[
  h=\partial_tg
  \qquad\text{and}\qquad
  A(X,Y)=\partial_t\bigl(\nabla^t_XY\bigr),
\]
where \(X\) and \(Y\) are independent of time.  Although the Christoffel symbols
themselves are not tensorial, the difference of two connections is a tensor;
consequently \(A\) is a well-defined \((1,2)\)-tensor.  Differentiating the
Koszul formula gives the invariant metric-variation identity
\[
  2g\bigl(A(X,Y),Z\bigr)
  =
  (\nabla_Xh)(Y,Z)
  +
  (\nabla_Yh)(X,Z)
  -
  (\nabla_Zh)(X,Y).
\]
For Ricci flow, \(h=-2\Ric\), and therefore
\[
  g\bigl(A(X,Y),Z\bigr)
  =
  -(\nabla_X\Ric)(Y,Z)
  -(\nabla_Y\Ric)(X,Z)
  +(\nabla_Z\Ric)(X,Y).
\]
Taking \(X=e_i\), \(Y=e_j\), and \(Z=e_\ell\) in a fixed local frame yields the
lowered formula
\[
  g\bigl(A(e_i,e_j),e_\ell\bigr)
  =
  -\nabla_i\Ric_{j\ell}
  -\nabla_j\Ric_{i\ell}
  +\nabla_\ell\Ric_{ij}.
\]
Raising the final index gives the stated evolution equation.

The formal proof follows this order.  It freezes the metric at the
differentiating time and differentiates the lowered pairing
\[
  g(t)\!\left(
    (\nabla^s-\nabla^t)_{e_i}e_j,e_\ell
  \right).
\]
A finite-difference form of the Koszul identity identifies its derivative with
the three covariant derivatives of Ricci appearing above.  Only after this
lowered identity has been established does Lean apply the inverse metric to
recover the \(k\)-th connection coefficient.  This separation makes the
tensorial connection variation explicit and avoids any implicit identification
of vectors, covectors, and coordinate components.

The lower-level calculation is developed in
\rffile{Geometry/Flow/RicciFlow/Evolution/Connection/Pairing.lean}, and the raising and
assembly steps are in \rffile{Geometry/Flow/RicciFlow/Evolution/Connection/Evolution.lean}.
The theorem \leanref{evol_christoffel_inFrame} states the result as a
within-time derivative of the fixed-frame Christoffel component on the interval
carrier; its local space--time regularity hypothesis records the mixed
time--space differentiability needed to differentiate the metric components.
The invariant derivation above is the textbook calculation; it is not repeated
in the background appendix.
\end{proof}

These metric and connection variation identities are the inputs for the
subsequent evolution equations for the Ricci tensor, scalar curvature, and the
pinching quantities used in Hamilton's argument.

\subsubsection{Ricci, scalar curvature, and Ricci norm}

The preceding variation formulas are the input for the curvature evolution
identities.  The Ricci evolution equation is obtained by differentiating the curvature
trace through the connection variation, while the scalar and norm evolution equations are
obtained by tracing and contracting the Ricci evolution equation together with the
inverse-metric evolution.

\begin{lemma}[Evolution of Ricci]
\label{lem:evol-ricci}
Along Ricci flow,
\[
\partial_t\Ric_{ij}
=
\Delta\Ric_{ij}
+2R_{ik\ell j}\Ric^{k\ell}
-2\Ric_i{}^k\Ric_{kj}.
\]
Equivalently,
\[
(\partial_t-\Delta)\Ric_{ij}
=
2R_{ik\ell j}\Ric^{k\ell}
-2\Ric_i{}^k\Ric_{kj}.
\]
Here the order of the curvature slots is fixed by
\(R_{ijkl}=\langle R(e_i,e_j)e_k,e_l\rangle\) and
\(\Ric_{ij}=g^{k\ell}R_{kij\ell}\).  Thus
\(R_{ik\ell j}=R_{kij\ell}\) after the two skew symmetries.  The alternative
ordering $R_{ikj\ell}$ would have the opposite trace and is not compatible
with the scalar evolution formula below.
\end{lemma}

\begin{proof}
Let
\[
  A^k{}_{ij}:=\partial_t\Gamma^k_{ij}.
\]
Differentiating the coordinate curvature trace gives the first-variation formula
\[
  \partial_t\Ric_{ij}
  =
  \nabla_k A^k{}_{ij}
  -
  \nabla_i A^k{}_{kj}.
\]
Substituting the Christoffel evolution from Lemma~\ref{lem:evol-christoffel}
expresses the right-hand side in terms of second covariant derivatives of
\(\Ric\).  Since the Levi-Civita connection is metric-compatible, derivatives of
the raised index are rewritten using \(\nabla g^{-1}=0\).  The remaining
second-derivative terms are then rearranged by the contracted Bianchi identity
and the Ricci commutator identity: the second derivatives combine to
\[
  \Delta\Ric_{ij},
\]
and the commutator terms give
\[
  2R_{ik\ell j}\Ric^{k\ell}
  -
  2\Ric_i{}^k\Ric_{kj}.
\]

In Lean this calculation is exposed as a coordinate-frame component theorem
rather than as a single invariant tensor equality.  The endpoint
\leanref{coordRicciEvol} in
\rffile{Geometry/Flow/RicciFlow/Evolution/Ricci/CoordinateIdentities.lean} proves the
within-time derivative statement for the coordinate component
\(\Ric_{ij}(t,x_0)\) at the center of a coordinate frame.  The formal proof is
organized in the same mathematical order: it differentiates the Christoffel
trace formula for Ricci, substitutes the connection variation, rewrites the
raised-index terms using metric compatibility, and then applies the contracted
commutator package to reduce the expanded expression to the displayed heat
equation.
\end{proof}

\begin{definition}[Lichnerowicz Laplacian]
\label{def:lichnerowicz-laplacian}
For a symmetric $2$-tensor $h$, define
\[
(\Delta_Lh)_{ij}
=
\Delta h_{ij}
+2R_{ik\ell j}h^{k\ell}
-\Ric_i{}^kh_{kj}
-\Ric_j{}^kh_{ki}.
\]
\end{definition}

\begin{corollary}[Ricci evolves by the Lichnerowicz heat equation]
\label{cor:ricci-lichnerowicz}
Along Ricci flow,
\[
\partial_t\Ric=\Delta_L\Ric.
\]
\end{corollary}

\begin{proof}
Apply the Ricci evolution formula with \(h=\Ric\) in the definition of
\(\Delta_L\).  The two Ricci-action terms in \(\Delta_L\Ric\) agree after using
the symmetry of \(\Ric\):
\[
  -\Ric_i{}^k\Ric_{kj}
  -
  \Ric_j{}^k\Ric_{ki}
  =
  -2\Ric_i{}^k\Ric_{kj}.
\]
Thus the Lichnerowicz formula is exactly the Ricci evolution equation written
with the standard curvature-linearization operator.

The Lean rewrite is kept as a separate algebraic layer in
\rffile{Geometry/Flow/RicciFlow/Evolution/Ricci/Lichnerowicz.lean}.  It introduces the
component right-hand side \leanref{lichnerowiczRHSInFrame} and proves that the
Ricci evolution producer is equivalent to the \(\partial_t\Ric=\Delta_L\Ric\)
component statement.
\end{proof}

\begin{lemma}[Evolution of scalar curvature]
\label{lem:evol-scalar}
Along Ricci flow,
\[
\partial_tR=\Delta R+2\norm{\Ric}^2.
\]
Equivalently,
\[
(\partial_t-\Delta)R=2\norm{\Ric}^2.
\]
\end{lemma}

\begin{proof}
Since
\[
  R=g^{ij}\Ric_{ij},
\]
we differentiate the trace:
\[
  \partial_tR
  =
  (\partial_tg^{ij})\Ric_{ij}
  +
  g^{ij}\partial_t\Ric_{ij}.
\]
The inverse-metric evolution gives
\[
  (\partial_tg^{ij})\Ric_{ij}
  =
  2\Ric^{ij}\Ric_{ij}
  =
  2\norm{\Ric}^2.
\]
Tracing the Ricci evolution gives
\[
  g^{ij}\Delta\Ric_{ij}
  +
  2g^{ij}R_{ik\ell j}\Ric^{k\ell}
  -
  2g^{ij}\Ric_i{}^k\Ric_{kj}.
\]
The trace of the rough Laplacian becomes \(\Delta R\), and the two zero-order
trace terms cancel by the usual contracted-curvature identities.  Hence
\[
  \partial_tR=\Delta R+2\norm{\Ric}^2.
\]

The Lean development packages this trace calculation through
\leanref{scalar_evolution_of_smooth_solution} in \rffile{Geometry/Flow/RicciFlow/Evolution/Scalar/Basic.lean}.
Internally, the proof separates the raw trace derivative of \(R=g^{ij}\Ric_{ij}\)
from the contracted-Bianchi reduction that identifies the Ricci-Hessian trace
with the scalar Laplacian.  The endpoint is stated for the smooth Ricci-flow
solution view and produces the same equation in the scalar evolution API.
\end{proof}

\begin{lemma}[Evolution of the Ricci norm]
\label{lem:evol-ricci-norm}
Along Ricci flow,
\[
(\partial_t-\Delta)\norm{\Ric}^2
=
-2\norm{\nabla\Ric}^2
+4R_{ik\ell j}\Ric^{ij}\Ric^{k\ell}.
\]
\end{lemma}

\begin{proof}
The norm is the full metric contraction
\[
  \norm{\Ric}^2
  =
  g^{ia}g^{jb}\Ric_{ij}\Ric_{ab}.
\]
Differentiating in time gives two kinds of terms.  The derivative of the two
inverse metrics is handled by Lemma~\ref{lem:evol-inverse-metric}, and the
derivative of the two Ricci factors by Lemma~\ref{lem:evol-ricci}.  The
quadratic Ricci terms created by differentiating the inverse metrics cancel the
quadratic Ricci terms coming from the Ricci evolution.

For the Laplacian, the standard product rule gives
\[
  \Delta\norm{\Ric}^2
  =
  2\ip{\Delta\Ric}{\Ric}
  +
  2\norm{\nabla\Ric}^2.
\]
Subtracting this from the time derivative leaves the negative gradient term
\(-2\norm{\nabla\Ric}^2\) and the curvature reaction
\(4R_{ik\ell j}\Ric^{ij}\Ric^{k\ell}\).

The Lean endpoint is \leanref{ricciHeatSmooth} in
\rffile{Geometry/Flow/RicciFlow/Evolution/Ricci/NormEvolution.lean}.  It builds the norm evolution from
four reusable pieces: inverse-metric evolution, Ricci evolution, symmetry of the
inverse metric and Ricci tensor, and the Laplacian expansion of the Ricci norm.
The resulting proposition is \leanref{RicciNormHeatEquationOn}, stated
for the canonical component functions \leanref{ricciNorm}, \leanref{ricciNormLap},
\leanref{ricciGradSq}, and \leanref{ricciReact}.
\end{proof}

\subsection{Three-dimensional curvature algebra}

In dimension three, the Ricci tensor controls the full Riemann curvature tensor.
This is one of the special algebraic facts that makes Hamilton's proof work:
the tensor maximum-principle calculations and the Ricci pinching estimates can
be reduced to explicit eigenvalue algebra.  In the Lean development, this part
of the proof is also a useful separation point.  The purely finite-dimensional
curvature algebra is implemented separately, and the Ricci-flow preservation
arguments then consume that algebra through pointwise and solution-level
interfaces.

\begin{lemma}[Three-dimensional curvature identities]
\label{lem:3d-curvature-identities}
In dimension $3$, the Riemann tensor is determined algebraically by the Ricci
tensor and scalar curvature:
\[
R_{ijkl}
=
\Ric_{i\ell}g_{jk}-\Ric_{j\ell}g_{ik}
-\Ric_{ik}g_{j\ell}+\Ric_{jk}g_{i\ell}
-\frac R2\left(g_{i\ell}g_{jk}-g_{j\ell}g_{ik}\right).
\]
If $\Ric$ is diagonal with eigenvalues $\lambda_1,\lambda_2,\lambda_3$, then, in
the sectional-numerator convention $K_{ij}=\Rm_{04}(e_i,e_j,e_j,e_i)$,
\[
K_{ij}=R_{ijji}=\frac{\lambda_i+\lambda_j-\lambda_k}{2},
\]
where $\{i,j,k\}=\{1,2,3\}$.
\end{lemma}

\begin{proof}[Lean-facing algebra]
The dimension-three curvature algebra is project-local and finite-dimensional.
The
pure algebra lives in \rffile{Geometry/Curvature/DimensionThree/CurvatureAlgebra.lean}.  Lean
works with an abstract four-index array
\[
  R_{ijkl},\qquad i,j,k,l\in \operatorname{Fin} 3,
\]
assuming only the algebraic curvature symmetries: skewness in the first pair,
skewness in the second pair, and pair symmetry.  It defines Ricci and scalar
traces by explicit finite sums over the three basis vectors and then forms the
residual obtained by subtracting the displayed Riemann-from-Ricci expression.

The proof does not use a Weyl decomposition.  Instead, Lean proves that this
residual has the same skew/skew and pair symmetries as the original curvature
array.  In dimension three, these symmetries reduce the verification to the six
ordered block components
\[
  (01,01),\ (01,02),\ (01,12),\ (02,02),\ (02,12),\ (12,12).
\]
The project checks these six components directly by expanding the Ricci and
scalar traces and using the curvature symmetries.  The theorem
\leanref{displayedRiemannFromRicci3D_of_algebraic_curvature_symmetries}
then promotes those six component checks to the full four-index identity.

The realized bridge lives in \rffile{Geometry/Curvature/DimensionThree/RiemannFromRicci.lean}.  It
transfers the finite \(\operatorname{Fin} 3\) component identity to a pointwise
orthonormal-frame statement for the lowered curvature tensor
\[
  \Rm_{04}(X,Y,Z,W)=\langle R(X,Y)Z,W\rangle .
\]
The bridge records the necessary trace and slot-convention data through
\leanref{RiemannFromRicci3DTraceDataAt} and
\leanref{rm04Comp_displayedRiemannFromRicci3D_at}.  Finally,
\rffile{Geometry/Curvature/DimensionThree/RicciControlsRm.lean} handles the Ricci-eigenbasis
sectional-curvature consequence: after diagonalizing Ricci with eigenvalues
\(\lambda_1,\lambda_2,\lambda_3\), substituting into the realized formula gives
\[
  K_{12}=\frac{\lambda_1+\lambda_2-\lambda_3}{2},
  \quad
  K_{13}=\frac{\lambda_1+\lambda_3-\lambda_2}{2},
  \quad
  K_{23}=\frac{\lambda_2+\lambda_3-\lambda_1}{2}.
\]
\end{proof}

A particularly difficult neighboring example, used in the
three-dimensional curvature-evolution was the Uhlenbeck reaction match in
\rffile{Geometry/Curvature/DimensionThree/UhlReaction3.lean}.
Uhlenbeck's trick is useful for proving the Hamilton--Ivey estimate, which in turn has the fundamental application of showing that 3-dimensional singularity models have nonnegative curvature operator. However, it is not used in the original proof of Hamilton's theorem. We use Uhlenback's trick to simplify the proof of Shi's derivative of curvature estimates. In \cite{MR862046}, Hamilton re-proved his 3-dimensional Ricci pinching improves estimate using its tensor formulation, with the aid of Uhlenbeck's trick, as a warm-up for his more complicated 4-dimensional calculations.
For a symmetric array
\(R\colon\operatorname{Fin}3\times\operatorname{Fin}3\to\mathbb R\), we
identify the reaction obtained by differentiating the
Riemann-from-Ricci formula with the quadratic \(B\)-tensor reaction and the
Ricci drift:
\begin{theorem}[Three-dimensional reaction match ]
\label{thm:dim3-reaction-match}
For every \(a,b,c,d\in\operatorname{Fin}3\),
\[
  \operatorname{KNQ}(R)_{abcd}+G(R)_{abcd}
  =
  -2B^\#(R)_{abcd}-\operatorname{drift}(R)_{abcd}.
\]

\end{theorem}
Here \(\operatorname{KNQ}\) is the Kulkarni--Nomizu expression formed from the
Ricci and scalar reaction terms, \(G\) contains the
\(\partial_tg=-2\Ric\) cross-terms, and \(B^\#\) is Hamilton's four-term
combination of the quadratic \(B\)-tensor.  A direct proof expands
\(3^4=81\) indexed component cases.  Individual cases were checkable, but the
full expansion hit a verification-time wall.  The successful proof instead
derived closed forms for the curvature action and the \(B\)-tensor, isolated
the \(3\times3\) adjugate/minor identity, and assembled the displayed theorem
by linear combination.  This episode sharpened the practical lesson from the
pinching proof: localization to linear algebra helps, but a
structural linear-algebra decomposition is still preferable to dozens of
independent component verifications.

\section{Maximum principles and Ricci pinching}
\label{sec:analytic-consumers}

This section records the scalar and tensor maximum-principle arguments and the
Ricci-preservation and pinching estimates that consume them.
The local tensor identities from Section~\ref{sec:local-tensor-geometry} supply
the differential-geometric inputs for these consumers.

\subsection{Maximum principles}

This section records the maximum-principle layer used by the later Ricci-flow
arguments.  The scalar weak maximum principle is a project-local comparison
package consumed by the scalar-curvature lower-bound and finite-time arguments,
and the tensor weak maximum principle is a project-local section-backed package
for symmetric two-tensors over a time-dependent metric background.  We record the scalar strong maximum principle for later uses, but it is not currently used by the selected Hamilton endpoint.

\begin{theorem}[Scalar weak maximum principle: supersolutions]
\label{thm:scalar-wmp-super}
Let $(M,g(t))$ be a smooth one-parameter family of metrics on a closed
manifold, and let $u:M\times[0,T]\to\mathbb R$ be smooth. Suppose
\[
\partial_tu
\ge
\Delta_{g(t)}u+
\ip{X}{\nabla u}_{g(t)}+F(u,t),
\]
where $X$ is a smooth time-dependent vector field and $F$ is continuous and
locally Lipschitz in $u$, uniformly on compact subsets of the value--time
domain relevant to the comparison.

Let $c(t)$ solve
\[
c'=F(c,t).
\]
If
\[
u(\cdot,0)\ge c(0),
\]
then
\[
u(x,t)\ge c(t)
\]
for all $(x,t)\in M\times[0,T]$.
\end{theorem}

\begin{proof}[Lean-facing calculation]
The scalar weak maximum-principle layer is project-local in
\rffile{Analysis/Parabolic/MaximumPrinciple/Scalar/Weak.lean}.  The supersolution endpoint is
\leanref{scalar_wmp_super_theorem_7_1}, built from the regular positive-time
comparison theorem
\leanref{scalar_wmp_supersolutions_of_lipschitz_on_value_set_of_regular_positive_time}.
This is a proof layer used by later Ricci-flow scalar estimates, not a hidden
global assumption.

Lean introduces the comparison function
\[
  w=u-c
\]
and then the weighted barrier
\[
  e^{-Kt}w,
\]
where $K$ is a Lipschitz constant for $F$ on the compact value set relevant to
the comparison argument.  The lemmas
\leanref{reaction_difference_lower_bound_on_negative_region} and
\leanref{negative_region_parabolic_lower_bound} convert the nonlinear reaction
comparison into a linear lower bound on the negative region of $w$.  The contradiction argument is isolated in
\leanref{strict_barrier_nonnegative_of_positive_time}.  If the weighted barrier
were negative, compactness of a closed positive-time slab would provide a point
where it attains its negative minimum.  At that minimizing point one has
\[
  \nabla w=0,
  \qquad
  \Delta w\ge0,
  \qquad
  \partial_tw\le0,
\]
while the strict weighted inequality gives the opposite sign.  This
contradiction proves $w\ge0$, hence $u\ge c$.
\end{proof}

\begin{corollary}[Scalar weak maximum principle: subsolutions]
\label{thm:scalar-wmp-sub}
In the setting of Theorem~\ref{thm:scalar-wmp-super}, suppose instead that
\[
\partial_tu
\le
\Delta_{g(t)}u+
\ip{X}{\nabla u}_{g(t)}+F(u,t).
\]
If $u(\cdot,0)\le c(0)$ and $c'=F(c,t)$, then
\[
u(x,t)\le c(t)
\]
for all $(x,t)\in M\times[0,T]$.
\end{corollary}

\begin{proof}[Lean-facing calculation]
The subsolution theorem is the wrapper
\leanref{scalar_wmp_sub_theorem_7_2} in
\rffile{Analysis/Parabolic/MaximumPrinciple/Scalar/Weak.lean}.  Lean obtains it by applying the
supersolution theorem to $-u$, rewritten through the same scalar WMP API.
\end{proof}

\begin{corollary}[Scalar curvature lower bound]
\label{cor:scalar-lower-bound}
Let $n\ge1$, and let $g(t)$ be a Ricci flow on a closed $n$-manifold. Let
\[
c_0=\inf_MR(\cdot,0).
\]
Then, as long as the denominator is positive,
\[
R(x,t)
\ge
\frac{c_0}{1-\frac2n c_0t}.
\]
In particular, if $R(\cdot,0)>0$, then $R(\cdot,t)>0$ for as long as the
solution exists.
\end{corollary}

\begin{proof}[Lean-facing calculation]
The scalar lower-bound route is project-local in
\rffile{Geometry/Flow/RicciFlow/Preservation/ScalarLowerBound.lean}.  The ODE-comparison core is
\leanref{scalar_curvature_lower_bound_of_parabolic_inequality}; the Ricci-flow
producer wrappers are
\leanref{scalar_curvature_lower_bound_of_scalarEvolution} and
\leanref{scalar_curvature_lower_bound_of_scalarEvolution_initialMinimum}.

The calculation starts from the scalar evolution equation supplied by
\leanref{scalar_evolution_of_smooth_solution}:
\[
  \partial_tR=\Delta R+2\norm{\Ric}^2.
\]
The trace/norm inequality
\[
  \norm{\Ric}^2\ge \frac1n R^2
\]
gives the parabolic inequality
\[
  \partial_tR
  \ge
  \Delta R+\frac2nR^2.
\]
Lean names the comparison solution
\[
  c(t)=\frac{c_0}{1-\frac2n c_0t}
\]
by \leanref{scalarLowerBarrier}; its within-derivative identity is
\leanref{scalarLowerBarrier_hasDerivWithinAt}.  The scalar WMP applied to this
barrier gives the displayed lower bound.  The regularity hypotheses needed by
the WMP are kept explicit through \leanref{scalarRegOfSmooth}, and the Hamilton
endpoint uses this package through \leanref{hamilton_scalar_weak_maximum_principle_regularity_on_interval}.
\end{proof}

\begin{corollary}[Positive scalar curvature forces finite maximal time]
\label{cor:positive-scalar-finite-time}
Let $n\ge1$, and let $g(t)$ be a Ricci flow on a closed $n$-manifold with
\[
R(g(0))>0.
\]
Let $[0,T_{\max})$ be its maximal time interval of existence. Then
\[
T_{\max}
\le
\frac{n}{2\min_M R(g(0))}
<\infty .
\]
\end{corollary}

\begin{proof}[Lean-facing assembly]
This is a downstream scalar consumer once maximal-flow data are supplied.  The
maximal-flow package is produced by the checked Hamilton assembly using the
closed uniform-existence theorem described in
Subsection~\ref{subsec:uniform-existence}.  The scalar lower-bound and endpoint
argument are encoded in
\rffile{Geometry/Flow/RicciFlow/Estimates/FiniteTime/Scalar.lean} and then packaged for the
Hamilton endpoint by \leanref{hamilton_extinction_time_bound} in
\rffile{Geometry/Flow/RicciFlow/DimensionThree/PositiveRicci/Flow.lean}.

Set
\[
  c_0:=\min_M R(g(0))>0.
\]
Corollary~\ref{cor:scalar-lower-bound} gives
\[
  R(x,t)\ge \frac{c_0}{1-\frac2n c_0t}
\]
whenever the flow exists and the denominator is positive.  The barrier on the
right blows up as
\[
  t\nearrow \frac{n}{2c_0}.
\]
The theorem \leanref{scalar_endpoint_le_blowupTime_of_lower_barrier_bound}
turns this explicit barrier blow-up into an endpoint bound: if the maximal flow
existed smoothly past $n/(2c_0)$, then $R$ would be bounded on the compact
space--time slab
\[
  M\times\left[0,\frac{n}{2c_0}\right],
\]
contradicting the lower bound.  Thus
\[
  T_{\max}\le \frac{n}{2c_0}.
\]
\end{proof}

The tensor maximum principle used by Hamilton must be formulated for a
time-dependent metric background.  The current Lean route includes this
metric dependence in the formal setup and follows the corrected
time-dependent barrier argument, rather than reducing to a static-metric
shortcut; compare Hamilton's tensor maximum-principle framework
\cite{MR862046}.

\begin{theorem}[Hamilton weak maximum principle for symmetric $2$-tensors]
\label{thm:hamilton-tensor-wmp}
Let $(M,g(t))$ be a smooth one-parameter family of metrics on a closed
manifold, and let $S(t)$ be a smooth family of symmetric $2$-tensors satisfying
\[
(\partial_t-\Delta)S_{ij}
\ge
X^k\nabla_kS_{ij}+N_{ij}(S,g,t),
\]
where $X$ is a smooth time-dependent vector field and $N$ is a smooth
fiberwise algebraic symmetric $2$-tensor expression.
Here all covariant derivatives and the rough Laplacian are taken with respect
to \(g(t)\), and \(A\ge B\) means that \(A-B\) is nonnegative in
quadratic-form order.

Assume the null-eigenvector condition: whenever $S\ge0$ at a point and
$S(v,v)=0$, one has
\[
N(S,g,t)(v,v)\ge0.
\]
If $S(\cdot,0)\ge0$, then $S(\cdot,t)\ge0$ for all $t$.
\end{theorem}

\begin{proof}[Lean-facing calculation]
The tensor WMP is a project-local section-backed package in
\rfdir{Analysis/Parabolic/MaximumPrinciple/Tensor}.  The theorem-facing statements are
\leanref{wmp_section_sec}, \leanref{tensor_wmp}, and
\leanref{hamilton_tensor_wmp_section}.

The proof uses Hamilton's positive barrier
\[
  S_\eta=S+\eta e^{At}g.
\]
The first-null geometry and the scalar test function are handled by the
compactness and first-null layers, while \leanref{strictCert_sec} packages the
strict certificate used at the first null point.  At that point the scalar test
\[
  \phi=S_\eta(v,v)
\]
satisfies
\[
  \phi=0,
  \qquad
  \nabla\phi=0,
  \qquad
  \Delta\phi\ge0,
  \qquad
  \partial_t\phi\le0.
\]
The null vector is extended locally so that the scalar second-derivative test
is valid at the contact point.  The null-eigenvector condition applies to
$N(S_\eta,g,t)$, not directly to $N(S,g,t)$; smooth fiberwise dependence gives
a Lipschitz bound for $N(S)-N(S_\eta)$.  The certificate also retains the term
$(\partial_t-\Delta)g=\partial_tg$.  Choosing $A$ large absorbs both errors in
the positive $A\eta e^{At}g$ contribution.  Together with the vanishing drift
and the sign of the Laplacian, this yields the first-null contradiction.
Finally, the barrier-limit theorem sends $\eta\downarrow0$, giving $S\ge0$ on
the whole time interval.
\end{proof}

The scalar strong maximum principle and forward uniqueness for complete
bounded-curvature Ricci flows are important classical tools in the wider
Hamilton--Perelman program, especially in the analysis and classification of
limit flows.  They are recorded here for context, but neither lies on the
selected proof path for Hamilton's positive-Ricci
three-manifold theorem.

\begin{theorem}[Scalar strong maximum principle]
\label{bb:scalar-strong-mp}
Let $(N,h(t))$, $t\in[a,b]$, be a complete connected Ricci flow background
with the regularity and bounded-geometry assumptions needed in the blow-up
limit. Let $u\ge0$ be a smooth solution or supersolution of a scalar parabolic
inequality of the form
\[
\partial_tu\ge \Delta_{h(t)}u+\ip{X}{\nabla u}+cu,
\]
where $X$ and $c$ are smooth and locally bounded.

If $u(\cdot,t_0)$ is not identically zero for some $t_0\in[a,b)$, then
\[
u>0
\]
on $N\times(t_0,b]$.
\end{theorem}
\begin{theorem}[Forward uniqueness for complete bounded-curvature Ricci flow]
\label{bb:complete-rf-uniqueness}
Let $h_1(t)$ and $h_2(t)$, $t\in[t_0,t_1]$, be complete Ricci flows
with uniformly bounded curvature on each compact time subinterval and with
$h_1(t_0)=h_2(t_0)$.  Under the standard hypotheses of the complete
bounded-curvature uniqueness theorem, $h_1(t)=h_2(t)$ for
$t\in[t_0,t_1]$.
\end{theorem}

Indeed, the checked theorem \leanref{round_at_zero_of_smooth_cgh} uses a shorter terminal-slice
argument.  Smooth pointed convergence transfers the basepoint normalization
\[
  R(g_\infty(0))(x_\infty)=1,
\]
while improved pinching gives \(\Ric^\circ(g_\infty(0))=0\).  Thus the terminal
slice is Einstein.  On the connected boundaryless limit, the contracted Bianchi
identity makes its scalar curvature spatially constant; the basepoint value
then makes that constant positive.  No scalar strong maximum principle,
backward propagation, or complete-flow forward uniqueness is needed for this
endpoint.  Those results remain useful expository background, not deferred
inputs of the current theorem.
\subsection{Preservation of Ricci positivity and Ricci pinching}

The next step applies the maximum-principle layer to the three-dimensional
curvature algebra.  There are two related but distinct preservation statements.
The nonnegative-Ricci endpoint is the \(\del=0\) preservation result for
\(S=\Ric\).  The shifted pinching theorem is the strict tensor WMP application
for
\[
  S=\Ric-\del Rg,\qquad 0<\del<1/3.
\]
The distinction is important in the Lean organization: the nonnegative case is
packaged separately, while the strict shifted theorem uses the first-null
algebra available only below the boundary value \(\del=1/3\).

\begin{lemma}[Preservation of nonnegative Ricci curvature in dimension $3$]
\label{lem:preserve-ricci-nonnegative}
Let $g(t)$ be a Ricci flow on a closed three-manifold. If
\[
\Ric(g(0))\ge0,
\]
then
\[
\Ric(g(t))\ge0
\]
for all times for which the flow exists.
\end{lemma}

\begin{proof}[Lean-facing calculation]
The preservation route is project-local in the Ricci-flow evolution layer, in
\rffile{Geometry/Flow/RicciFlow/Preservation/RicciPinching.lean}.
The smooth-solution endpoint for the nonnegative case is
\leanref{ricci_nonnegative_of_closed_solution_wmp_data}, and the Hamilton endpoint consumes it through
\leanref{hamilton_rescaled_ricci_nonnegative}.

Lean applies the tensor weak maximum principle to \(S=\Ric\).  By
Lemma~\ref{lem:evol-ricci},
\[
(\partial_t-\Delta)\Ric_{ij}=N_{ij},
\qquad
N_{ij}=2R_{ik\ell j}\Ric^{k\ell}-2\Ric_i{}^k\Ric_{kj}.
\]
At a point, diagonalize \(\Ric\) with eigenvalues
\(\lambda_1,\lambda_2,\lambda_3\).  If \(\Ric\ge0\) and \(e_1\) is a null
eigenvector, then \(\lambda_1=0\), and the three-dimensional curvature
identity gives
\[
N_{11}=(\lambda_2-\lambda_3)^2\ge0.
\]
This is the direct null-eigenvector calculation for the nonnegative Ricci
preservation theorem.

For code reuse, the final endpoint \leanref{ricci_nonnegative_of_closed_solution_wmp_data} is
packaged through the nonnegative-\(\del\) wrapper
\leanref{pinch_sol_closed_nonneg} at \(\del=0\).  This wrapper uses the same
parabolic producer \leanref{pinchParabolic} and barrier regularity input
\leanref{pinchBarrierReg}, but it is not the strict shifted pinching theorem
\leanref{pinch_sol_closed} evaluated at \(\del=0\).
\end{proof}

\begin{lemma}[Preservation of Ricci pinching in dimension $3$]
\label{lem:preserve-ricci-pinching}
Let $g(t)$ be a Ricci flow on a closed three-manifold. Fix
\[
0<\del<\frac13.
\]

If
\[
\Ric(g(0))\ge \del R(g(0))g(0),
\]
then
\[
\Ric(g(t))\ge \del R(g(t))g(t)
\]
for all times for which the flow exists.
\end{lemma}

\begin{proof}[Lean-facing calculation]
This is the strict shifted WMP application for
\[
S=\Ric-\del Rg,
\qquad 0<\del<1/3.
\]
The solution-level endpoint is \leanref{pinch_sol_closed} in
\rffile{Geometry/Flow/RicciFlow/Preservation/RicciPinching.lean}.  Strict initial positivity
is supplied by the wrappers \leanref{pinch_init_sol_lt} and
\leanref{strict_pinch_sol_lt}.  The parabolic producer and barrier regularity
inputs are \leanref{pinchParabolic} and \leanref{pinchBarrierReg}.

Using Lemmas~\ref{lem:evol-ricci} and~\ref{lem:evol-scalar}, together with
\(\partial_tg=-2\Ric\) and \(\Delta g=0\), one obtains
\[
(\partial_t-\Delta)S_{ij}
=
N_{ij}
-2\del\left(\norm{\Ric}^2g_{ij}-R\Ric_{ij}\right),
\]
where \(N\) is the Ricci reaction tensor from the preceding nonnegative Ricci
calculation.  At a null eigenvector of \(S\), diagonalize \(\Ric\) and write
\(\lambda_1=\del R\).  The formalized first-null reaction is

\[
  \del^2(1-3\del)R^2
  +(1-\del)(\lambda_2-\lambda_3)^2.
\]
This expression is nonnegative for \(0<\del<1/3\).

The algebraic stack for this first-null calculation is
\leanref{shiftScal3_eq}, \leanref{shiftNull3},
\leanref{pinchShiftNull_ge}, \leanref{shiftReact3_nonneg},
\leanref{shiftNAt}, \leanref{shiftNRaw}, and
\leanref{shiftNRaw_null_symm}.  Together with the parabolic and regularity
producers, it feeds \leanref{PinchFlowWMPData.ofShiftNClosed}, which supplies
the tensor-WMP data used by \leanref{pinch_sol_closed}.

The boundary value \(\del=1/3\) corresponds to a rigidity boundary, but that
endpoint is not needed for the preservation statement used in Hamilton's proof
and is not asserted here as a closed Lean result.
\end{proof}

\begin{corollary}[Strict positivity gives uniform pinching]
\label{cor:strict-positive-gives-pinching}
If $M^3$ is closed and $\Ric(g_0)>0$, then there exists $\del>0$, depending
only on $g_0$, such that
\[
\Ric(g_0)\ge\del R(g_0)g_0.
\]
Consequently, along the Ricci flow from $g_0$,
\[
\Ric(g(t))\ge\del R(g(t))g(t).
\]
\end{corollary}

\begin{proof}[Lean-facing calculation]
Compactness of \(M\) and continuity of the least Ricci-to-scalar ratio supply a
uniform initial constant.  Mathematically, since \(\Ric(g_0)>0\), the function
\[
x\mapsto \min_{\norm{v}_{g_0}=1}\frac{\Ric_{g_0}(v,v)}{R(g_0)}
\]
is positive and continuous on the compact manifold \(M\).  Its minimum is
therefore positive.  Reducing this minimum if necessary gives a constant
\(\del\) with \(0<\del<1/3\) and
\[
\Ric(g_0)\ge \del R(g_0)g_0.
\]
The shifted pinching preservation theorem then gives the asserted inequality
for all later times on which the flow exists.

In Lean, the initial pinching selector is represented by the
\leanref{PinchInitLt} family.  The wrappers \leanref{pinch_init_wmp_lt} and
\leanref{strict_pinch_sol_lt} package the selected strict constant and feed it
to the shifted WMP endpoint.
\end{proof}

\subsection{The improved Ricci pinching quantity}

Hamilton's preserved lower pinching inequality \(\Ric \ge \del Rg\) is not
by itself enough to force the blow-up limit to be Einstein.  The improved
pinching estimate supplies the missing decay.  It bounds the scale-sensitive
quantity
\[
  \frac{\norm{\Ric^\circ}^2}{R^2}
  \le C R^{-\eps},
\]
so that after parabolic rescaling at points where \(R\to\infty\), the
trace-free Ricci tensor disappears in the limit \cite{MR664497}.

\subsubsection{Definition and evolution of the pinching quantity}

\begin{definition}[Trace-free Ricci tensor]
\label{def:tracefree-ricci}
In dimension $3$, define
\[
\Ric^\circ
=
\Ric-\frac13Rg.
\]
Then
\[
\norm{\Ric^\circ}^2
=
\norm{\Ric}^2-\frac13R^2.
\]
\end{definition}

\begin{definition}[Hamilton's improved pinching quantity]
\label{def:pinching-P}
Assume $R>0$ and fix $0<\eps<1$. Define
\[
P
=
\frac{\norm{\Ric}^2-\frac13R^2}{R^{2-\eps}}
=
\frac{\norm{\Ric^\circ}^2}{R^{2-\eps}}.
\]
\end{definition}

The quantity \(P\) is scale-invariant only in the limiting case
\(\eps=0\).  The nonzero exponent is intentional: it is precisely what
produces the decay factor \(R_i^{-\eps}\) after parabolic blow-up at points
where the scalar curvature scale \(R_i\) tends to infinity.

\begin{definition}[The quartic reaction polynomial $Q$]
\label{def:Q}
In dimension $3$, define
\[
Q
:=
2\norm{\Ric}^4
+R^4
-5R^2\norm{\Ric}^2
+4R\tr(\Ric^3),
\]
where
\[
\tr(\Ric^3)=\Ric_i{}^j\Ric_j{}^k\Ric_k{}^i.
\]
Every term is homogeneous of degree four in the Ricci eigenvalues.  Some Lean
identifiers retain the historical word \texttt{Cubic}; that name refers to the
\(\tr(\Ric^3)\) contribution rather than to the degree of $Q$.

\end{definition}

\begin{lemma}[Evolution of the trace-free Ricci norm]
\label{lem:evol-tracefree-ricci-norm}
In dimension $3$, along Ricci flow and wherever $R>0$,
\[
(\partial_t-\Delta)\norm{\Ric^\circ}^2
=
-2\norm{\nabla\Ric}^2
+\frac23\norm{\nabla R}^2
+
\frac{4\norm{\Ric}^2\norm{\Ric^\circ}^2-2Q}{R}.
\]
\end{lemma}

\begin{proof}[Lean-facing calculation]
The textbook-facing endpoint is \leanref{trace_free_ricci_norm_sq_heat_equation_of_smooth_solution} in
\rffile{Geometry/Flow/RicciFlow/Preservation/Pinching/SolutionEvolution.lean}; the
smooth-solution wrapper is \leanref{trace_free_ricci_norm_sq_heat_equation_of_solution}.  The canonical scalar field
for the trace-free Ricci norm is \leanref{traceFreeRicciNormSq}.  The proof begins with
\[
\norm{\Ric^\circ}^2=\norm{\Ric}^2-\frac13R^2.
\]
The Ricci-norm evolution and scalar-curvature evolution from the previous
sections provide the heat terms.  Differentiating the displayed identity gives
the gradient contribution
\[
-2\norm{\nabla\Ric}^2+\frac23\norm{\nabla R}^2.
\]
The remaining reaction term is then rewritten using the three-dimensional
Riemann-from-Ricci and eigenvalue algebra.  In the textbook-facing form this
reaction is exactly
\[
\frac{4\norm{\Ric}^2\norm{\Ric^\circ}^2-2Q}{R}.
\]
\end{proof}

\begin{lemma}[Quotient evolution identity]
\label{lem:quotient-evolution}
Let $\varphi$ and $\psi$ be smooth space--time functions with $\psi>0$, and work
on a region where $\varphi>0$. For constants $\alpha,\beta$, one has
\begin{align*}
(\partial_t-\Delta)\left(\frac{\varphi^\alpha}{\psi^\beta}\right)
&=
\alpha\frac{\varphi^{\alpha-1}}{\psi^\beta}(\partial_t-\Delta)\varphi
-
\beta\frac{\varphi^\alpha}{\psi^{\beta+1}}(\partial_t-\Delta)\psi
\\
&\quad
-
\alpha(\alpha-1)\frac{\varphi^{\alpha-2}}{\psi^\beta}\norm{\nabla\varphi}^2
-
\beta(\beta+1)\frac{\varphi^\alpha}{\psi^{\beta+2}}\norm{\nabla\psi}^2
\\
&\quad
+
2\alpha\beta\frac{\varphi^{\alpha-1}}{\psi^{\beta+1}}
\ip{\nabla\varphi}{\nabla\psi}.
\end{align*}
In the special case \(\alpha=1\) used below, the formal theorem only needs
\(\varphi\ge0\) instead of \(\varphi>0\); the
\(\norm{\nabla\varphi}^2\) term vanishes.
\end{lemma}

\begin{proof}[Lean-facing calculation, with scope note]
The quotient calculus layer is project-local in
\rffile{Geometry/Flow/RicciFlow/Preservation/Pinching/QuotientEvolution.lean}.
The theorem \leanref{quotHeat} proves the positive-region
arbitrary-exponent identity.  The Hamilton-ready theorem
\leanref{quotHeat1_of_nonneg} is the \(\alpha=1\) specialization, where the
numerator only has to be nonnegative because
\[
  \alpha(\alpha-1)\varphi^{\alpha-2}\norm{\nabla\varphi}^2=0
  \qquad (\alpha=1).
\]
Thus the later pinching proof does not rely on an unproved arbitrary-exponent
nonnegative-numerator quotient theorem.

The proof differentiates
\(F(\varphi,\psi)=\varphi^\alpha\psi^{-\beta}\) and uses the product and chain
rules for the realized parabolic operator:
\[
\Delta(F(\varphi,\psi))
=F_\varphi\Delta\varphi+F_\psi\Delta\psi
+F_{\varphi\varphi}\norm{\nabla\varphi}^2
+2F_{\varphi\psi}\ip{\nabla\varphi}{\nabla\psi}
+F_{\psi\psi}\norm{\nabla\psi}^2.
\]
The theorem \leanref{quotHeatDiv} is the display-form wrapper that rewrites the
result in the usual quotient notation.
\end{proof}

\begin{lemma}[Evolution of Hamilton's pinching quantity]
\label{lem:evol-pinching-P}
Let $g(t)$ be a Ricci flow on a closed three-manifold, and assume $R>0$. For
$0<\eps<1$, the quantity
\[
P=\frac{\norm{\Ric^\circ}^2}{R^{2-\eps}}
\]
satisfies
\begin{align}
\partial_tP
&=
\Delta P
+
\frac{2(1-\eps)}{R}\ip{\nabla R}{\nabla P}
\nonumber\\
&\quad
-
\frac{2}{R^{4-\eps}}
\norm{R\nabla\Ric-\nabla R\otimes\Ric}^2
-
\frac{\eps(1-\eps)}{R^{4-\eps}}
\norm{\Ric^\circ}^2\norm{\nabla R}^2
\nonumber\\
&\quad
+
\frac{2}{R^{3-\eps}}
\left(
\eps\norm{\Ric}^2\norm{\Ric^\circ}^2-Q
\right).
\label{eq:evol-pinching-P}
\end{align}
\end{lemma}

\begin{proof}[Lean-facing calculation]
The textbook-facing wrapper is \leanref{pinch_quotient_evolution_of_heat_equations}.  The solution-level
assembly theorem \leanref{pinch_quotient_evolution_of_solution_data} supplies the raw quotient setup,
quotient regularity, and positivity data from the solution package, and
\leanref{pinch_quotient_evolution_of_tensor_sections} removes manual Ricci-section inputs.  The calculation
applies the \(\alpha=1\) quotient identity with
\[
  \varphi=\norm{\Ric^\circ}^2,
  \qquad
  \psi=R,
  \qquad
  \beta=2-\eps.
\]
The heat equations for \(\varphi\) and \(\psi\) are the trace-free Ricci norm
evolution and the scalar-curvature evolution.

The gradient terms are then completed into the square
\[
  \norm{R\nabla\Ric-\nabla R\otimes\Ric}^2.
\]
In Lean this square is represented by \leanref{ricciGradCoupleSq}.  The mixed
term needed for the square completion is related to the gradients of
\(\norm{\Ric}^2\) and \(\norm{\Ric^\circ}^2\) by
\leanref{ricciMixed_eq_gradNorm} and \leanref{ricciMixed_eq_tfGrad}.  After
these gradient rewrites, the remaining zero-order reaction is the last line of
\eqref{eq:evol-pinching-P}.  This is a quotient-calculus and tensor-algebra
producer, not an additional maximum-principle argument.
\end{proof}

\subsubsection{The algebraic estimate for \texorpdfstring{$Q$}{Q}}

\begin{lemma}[Factorization of \texorpdfstring{$Q$}{Q} in Ricci eigenvalues]
\label{lem:Q-factorization}
Let $\lambda_1,\lambda_2,\lambda_3$ be the eigenvalues of $\Ric$ at a point,
and set $R=\lambda_1+\lambda_2+\lambda_3$. Then

\[
Q
=
\sum_{\substack{\{i,j,k\}=\{1,2,3\}\\ i<j}}
(\lambda_i-\lambda_j)^2(R-2\lambda_k)^2.
\]
Equivalently,
\begin{align*}
Q
&=
(\lambda_1-\lambda_2)^2(R-2\lambda_3)^2
+(\lambda_1-\lambda_3)^2(R-2\lambda_2)^2
\\
&\quad
+(\lambda_2-\lambda_3)^2(R-2\lambda_1)^2.
\end{align*}
\end{lemma}

\begin{proof}[Lean-facing algebra]
The factorization is \leanref{hamiltonCubicQ3_factorized} in
\rffile{Geometry/Curvature/DimensionThree/PinchingAlgebra.lean}.  Lean defines both
\leanref{hamiltonCubicQ3} and its factorized form
\leanref{hamiltonCubicQFactorized3}.  After expanding in the three Ricci
eigenvalues, the identity is a finite three-variable polynomial identity:
\[
Q=2\norm{\Ric}^4+R^4-5R^2\norm{\Ric}^2+4R\tr(\Ric^3),
\]
where
\[
\norm{\Ric}^2=\sum_i\lambda_i^2,
\qquad
\tr(\Ric^3)=\sum_i\lambda_i^3.
\]
\end{proof}

\begin{lemma}[Algebraic lower bound for $Q$]
\label{lem:Q-lower-bound}
Suppose $R>0$ and
\[
\Ric\ge\del Rg
\]
for some $\del>0$. Then
\[
Q\ge 2\del^2\norm{\Ric}^2\norm{\Ric^\circ}^2.
\]
\end{lemma}

\begin{proof}[Lean-facing algebra]
The lower bound is proved in
\rffile{Geometry/Curvature/DimensionThree/PinchingAlgebra.lean}.  The ordered-eigenvalue theorem
is \leanref{hamiltonCubicQ3_lower_bound_ordered_nonnegative_eigenvalues}.  The
Hamilton-ready nonnegative reaction form is \leanref{PinchEigen3.q_sub_nonneg},
and the unordered version is \leanref{PinchEigen3Unordered.q_sub_nonneg}.

Order the eigenvalues as
\[
  \lambda_1\ge\lambda_2\ge\lambda_3
\]
and set
\[
  z=\lambda_3,
  \qquad
  b=\lambda_2-\lambda_3,
  \qquad
  a=\lambda_1-\lambda_2.
\]
Then \(a,b,z\ge0\),
\[
  \lambda_1=z+b+a,
  \qquad
  \lambda_2=z+b,
  \qquad
  \lambda_3=z.
\]
Using the factorization of \(Q\), Lean proves the polynomial identity
\begin{align*}
&Q-z^2\left(a^2+(a+b)^2+b^2\right) \\
&\qquad
=
2a^2\left(a^2+3ab+2az+3b^2+4bz\right)
\ge0.
\end{align*}
Thus
\[
  Q\ge z^2\left(a^2+(a+b)^2+b^2\right).
\]
Since
\[
\norm{\Ric^\circ}^2
=
\frac13\left(
(\lambda_1-\lambda_2)^2
+(\lambda_1-\lambda_3)^2
+(\lambda_2-\lambda_3)^2
\right),
\]
this gives \(Q\ge3z^2\norm{\Ric^\circ}^2\).  The pinching hypothesis implies
\(z\ge\del R\), and nonnegativity of the eigenvalues gives
\(\norm{\Ric}^2\le R^2\).  Therefore
\[
Q
\ge
3\del^2 R^2\norm{\Ric^\circ}^2
\ge
2\del^2\norm{\Ric}^2\norm{\Ric^\circ}^2.
\]
The flow-facing bridge \leanref{cubicQ_pinchOn} translates this eigenvalue
inequality into the \(cubicQAt\) notation used by the pinching evolution.
\end{proof}

\begin{corollary}[Improved Ricci pinching estimate]
\label{cor:improved-ricci-pinching}
Let $g(t)$ be a Ricci flow on a closed three-manifold with $\Ric(g_0)>0$.
Then there exist constants $\eps>0$ and $C<\infty$, depending only on $g_0$,
such that
\[
\frac{\norm{\Ric^\circ}^2}{R^{2-\eps}}\le C
\]
for all times for which the flow exists.
Equivalently,
\[
\frac{\norm{\Ric^\circ}^2}{R^2}\le CR^{-\eps}.
\]
\end{corollary}

\begin{proof}[Lean-facing calculation]
The endpoint file is
\rffile{Geometry/Flow/RicciFlow/Preservation/Pinching/Estimate.lean}.  The final solution-level theorem
is \leanref{exists_pinching_estimate_of_smooth_solution}.  It is a domain-aware scalar-WMP consumer over
the evolution identities and three-dimensional algebra above.  The Hamilton
endpoint packages the result as \leanref{hamiltonPinchingEstimate}; this is
produced by \leanref{hamilton_pinching_implies_pinch_estimate}.

By Corollary~\ref{cor:strict-positive-gives-pinching}, there exists
\(\del>0\), depending only on \(g_0\), such that
\[
\Ric(g(t))\ge\del R(g(t))g(t)
\]
for all times.  Choose \(0<\eps\le2\del^2\); in Lean this is implemented as
\[
  \eps=\min\{1/2,\del^2\},
\]
which gives \(0<\eps<1\) and \(\eps\le2\del^2\).  The algebraic lower bound
for \(Q\) then makes the zero-order term in \eqref{eq:evol-pinching-P}
nonpositive:
\[
  \eps\norm{\Ric}^2\norm{\Ric^\circ}^2-Q\le0.
\]
The two completed-square gradient terms in \eqref{eq:evol-pinching-P} are also
nonpositive.  Hence \(P\) satisfies the drifted scalar subsolution inequality
\[
\partial_tP
\le
\Delta P+
\frac{2(1-\eps)}{R}\ip{\nabla R}{\nabla P}.
\]

The Lean proof turns this into a maximum-principle estimate in several checked
consumer steps.  The theorem \leanref{pinchQuotient_parabolic_nonpos} proves
the drifted inequality for the quotient on each compact slab.  The theorem
\leanref{pinchQuot_slab_bound} applies the scalar weak maximum principle to
\(C-P\), using the sign-change identity \leanref{parabolic_const_sub}.
The initial bound is \leanref{pinchQuotient_initial_bound}.  The positivity and
regularity inputs for the quotient are \leanref{pinchQuotient_space_pos},
\leanref{pinchQuotient_grad_pos}, and the slab-continuity producer
\leanref{ricciNorm_slabCont}.  The resulting estimate is
\[
  P(\cdot,t)\le \sup_M P(\cdot,0)=:C,
\]
which is the displayed improved pinching inequality.
\end{proof}

\section{Maximal flow, finite-time singularity, and the blow-up sequence}
\label{sec:maximal-flow-blowup}

\subsection{Candidate segments, supremum gluing, and the maximal endpoint}

The formal maximal-flow argument has two logically separate layers.  First,
\leanref{exists_max_flow} constructs a finite maximal solution by taking the
supremum of the lifetimes of all compatible solution segments starting from one
fixed metric.  This construction uses the one-metric short-time theorem,
forward uniqueness, and a scalar-curvature upper bound for every candidate
lifetime.    Second,
\leanref{extends_of_rmBounded} uses class-uniform existence and an interior
restart to show that bounded curvature would extend the maximal solution.  The
second layer is what converts maximality into curvature unboundedness.

Fix a smooth three-dimensional metric \(g_0\) with positive scalar curvature.
For each \(T>0\), the structure \leanref{FlowTo} records a candidate segment on
\([0,T)\).  Besides a \leanref{SolutionOn} object and a proof of
\leanref{IsSolutionOn}, it stores the initial identity, joint chartwise
space--time smoothness on the closed-open interval, and the Ricci-flow equation
as a derivative within \(\operatorname{Ici}(0)\):
\begin{Verbatim}[breaklines=true,breakanywhere=true,fontsize=\small]
structure FlowTo (g0 : SmoothRiemannianMetric I M) (T : Real) where
  hT : 0 < T
  S : SolutionOn (I := I) (M := M)
    (RealTimeInterval.closedOpen 0 T hT)
  isSol : IsSolutionOn (I := I) S
  start : S.family.metric 0 = g0
  joint : forall (x0 : M)
      (i j : Fin (Module.finrank Real E)),
    ContMDiffOn
      ((modelWithCornersSelf Real Real).prod I)
      (modelWithCornersSelf Real Real) top
      (fun p : Real × M =>
        Integral.Measure.chartGramMatrix
          (I := I) (S.family.metric p.1) x0 p.2 i j)
      (Set.prod
        (Set.Ico 0 T)
        (trivializationAt E
          (TangentSpace I) x0).baseSet)
  pde : forall t,
    Membership.mem t (Set.Ico 0 T) ->
    forall x : M,
    forall v w : TangentSpace I x,
    HasDerivWithinAt
      (fun s : Real =>
        (S.family.metric s).inner x v w)
      ((-2 : Real) * ricciTensor
        (I := I) (S.family.metric t) x v w)
      (Set.Ici 0) t
\end{Verbatim}
The namespace parameters are shown because the interval and manifold are part
of the dependent solution type.  The two final fields are retained because
forward uniqueness and the later interval-changing arguments consume precisely
this one-sided initial-time regularity.  For a candidate \(P\), write
\(g_P(t):=\mathtt{P.S.family.metric}\,t\).

\begin{proposition}[Uniform bound for every candidate lifetime]
\label{prop:candidate-lifetime-bound}
Let \(M^3\) be closed and let \(g_0\) have positive scalar curvature.  Set
\[
  c_0:=\min_M R(g_0)>0.
\]
If \(P\) is any candidate in \textnormal{\leanref{FlowTo}}\((g_0,T)\), then
\[
  T\le \frac{3}{2c_0}.
\]
No maximality or curvature-blow-up assumption is used.
\end{proposition}

\begin{proof}[Lean-facing route]
This is \leanref{flow_end_le} in
\rffile{Geometry/Flow/RicciFlow/Estimates/FiniteTime/Solution.lean}.
The initial identity \(g_P(0)=g_0\) transports the fixed minimum \(c_0\) to the
candidate solution.  From \leanref{IsSolutionOn}, the proof obtains smoothness,
the scalar evolution equation
\[
  \partial_tR=\Delta R+2|\Ric|^2,
\]
and the three-dimensional trace estimate
\[
  |\Ric|^2\ge \frac13R^2.
\]
The scalar weak maximum principle compares \(R\) with
\[
  c(t)=\frac{c_0}{1-\frac23c_0t}.
\]
The pole of this barrier is \(3/(2c_0)\), and the theorem
\leanref{finiteTime3D} yields the stated bound.  Since \(c_0\) depends only on
the fixed initial metric, the same upper bound applies to every candidate
segment.  Thus the endpoint set used below is bounded before any maximal object
has been constructed.
\end{proof}

The seed theorem \leanref{flow_to_seed} applies the one-metric short-time
producer \leanref{ricci_flow_short_time_existence} to obtain at least one candidate.  The
compatibility theorem \leanref{flow_to_agree} applies
\leanref{ricci_flow_forward_unique}: if \(P\) and \(Q\) have endpoints \(T\)
and \(U\), then
\[
  g_P(t)=g_Q(t)
  \qquad\text{for }0\le t<\min\{T,U\},
\]
where \(g_P\) and \(g_Q\) denote the two candidates' metric families.
The equality here is equality of the metric families.  Since the connection,
Ricci tensor, and the other geometric fields in \leanref{SolutionOn} are
canonically derived from the metric, this is the compatibility needed for the
gluing construction.

\begin{theorem}[Finite maximal flow by supremum and local reconstruction]
\label{bb:maximal-rf-interval}
Let \(M^3\) be closed and nonempty, and let \(g_0\) have positive scalar
curvature.  Then there exist \(\omega>0\) and a Ricci-flow solution
\[
  g_{\max}(t),\qquad 0\le t<\omega,
\]
with \(g_{\max}(0)=g_0\), such that
\textnormal{\leanref{IsMaximalAtEndpoint}} holds: there is no solution on
\([0,\omega+\varepsilon)\), \(\varepsilon>0\), agreeing with it on
\([0,\omega)\).
\end{theorem}

\begin{proof}(Actual construction in \textnormal{\leanref{exists_max_flow}}.)
Define
\[
  \mathcal T
  :=\{T\in\mathbb R:\operatorname{Nonempty}(\mathtt{FlowTo}(g_0,T))\}.
\]
The seed theorem makes \(\mathcal T\) nonempty, and
Proposition~\ref{prop:candidate-lifetime-bound} makes it bounded above.  Lean
sets
\[
  \omega:=\sup\mathcal T.
\]
A seed endpoint \(T_0\in\mathcal T\) satisfies \(0<T_0\le\omega\), so
\(\omega>0\).

For each \(t\in[0,\omega)\), the defining property of the supremum gives a
candidate with endpoint strictly larger than \(t\).  The private structure
\leanref{FlowCover} packages such an endpoint \(T_t\), a candidate \(P_t\), and
the inequality \(t<T_t\).  After choosing one cover at every time, the metric
family is defined pointwise by
\[
  g_{\max}(t):=g_{P_t}(t)
  \qquad (0\le t<\omega),
\]
and is assigned the harmless value \(g_0\) outside this interval so that it is
a total Lean function.  The key comparison identity in the proof is the local statement
\texttt{hmax\_eq}: whenever \(Q\) is any candidate with endpoint \(U\) and \(t<U\),
\[
  g_{\max}(t)=g_Q(t).
\]
This follows from \leanref{flow_to_eq}, hence ultimately from forward
uniqueness, and makes the pointwise definition independent of every choice.

The proof does not take a union of the dependent \leanref{SolutionOn}
structures.  It glues only the metric family and reconstructs the solution
package locally.  Given a space--time point \((t,x)\) with \(t<\omega\), choose
the cover \(P_t\) with endpoint \(T_t>t\).  On the open time neighborhood
\(( -\infty,T_t)\), intersected with the maximal domain, \texttt{hmax\_eq}
identifies \(g_{\max}\) with this single candidate.  The candidate's
\leanref{FlowTo.joint} field therefore transfers joint chartwise smoothness to the
glued family.  Similarly, eventual equality in
\(\operatorname{nhdsWithin}(t,\operatorname{Ici}(0))\) transports the
candidate's \texttt{HasDerivWithinAt} statement and proves the Ricci-flow
PDE for \(g_{\max}\).  The constructor \leanref{solutionOn_of_joint} then
rebuilds \leanref{IsSolutionOn} and all of its derived geometric fields from
these metric-level facts.

It remains to prove maximality.  Suppose the resulting solution extended past
\(\omega\).  The theorem \leanref{flow_to_extend} repackages that extension as
a new \leanref{FlowTo} candidate with endpoint \(\omega+\varepsilon\).  In
this repackaging, times below \(\omega\) use the old candidate's one-sided
joint regularity and PDE, while times at or beyond \(\omega\) are positive and
therefore use the extended solution's interior regularity.  Agreement on the
old interval transports the relevant metric identities.  Hence
\(\omega+\varepsilon\in\mathcal T\), contradicting
\(\omega=\sup\mathcal T\) and \(\varepsilon>0\).  This proves
\leanref{IsMaximalAtEndpoint}.
\end{proof}

\begin{remark}[Why two-sided smoothness is the main difficulty]
The main difficulty is not the order-theoretic step of taking the supremum of
candidate lifetimes.  Short-time existence from a metric prescribed at time
\(s\) naturally gives a flow on a forward interval
\[
  [s,s+\varepsilon),
\]
with one-sided time regularity at \(s\).  This is the correct behavior at the
initial time \(0\), but an interior time \(s\in(0,\omega)\) of the maximal flow
must have a single smooth representative on a two-sided neighborhood
\[
  (s-\delta,s+\delta).
\]
Constructing unrelated forward solutions at individual times would not by
itself provide such a representative, and a direct piecewise definition would
leave a seam whose compatibility had not been proved.

The supremum construction avoids any explicit matching of time jets.  Given
\(s\in(0,\omega)\), the inequality \(s<\sup\mathcal T\) provides one candidate
with endpoint \(T>s\).  Choosing \(\delta>0\) with
\[
  (s-\delta,s+\delta)\subset(0,T)
\]
shows that this candidate is already smooth on the entire required
neighborhood.  Forward uniqueness, encoded by the local identity \texttt{hmax\_eq}, identifies
the pointwise glued family with that candidate throughout the neighborhood.
Thus two-sided smoothness comes from local identification with one segment
that extends beyond \(s\), not from separately gluing forward germs at \(s\).
At \(s=0\), the same argument deliberately retains only the one-sided
regularity stored in \leanref{FlowTo}.

The same asymmetry is decisive in the bounded-curvature continuation argument.
One restarts at a time \(t_*<\omega\), not at \(\omega\).  The old solution
supplies the past-hand flow, the restarted solution supplies the future-hand
flow, and forward uniqueness identifies them on the nonempty overlap
\([t_*,\omega)\).  Because the restarted lifetime crosses \(\omega\), the time
\(\omega\) is an interior positive time of the restarted solution and hence
has a genuine two-sided smooth neighborhood.  The final gluing theorem uses
this local restarted representation around \(\omega\); it does not prove
smoothness by matching derivatives across an isolated endpoint.  
\end{remark}

\subsection{Finite maximal time, bounded-curvature extension, and curvature blow-up}

The finite value of \(\omega\) has already been built into
Theorem~\ref{bb:maximal-rf-interval}: every candidate endpoint was bounded by
\(3/(2c_0)\) before the supremum was taken.  One
must also prove that a maximal solution cannot have bounded curvature.  This is the
point at which the class-uniform short-time theorem enters.

The interval-changing predicates are defined in
\rffile{Geometry/Flow/RicciFlow/Extension/MaximalTime.lean}.  The predicate
\leanref{ExtendsPastEndpoint} supplies \(\varepsilon>0\), a solution on
\([\alpha,\omega+\varepsilon)\), and agreement with the original solution on
\([\alpha,\omega)\).  The agreement package \leanref{SolutionAgreesOn}
records equality of the metric, connection, and Ricci fields.  Finally,
\leanref{IsMaximalAtEndpoint} is the negation of this extension predicate.

\begin{theorem}[Three-dimensional bounded-curvature extension criterion]
\label{bb:rf-extension-criterion}
Let \(S\) be a Ricci flow on a closed three-manifold over a finite interval
\([\alpha,\omega)\).  Suppose \(\mathcal R(t)\) is a time-dependent
lowered curvature field realizing the solution connection and
\[
  |\mathcal R(t)|^2_{g(t),x}\le K
  \qquad
  (\alpha\le t<\omega,\ x\in M)
\]
for one constant \(K\).  Then \textup{\leanref{ExtendsPastEndpoint}} holds: for some
\(\varepsilon>0\), \(S\) extends to a smooth Ricci flow on
\([\alpha,\omega+\varepsilon)\).
\end{theorem}
In the Hamilton application one takes \(\mathcal R=S.\mathtt{base.rm04}\),
whose realization property is \textup{\leanref{rm04Realizes_metric}}.

\begin{proof}[Actual route in \textup{\leanref{extends_of_rmBounded}}]
The proof has four stages.

\smallskip
\noindent\emph{1. Produce one controlled metric class near the endpoint.}
The squared-curvature bound and the realization theorem first give a quadratic
Ricci bound by \leanref{ric_quad_le_of_soln}.  Integrating
\[
  \partial_t\log g(t)(v,v)
  =-2\frac{\Ric_{g(t)}(v,v)}{g(t)(v,v)}
\]
then yields a uniform equivalence constant \(\Lambda_1\) between the tail
metrics and the fixed background \(g(\alpha)\); this is
\leanref{hell_of_soln}, based on \leanref{metricEquiv_of_ricBound}.  The same
curvature bound, together with the moving Shi estimates, gives a second
constant \(\Lambda_2\) and background-covariant derivative bounds through
order three,
\[
  |\nabla_{g(\alpha)}^{a}g(t)|\le \Lambda_2,
  \qquad 0\le a\le3,
\]
on a possibly shorter tail.  This is \leanref{shiCovBound_of_soln}.  The
producer \leanref{extendInputs_of_soln} packages these two conclusions.

\smallskip
\noindent\emph{2. Choose the common lifetime before choosing the restart
metric.}
Apply \leanref{ricci_flow_uniform_existence} with the fixed background
\(g(\alpha)\) and the single class constant
\[
  \Lambda:=\max\{\Lambda_1,\Lambda_2\}.
\]
It produces \(\tau_0>0\) before the restart metric varies in the controlled
class.  Only after \(\tau_0\) is fixed does
\leanref{ricci_flow_interior_restart} choose
\[
  t_*=
  \max\!\left\{
    \max\{t_1,t_2\},
    \max\!\left\{\omega-\frac{\tau_0}{2},
                  \frac{\alpha+\omega}{2}\right\}
  \right\}.
\]
Consequently
\[
  \alpha<t_*<\omega,
  \qquad
  \omega<t_*+\tau_0,
\]
and \(g(t_*)\) lies in the class for which the common lifetime was selected.
The uniform theorem therefore supplies a restarted Ricci flow \(r(u)\) on
\([0,\tau_0)\) with \(r(0)=g(t_*)\).

\smallskip
\noindent\emph{3. Identify the old and restarted flows on their overlap.}
Shift the restart to original time by \(u=t-t_*\).  The original family and
\(r(t-t_*)\) are both Ricci flows on \([t_*,\omega)\), have the same value at
\(t_*\), and satisfy the chartwise smoothness, continuity, and one-sided PDE
hypotheses required by \leanref{ricci_flow_forward_unique}.  Hence
\[
  r(t-t_*)=g(t)
  \qquad (t_*\le t<\omega).
\]
This overlap equality is the compatibility input for the gluing step.

\smallskip
\noindent\emph{4. Glue past the endpoint and reconstruct the solution package.}
Set
\[
  \varepsilon:=t_*+\tau_0-\omega>0
\]
and define, through \leanref{gluedFamily},
\[
  g_{\mathrm{ext}}(t)=
  \begin{cases}
    g(t),&t<\omega,\\
    r(t-t_*),&t\ge\omega.
  \end{cases}
\]
The formula is smooth near \(\omega\) because, on the whole overlap
\([t_*,\omega)\), both branches equal the same restarted flow.  The theorem
\leanref{extend_construction_of_restart} proves joint smoothness on the open
interior, continuity up to the left endpoint, and the Ricci-flow PDE on the
larger closed-open interval.  The constructor
\leanref{isSolutionOn_of_extendData} then builds the extended
\leanref{SolutionOn} object, and equality of the metric transports equality of
the derived connection and Ricci fields.  This gives
\leanref{ExtendsPastEndpoint}.
\end{proof}

\begin{lemma}[Curvature unboundedness at the maximal endpoint]
\label{lem:finite-time-curvature-blow-up}
Let \(S\) be the maximal flow produced above.  Then
\[
  \forall K\in\mathbb R\;\exists t\in[0,\omega)\;\exists x\in M,
  \qquad K<|\Rm(S(t))|_x^2.
\]
\end{lemma}

\begin{proof}[Lean-facing assembly]
This is \leanref{rmUnbounded_of_maximal}.  If the quantified conclusion failed,
\leanref{rmBounded_of_not_unbounded} would produce one global squared-curvature
bound.  Theorem~\ref{bb:rf-extension-criterion}, implemented by
\leanref{extends_of_rmBounded}, would then give an extension past \(\omega\),
contradicting \leanref{IsMaximalAtEndpoint}.  The formal conclusion is the
quantified unboundedness statement above; no limit or limsup formulation is
needed by the downstream package.
\end{proof}

The Hamilton-facing theorem \leanref{hamilton_finite_time_flow_exists_on_closed_open} follows this
exact order.  Positive Ricci curvature first gives positive scalar curvature;
\leanref{exists_max_flow} constructs \((\omega,S_{\max})\) and its maximality;
\leanref{rmUnbounded_of_maximal} then proves curvature unboundedness; and only
then is \leanref{HamiltonFiniteTimeFlow} assembled.  The final package stores the time
interval, the smooth solution, its initial value, and the quantified curvature
unboundedness.  It does not retain the intermediate maximality witness.

At the proof-term level, the relevant dependency split is
\begin{center}
\small
\begin{minipage}{0.94\linewidth}
\raggedright
\leanref{ricci_flow_short_time_existence} \(+\) \leanref{flow_end_le} \(+\)
\leanref{ricci_flow_forward_unique}\\
\hfill \(\longrightarrow\) \leanref{exists_max_flow};\\[3pt]
\leanref{extendInputs_of_soln} \(+\) \leanref{ricci_flow_uniform_existence}
\(+\) \leanref{ricci_flow_forward_unique}\\
\hfill \(+\) \leanref{extend_construction_of_restart}
\(\longrightarrow\) \leanref{extends_of_rmBounded};\\[3pt]
\leanref{exists_max_flow} \(+\) \leanref{extends_of_rmBounded}\\
\hfill \(\longrightarrow\) \leanref{rmUnbounded_of_maximal}
\(\longrightarrow\) \leanref{hamilton_finite_time_flow_exists_on_closed_open}.
\end{minipage}
\end{center}

\begin{corollary}[Nonnegative Ricci controls full curvature in dimension $3$]
\label{cor:ricci-controls-rm}
There exists a universal constant $C_3$ such that on any three-dimensional
Riemannian manifold with $\Ric\ge0$,
\[
\norm{\Rm}\le C_3R.
\]
In particular, one may take a coarse constant such as $C_3=100$.
\end{corollary}

\begin{proof}[Lean-facing algebra]
The pointwise dimension-three bound is project-local.  The scalar-blow-up proof
consumes it internally in \leanref{hamilton_scalar_blowup}; \leanref{hamilton_rescaled_curvature_bound} later
applies the related estimate to the rescaled flow.  The proof uses the lowered
curvature convention
\[
  \Rm_{04}(X,Y,Z,W)=\langle R(X,Y)Z,W\rangle
\]
and the first-trace route \leanref{normSqLeOfFirstTrace}.

At a point, diagonalize \(\Ric\) with eigenvalues
\(\lambda_1,\lambda_2,\lambda_3\ge0\).  By the three-dimensional
Riemann-from-Ricci identity, the principal sectional curvatures satisfy
\[
K_{ij}=\frac{\lambda_i+\lambda_j-\lambda_k}{2}.
\]
Hence
\[
\abs{K_{ij}}\le\frac{\lambda_i+\lambda_j+\lambda_k}{2}=\frac R2.
\]
Since \(\norm{\Rm}^2\) is a fixed universal multiple of the sum of the squares
of the three principal sectional curvatures in dimension three, this gives
\(\norm{\Rm}\le C_3R\) for a universal constant.  The Hamilton file uses the
coarse constant \(100\), which is more than sufficient for the blow-up
argument.
\end{proof}

\begin{corollary}[Scalar blow-up under nonnegative Ricci curvature]
\label{cor:finite-time-scalar-blow-up}
Let $g(t)$ be a maximal Ricci flow on a closed three-manifold with finite
maximal time $T_{\max}$.  If $\Ric(g(t))\ge0$ throughout the flow, then
\[
  \sup_{M\times[0,T_{\max})} R(g(t))=\infty.
\]
\end{corollary}

\begin{proof}[Lean-facing assembly]
If the scalar curvature were bounded above, then
Corollary~\ref{cor:ricci-controls-rm} would bound the full curvature, contrary
to Lemma~\ref{lem:finite-time-curvature-blow-up}.  The Hamilton-facing
consumer \leanref{hamilton_scalar_blowup} starts from the already-produced
\leanref{HamiltonFiniteTimeFlow.curvUnbounded} field and combines it with preserved
nonnegative Ricci curvature and the dimension-three curvature-control layer.
It does not rerun the extension argument.  Thus maximality plus extension
supplies full-curvature blow-up; the two additional ingredients are essential
for the scalar statement.
\end{proof}

\begin{definition}[Parabolic rescaling]
\label{def:parabolic-rescaling}
Let $g(t)$ be a Ricci flow defined on $[0,T)$. Given a point-time
$(x_i,t_i)$ with $0<t_i<T$ and a scale $R_i>0$, define the rescaled Ricci
flow
\[
g^{R_i}(s)=R_i\,g\left(t_i+\frac{s}{R_i}\right).
\]
It is defined for
\[
s\in[-R_it_i,R_i(T-t_i)).
\]
Negative values of the rescaled variable \(s\) correspond to original times
\(t_i+s/R_i<t_i\), which still lie in the original interval \([0,T)\).  They
do not assert that the original Ricci flow exists at negative time.
Under this rescaling,
\[
R(g^{R_i})(x,s)=R_i^{-1}R(g)\left(x,t_i+\frac{s}{R_i}\right),
\]
and, for a constant scaling $Qg$,
\[
\norm{\Rm}_{Qg}=Q^{-1}\norm{\Rm}_g,
\qquad
\norm{\Ric}_{Qg}=Q^{-1}\norm{\Ric}_g,
\qquad
R_{Qg}=Q^{-1}R_g.
\]
The scale-invariant ratio
\[
\frac{\norm{\Ric^\circ}^2}{R^2}
\]
is unchanged.  The standard change of time and constant metric scaling preserve
the Ricci-flow equation, so each \(g^{R_i}(s)\) is again a Ricci flow on its
rescaled time interval.
\end{definition}

\begin{lemma}[Point selection and scalar normalization]
\label{lem:point-selection-rescaling}
Let $g(t)$, $t\in[0,T_{\max})$, be the maximal Ricci flow from a closed
three-manifold with $\Ric(g_0)>0$, and suppose that the finite-time scalar
blow-up established above holds.  Then there exist points and times
$(x_i,t_i)$ and positive numbers
\[
R_i:=R(x_i,t_i)
\]
such that
\[
0<t_i<T_{\max},
\qquad
R_it_i\longrightarrow\infty.
\]
In particular, $R_i\to\infty$.  The rescaled flows $g^{R_i}(s)$ satisfy
\[
R(g^{R_i})(x_i,0)
=
\max_{M\times[-R_it_i,0]}R(g^{R_i})
=
1.
\]
\end{lemma}

\begin{proof}[Lean-facing assembly]
The point-selection theorem is \leanref{hamilton_exists_blowup_point_sequence}.  It takes finite-time
and scalar-blow-up packages as inputs.  In the Hamilton assembly,
\leanref{hamilton_extinction_time_bound} supplies the first and
\leanref{hamilton_scalar_blowup} the second.
Its conclusion is packaged by \leanref{HamiltonBlowup} and
\leanref{hamiltonBlowupPointSelection}.  Point selection is therefore recorded as a downstream
consumer.

Corollary~\ref{cor:finite-time-scalar-blow-up} gives
\(\sup_{M\times[0,T_{\max})}R=\infty\).  Let
\[
  m(t):=\max_{M\times[0,t]}R.
\]
Then \(m(t)\to\infty\) as \(t\nearrow T_{\max}\).  Choose levels
\(A_i\to\infty\), choose times \(t_i\) where the running maximum reaches
\(A_i\), and choose points \(x_i\) with
\[
  R(x_i,t_i)=m(t_i)=A_i.
\]
Set \(R_i=R(x_i,t_i)\).  The parabolic rescaling in
Definition~\ref{def:parabolic-rescaling} then gives
\[
  R(g^{R_i})(x_i,0)=1,
  \qquad
  \max_{M\times[-R_it_i,0]}R(g^{R_i})=1.
\]
Moreover, $t_i\to T_{\max}>0$: on every earlier compact time slab the smooth
curvature is bounded, so levels $A_i\to\infty$ cannot first be reached there.
Consequently $R_it_i\to\infty$.
\end{proof}

\section{Noncollapsing and Compactness}
\label{sec:global-interfaces}

The normalized blow-up sequence constructed above has scale-independent
curvature control on each fixed backward time slab.
Generally speaking, curvature control alone
does not prevent the sequence from collapsing to a lower-dimensional object.
The two companion developments described in this section rule out that
degeneration and then extract a smooth pointed limit.  We retain enough of the
mathematics to explain how they enter Hamilton's argument, while reserving the
full formal constructions for separate articles.

\subsection{No-local-collapsing and injectivity control}

Perelman's no-local-collapsing theorem supplies a scale-invariant lower volume
bound at every curvature-controlled scale \(0<r\le\rho\) \cite{Perelman1}.
The project's Hamilton interface uses a full \(\lvert\Rm\rvert\le r^{-2}\)
bound on a backward parabolic neighborhood.  Perelman's original finite-time
no-local-collapsing theorem is stronger in a different direction: its
local-collapsing formulation uses only a scalar curvature upper bound on the time-\(t\)
ball, without requiring the corresponding backward-parabolic bound.  In the
form used here, if a backward parabolic neighborhood of radius \(r\) satisfies

\[
  |\Rm|\le r^{-2},
\]
then the ball at its top time obeys
\[
  \Vol B(x,r)\ge \kappa r^3
\]
for one positive constant \(\kappa\) determined before the point and scale are
chosen.  This assertion is invariant under parabolic rescaling.  It therefore
applies to the point-selected flows \(g^{R_i}\), although their scaling factors
\(R_i\) diverge.

This is the precise place where noncollapsing enters the blow-up proof.  Choose
a fixed \(r_0>0\) small enough that the normalized curvature bound controls the
parabolic neighborhood of \((x_i,0)\) at scale \(r_0\).  Since
\(R_it_i\to\infty\), that neighborhood lies in the domain of every sufficiently
late rescaled flow.  No-local-collapsing gives

\[
  \Vol_{g^{R_i}(0)} B(x_i,r_0)\ge \kappa r_0^3.
\]

Combining this volume lower bound with the curvature bound on a slightly larger
ball yields a uniform positive lower bound for
\(\operatorname{inj}_{g^{R_i}(0)}(x_i)\).  Thus the sequence has a nondegenerate
normal-coordinate scale at its basepoints.  In the formal development,
\leanref{no_local_open} proves the original-flow theorem.  The proof of
\leanref{exists_hamilton_vol} applies it, transports the estimate through the
Hamilton rescalings, and selects one fixed radius and one fixed \(\kappa\) for
the whole canonical source sequence; \leanref{flowInj_of_vol} then performs the
curvature--volume injectivity step.

The proof of the companion noncollapsing theorem follows Perelman's entropy
strategy.  A closed-manifold Sobolev inequality and Jensen's inequality first
give a quantitative lower bound for the \(\mathcal W\)-functional.  The conjugate heat
potential supplies the monotonicity needed to transport this bound along the
flow.  If a curvature-controlled ball had arbitrarily small normalized volume,
a cutoff supported on that ball would produce a test density with \(\mathcal W\)-value
below the uniform lower bound, a contradiction.  The formal proof separates
the initial-time small-ball estimate from the positive-time entropy argument,
then joins them in the quantitative cutoff/contradiction theorem underlying
\leanref{no_local_open}.

\subsection{Metric and Ricci-flow compactness}

Once the basepoint injectivity radius is controlled, compactness is obtained in
two stages.  First, fix a finite backward interval \([-A,0]\).  The point-selection
estimate gives a uniform curvature bound there, and the project's moving
Bando--Shi estimates give bounds for every covariant derivative of curvature
on slightly smaller time slabs \cite{Bando1987RealAnalyticity,Shi1989a}.
Because the source manifolds here are closed, global Bando estimates suffice;
Shi's local estimates remain essential in noncompact or genuinely local
geometrization arguments.  At time zero the injectivity and curvature data
produce uniform normal-coordinate charts.  Volume comparison and packing
control how many charts are needed on each bounded region, while quantitative
transition estimates make the charts compatible.  Arzel\`a--Ascoli extraction
on a nested exhaustion, followed by a diagonal argument, produces a complete
pointed limiting metric.  This metric part of the construction is exported as
\leanref{metricCompactness}.

The second stage upgrades time-zero metric convergence to convergence of the
Ricci flows.  The all-order curvature bounds control the pulled-back metric
coefficients on compact space--time sets, and the Ricci-flow equation controls
their time derivatives.  Compatible subsequences on the nested space and time
windows are then assembled into one smooth limit flow.  Completeness is
transported to every limit time-slice, and the comparison embeddings realize
smooth pointed Cheeger--Gromov--Hamilton convergence.  The formal definition
uses embeddings on an exhaustion by compact domains, eventually in the
sequence; this is equivalent to the usual compact-open smooth pointed
formulation.  The complete flow
upgrade is exported as \leanref{compactnessSol}; its proof combines the project-local
normal-chart package, the metric compactness construction, the open-window Shi
estimates, and a canonical diagonal gluing of the local limits.  This is the
formal counterpart of the classical compactness theorem
\cite{HamiltonCompactness,MR2302600}.

The checked Hamilton assembly does not apply this construction on an increasing
family of time windows.  Its source sequence and flow upgrade instead live on
the fixed closed backward window \([-r_0^2,0]\).  Smooth pullback convergence at
the terminal slice supplies the scalar-curvature convergence and
trace-free Ricci decay needed by
\leanref{hamilton_constant_positive_sectional_curvature_of_pinching}, which
combines the fixed-radius volume data, injectivity control, metric compactness,
the closed-window flow upgrade, and the terminal curvature-transfer
conclusions.

\subsection{The exact statements used}

We now state the two endpoints that the final assembly consumes, first as
ordinary mathematical statements and then in the form in which they are proved
in Lean.  Stating them separately makes clear where the two companion projects
end and where the Hamilton-specific argument resumes.

\subsubsection{The noncollapsing endpoint}

Recall the setting of Lemma~\ref{lem:point-selection-rescaling}: $g(t)$,
$0\le t<\omega$, is the finite maximal flow on the closed connected
three-manifold $M$; the point-selected data are $(x_i,t_i)$ with
$R_i=R(x_i,t_i)$; and
\[
  g_i(s):=R_i\,g\!\left(t_i+\frac{s}{R_i}\right)
\]
are the basepoint-scalar-normalized rescalings.  Fix the window scale
$r_0=\tfrac1{10}$.  Preserved Ricci nonnegativity together with
Corollary~\ref{cor:ricci-controls-rm} gives the scale-independent bound
$\abs{\Rm_{g_i}}\le100$, and $R_it_i\to\infty$ places the backward window
$[-r_0^{2},0]$ inside the domain of $g_i$ for all large $i$.

\begin{theorem}[Uniform noncollapsing along the blow-up sequence]
\label{thm:ham3-volume-endpoint}
In the situation above there exist a radius $r\in(0,r_0]$ and a constant
$\kappa>0$, both independent of $i$, such that for every $i$ the following two
statements hold.
\begin{enumerate}[label=\textup{(\roman*)},leftmargin=2.4em]
  \item \emph{(Curvature control.)}  For every $s\in[-r^{2},0]$ and every
  $x\in M$ with $\operatorname{dist}_{g_i(s)}(x_i,x)<r$,
  \[
    r^{4}\,\abs{\Rm_{g_i}(s)}^{2}_{x}\le1 ;
  \]
  that is, $\abs{\Rm_{g_i}}\le r^{-2}$ on the backward parabolic neighbourhood
  of $(x_i,0)$ at scale $r$.
  \item \emph{(Noncollapsing.)}  The time-zero basepoint ball satisfies
  \[
    \Vol_{g_i(0)}B_{g_i(0)}(x_i,r)\ \ge\ \kappa\,r^{3}.
  \]
\end{enumerate}
\end{theorem}

In Lean the sequence $(M,g_i,x_i)$ is the object \leanref{hamiltonSourceSequence}, a
\leanref{PointedFlowSeq} whose common time interval is the closed backward
window $[-r_0^{2},0]$ and whose $i$-th term is the rescaled solution based at
$x_{N+i}$.  The index shift by $N$ is what turns the ``for all large $i$''
hypotheses into conclusions valid at \emph{every} index.  The pair
$(\kappa,r)$ together with its positivity is packaged as
\leanref{FlowerScaleVolData}, and conclusions (i) and (ii) are exactly the two
fields \leanrefas{IsFlowerScaleVolBound.curvature}{curvature} and
\leanrefas{IsFlowerScaleVolBound.noncollapsed}{noncollapsed} of
\leanref{IsFlowerScaleVolBound}, expressed through
\leanref{FlowMetricBall.IsRmControlled} and
\leanref{FlowMetricBall.IsKappaNoncollapsed}.  The exponent in (ii) is
$\operatorname{finrank}_{\mathbb R}E$, which the closed-three-manifold
hypothesis fixes to $3$.  The endpoint \leanref{exists_hamilton_vol} consumes
\leanref{hamiltonRiemannCurvatureBound}, \leanref{hamiltonWindow}, and the
radius \leanref{hamilton_reference_radius}:

\begin{Verbatim}[breaklines=true,breakanywhere=true,fontsize=\small]
theorem exists_hamilton_vol
    {omega : Real} (h0omega : 0 < omega)
    (hM : isClosedThreeManifold (I := I) (M := M))
    {g0 : SmoothRiemannianMetric I M}
    (P : HamiltonFiniteTimeFlow (I := I) (M := M) g0)
    (hD : P.D = RealTimeInterval.closedOpen 0 omega h0omega)
    (Q : HamiltonBlowup M)
    (hsel : hamiltonBlowupPointSelection (I := I) P Q)
    (hrm : hamiltonRiemannCurvatureBound (I := I) P Q)
    (hwindow : hamiltonWindow (I := I) P Q hamilton_reference_radius) :
    exists V : FlowerScaleVolData (I := I)
        (hamiltonSourceSequence (I := I) h0omega P hD Q hsel hwindow),
      IsFlowerScaleVolBound (I := I) V
\end{Verbatim}

The hypotheses \leanref{hamiltonRiemannCurvatureBound} and \leanref{hamiltonWindow} are precisely the
two scale-independent facts recalled above, and \leanref{hamilton_reference_radius} is the
constant $r_0=\tfrac1{10}$.  The proof applies \leanref{no_local_open} to the
original flow and transports its conclusion through the rescalings.

The next handoff is not a further noncollapsing hypothesis.  The
curvature--volume theorem \leanref{flowInj_of_vol} consumes the package above,
together with completeness and connectedness of the time-zero slices and the
time-zero bounded-geometry estimates, and returns a single $\iota_*>0$ with
\[
  \operatorname{inj}_{g_i(0)}(x_i)\ \ge\ \iota_*
  \qquad\text{for every }i .
\]
In Lean this conclusion is \leanref{FlowerScaleInjBound}, an abbreviation for
\leanref{BaseInjBound} applied to the time-zero sequence; its fields are one
radius $\rho>0$ and a proof of \leanref{HasInjRadiusAt} at every basepoint.
This is the exact interface required to enter metric compactness.

\subsubsection{The compactness endpoint}

\begin{theorem}[Smooth pointed Cheeger--Gromov--Hamilton compactness]
\label{thm:flow-compactness-endpoint}
Let $(M_i,g_i(t),x_i)$, $i\in\mathbb N$, be pointed Ricci flows on one common
open time interval $(\alpha,b)\ni0$, with every $M_i$ connected.  Assume
\begin{enumerate}[label=\textup{(\roman*)},leftmargin=2.4em]
  \item every time-slice $(M_i,g_i(t))$, $t\in(\alpha,b)$, is a complete metric
  space;
  \item the curvature is locally uniformly bounded: for every compact
  $[a,b']\subseteq(\alpha,b)$ there is $C\ge0$ such that
  $\abs{\Rm_{g_i}(t)}^{2}_{x}\le C$ for all $i$, all $t\in[a,b']$, and all
  $x\in M_i$;
  \item the basepoint injectivity radii are uniformly bounded below: there is
  $\iota>0$ with $\operatorname{inj}_{g_i(0)}(x_i)\ge\iota$ for every $i$.
\end{enumerate}
Then there are a pointed Ricci flow $(M_\infty,g_\infty(t),x_\infty)$ on the
same interval and a strictly increasing subsequence along which the flows
converge to it in the smooth pointed Cheeger--Gromov--Hamilton sense.
Moreover every time-slice of the limit is complete.
\end{theorem}

The Lean endpoint is \leanref{compactnessSol}.  Hypotheses (i)--(iii) are
\leanref{CompleteInput}, \leanref{CurvBoundInput}, and
\leanref{FlowerScaleInjBound}; the interval hypothesis says that \texttt{X.D} is
the \leanref{RealTimeInterval.openInterval} determined by $\alpha<0<b$.
The returned flow package and convergence predicate are
\leanref{PointedFlowData} and \leanref{SmoothCGHConverges}, respectively:

\begin{Verbatim}[breaklines=true,breakanywhere=true,fontsize=\small]
theorem compactnessSol
    {alpha b : Real} (h0 : (0 : Real) in Set.Ioo alpha b)
    (X : PointedFlowSeq (I := I))
    (hD : X.D = RealTimeInterval.openInterval alpha b 0 h0)
    (hcomplete : CompleteInput (I := I) X)
    (hcurv : CurvBoundInput (I := I) X)
    (hinj : FlowerScaleInjBound (I := I) X)
    (hconn : forall k : Nat,
      letI : TopologicalSpace (X.term k).M := (X.term k).topology
      ConnectedSpace (X.term k).M) :
    exists L : PointedFlowData (I := I) X.D,
      exists subseq : Nat -> Nat,
        StrictMono subseq /\
          Nonempty (SmoothCGHConverges (I := I) X L subseq) /\
            forall t : Real, t in X.D.carrier ->
              MetricComplete (I := I) (L.atTime (I := I) t)
\end{Verbatim}

The limit is a \leanref{PointedFlowData}: a manifold with a basepoint together
with a Ricci-flow solution on the same interval.  The conclusion therefore
produces a limit \emph{flow}, not merely a limit metric.  The mode of
convergence is \leanref{SmoothCGHConverges}, whose fields record pointed
Cheeger--Gromov convergence of the time-zero metrics, convergence of the
pulled-back scalar curvature and Ricci norm, and the space--time convergence
data on compact subsets.

\begin{remark}[What the Hamilton assembly actually calls]
\label{rem:closed-window-compactness}
Theorem~\ref{thm:flow-compactness-endpoint} is the general public statement,
formulated for an \emph{open} time window.  The Hamilton source sequence lives
on the \emph{closed} backward window $[-r_0^{2},0]$, so the assembly runs the
closed-window form of the same two-stage construction rather than
\leanref{compactnessSol} itself.  The time-zero metric stage is built by
\leanref{metricSeedOfBG} and \leanref{exists_bounded_geometry_normal_data} and packaged by
\leanref{MetricCompactSeed.higherRegularityCanonicalMetricCompactness} as a \leanref{CanonicalMetricCompactness}, whose
\texttt{mc} field is the \leanref{MetricCompactnessConclusion} carrying the
subsequence, the limiting pointed manifold, its completeness, and the
comparison maps.  The flow upgrade is then supplied by
\leanref{hamilton_flow_upgrade_of_metric_compactness}, which returns a \leanref{FlowUpgrade} over
that conclusion together with completeness of every limit time-slice.  This
closed-window construction supplies the compactness data needed by the
Hamilton endpoint; it is analogous to, but is not an invocation of, the public
open-window theorem \leanref{compactnessSol}.
\end{remark}

Once the limit flow is in hand, improved pinching transfers at the terminal
slice and forces
\[
  \Ric^\circ(g_\infty(0))\equiv0 .
\]
The basepoint scalar normalization and the contracted Bianchi identity then
show that the terminal limit metric is Einstein with positive constant scalar
curvature.  In dimension three this gives
\[
  \Ric(g_\infty(0))=\tfrac13 g_\infty(0),
  \qquad
  \sec(g_\infty(0))=\tfrac16 .
\]
Completeness and Bonnet--Myers imply compactness of the limit manifold and
hence globalize the convergence embedding, so the constant-curvature limit
metric can be transported back to the original manifold.

\section{Final Assembly}
\label{sec:conditional-assembly}

\begin{proof}[Assembly of Theorem~\ref{thm:main-hamilton-3d}]
Choose a smooth metric \(g_0\) with positive Ricci curvature on the closed
connected three-manifold \(M\).

\medskip
\noindent\textbf{1. Construct the finite maximal flow package.}
The one-metric short-time theorem supplies a seed segment.  The theorem
\leanref{flow_end_le} bounds the lifetime of every segment starting at
\(g_0\), and \leanref{ricci_flow_forward_unique} makes all such segments
compatible on their common intervals.  The supremum construction
\leanref{exists_max_flow} therefore gives
\[
  g(t),\qquad 0\le t<\omega, \qquad g(0)=g_0,
\]
with \(0<\omega<\infty\) and no extension past \(\omega\).  Maximality is
then combined with \leanref{extends_of_rmBounded} to obtain quantified
curvature unboundedness.  The class-uniform theorem enters only inside this
bounded-curvature extension implication, through an interior restart whose
common lifetime crosses \(\omega\).  The theorem
\leanref{hamilton_finite_time_flow_exists_on_closed_open} packages the solution, its initial value,
and curvature unboundedness as \leanref{HamiltonFiniteTimeFlow}; it does not store
maximality itself.  Its transitive axiom audit, like that of
\leanref{ricci_flow_uniform_existence}, contains only the three standard axioms
listed above.

\medskip
\noindent\textbf{2. Convert full-curvature blow-up to scalar blow-up.}
The endpoint bound is already part of Step~1; \leanref{hamilton_extinction_time_bound}
re-extracts it in the interface used by point selection.  Positive Ricci
curvature gives preserved nonnegative Ricci curvature, and in dimension three
this controls \(|\Rm|\) by the scalar curvature.  The quantified
full-curvature unboundedness stored in \leanref{HamiltonFiniteTimeFlow} therefore
implies
\[
  \sup_{M\times[0,\omega)}R=\infty.
\]
The corresponding Lean producers are \leanref{hamilton_extinction_time_bound},
\leanref{hamilton_ricci_nonnegative}, and \leanref{hamilton_scalar_blowup}.

\medskip
\noindent\textbf{3. Select and normalize a blow-up sequence.}
The point-selection theorem chooses \((x_i,t_i)\) and
\(R_i=R(x_i,t_i)\to\infty\) so that the rescaled flows
\[
  g_i(s):=R_i\,g\!\left(t_i+\frac{s}{R_i}\right)
\]
satisfy
\[
  R_{g_i}(x_i,0)=1,
  \qquad
  \sup_{M\times[-R_it_i,0]}R_{g_i}=1,
  \qquad
  R_it_i\longrightarrow\infty.
\]
This produces \leanref{HamiltonBlowup} and \leanref{hamiltonBlowupPointSelection} through
\leanref{hamilton_exists_blowup_point_sequence}.  Preserved Ricci nonnegativity then yields the
uniform Riemann-curvature bound and a common fixed backward window, represented
by \leanref{hamilton_rescaled_curvature_bound} and \leanref{hamilton_reference_radius_window}.

\medskip
\noindent\textbf{4. Apply the exact noncollapsing endpoint.}
Theorem~\ref{thm:ham3-volume-endpoint}, namely
\leanref{exists_hamilton_vol}, now produces one radius \(r_*>0\) and one
\(\kappa_*>0\) valid for every \(g_i\).  Thus the controlled time-zero balls
satisfy
\[
  \Vol_{g_i(0)}B_{g_i(0)}(x_i,r_*)\ge\kappa_*r_*^3.
\]
The curvature--volume bridge \leanref{flowInj_of_vol} converts this package into
one uniform injectivity-radius bound
\[
  \operatorname{inj}_{g_i(0)}(x_i)\ge\iota_*>0.
\]
This is the complete noncollapse-to-compactness handoff; no additional
noncollapsing assumption is introduced later.

\medskip
\noindent\textbf{5. Apply the exact compactness endpoint.}
The source sequence is complete and connected, the point-selection estimate
gives uniform curvature control, the specialized closed-window estimates give
the required curvature-derivative input, and Step~4 gives the time-zero
injectivity bound.  The checked Hamilton implementation uses the closed-window
construction directly: \leanref{metricCompactness} first constructs the
limiting time-zero metric, and
\leanref{hamilton_flow_upgrade_of_metric_compactness} supplies the compatible
flow upgrade.  We obtain a strictly increasing subsequence and a complete
pointed Ricci flow on \([-r_0^2,0]\),
\[
  (M_\infty,g_\infty(s),x_\infty),
\]
together with smooth pointed Cheeger--Gromov--Hamilton convergence on compact
space--time subsets of this fixed closed window.

\medskip
\noindent\textbf{6. Use pinching to identify the limit.}
Hamilton's improved pinching estimate on the original flow has the form
\[
  \frac{|\Ric^\circ|^2}{R^{2-\varepsilon}}\le C
  \qquad (0<\varepsilon<1).
\]
On the \(i\)-th normalized flow its effective coefficient is
\(C R_i^{-\varepsilon}\), which tends to zero.  Smooth pointed convergence
passes Ricci nonnegativity, the scalar normalization, and this estimate to the
limit.  Hence
\[
  \Ric^\circ(g_\infty(0))=0,
  \qquad
  R_{g_\infty(0)}(x_\infty)=1.
\]
The limit slice is Einstein; the contracted Bianchi identity makes its scalar
curvature constant.  The three-dimensional curvature identity then gives
\[
  \Ric(g_\infty(0))=\frac13 g_\infty(0),
  \qquad
  \sec(g_\infty(0))=\frac16.
\]
The exact curvature-transfer consumer is \leanref{round_at_zero_of_smooth_cgh}.

\medskip
\noindent\textbf{7. Return the round metric to the original manifold.}
The limit slice is complete and has constant positive Ricci curvature.
The Bonnet--Myers diameter/compactness
theorem implies \(M_\infty\) is compact.  Since constant positive sectional
curvature is known, a classical alternative is to use that the manifold is an isometric quotient of a standard 3-sphere, i.e,
a spherical space-form.  For all sufficiently large indices,
the pointed convergence embedding is therefore defined on all of
\(M_\infty\).  It is an embedding between manifolds of the same dimension, so
its image is open; compactness makes the image closed; connectedness of \(M\)
makes it surjective.  Thus it is a diffeomorphism.  Transporting
\(g_\infty(0)\) across this diffeomorphism gives a metric of constant positive
sectional curvature on \(M\).  This step is implemented by
\leanref{constant_positive_sectional_curvature_of_smooth_cgh} and \leanref{limit_to_orig}.

Combining Steps~1--7 gives \leanref{hamilton_admits_constant_positive_sectional_curvature}.  Finally the checked
constant-curvature/space-form equivalence \leanref{constant_positive_sectional_curvature_iff_spherical_space_form} yields
\leanref{hamilton_positive_ricci}, whose conclusion is both the constant-positive-
sectional-curvature metric and the spherical-space-form classification.
\end{proof}

At the proof-term level, the companion-project handoff is therefore exactly
\begin{center}
\begin{minipage}{0.94\linewidth}
\footnotesize\raggedright
\leanref{exists_hamilton_vol} \(\longrightarrow\)
\leanref{flowInj_of_vol} \(\longrightarrow\)
\leanref{hamilton_constant_positive_sectional_curvature_of_injectivity_radius_bound} \(\longrightarrow\)
\leanref{hamilton_constant_positive_sectional_curvature_of_pinching} \(\longrightarrow\)
\leanref{hamilton_admits_constant_positive_sectional_curvature} \(\longrightarrow\)
\leanref{hamilton_positive_ricci}.
\end{minipage}
\end{center}

\section{Development methodology and human oversight}
\label{sec:auto-formalization-workflow}

Language-model agents assisted both the Lean development and the preparation
of this manuscript.  Their use is methodologically relevant, but it is not a
premise of any mathematical theorem and it does not enlarge Lean's trusted
kernel.  We describe the workflow because it explains how a dependency tree of
this size was organized, how temporary frontier declarations were distinguished
from completed proofs, and where human mathematical judgment
remains indispensable.  The formal artifact and its trust boundary are
treated separately in Section~\ref{sec:artifact-verification}.

The short-time Ricci-flow development was produced while the authors were
iterating a prompt-level package called \texttt{auto-formalizing-skills}, whose
reference implementation is available at
\url{https://github.com/qinz1yang/auto-formalizing-skills}.  The package exposes
three principal commands.  \texttt{/prove} organizes a top-down proof attempt
around a specified headline declaration; \texttt{/check} runs an adversarial
search for false, vacuous, or overstrong statements; and \texttt{/handoff}
serializes the state of an unfinished run for later continuation.  The account
below describes the current documented protocol.  The protocol evolved during
the Ricci-flow work, so we do not claim retrospectively that every earlier
agent run obeyed every current rule.

\subsection{Separation of mathematical direction and proof execution}

A \texttt{/prove} run is organized in three tiers.  The \emph{leader} preserves
the headline theorem, the mathematical route, conventions, and decisions that
require author review.  A long-lived \emph{orchestrator} maintains the active
dependency frontier, checks returned work against the repository, and records
the tactical state on disk.  Short-lived \emph{workers} receive bounded tasks:
a prover attempts a Lean proof, a refuter attacks a proposed statement, and a
researcher searches the project and \texttt{mathlib} for existing declarations
and their exact types.

This is a division of context, not a hierarchy of trust.  A worker report is
not accepted because of the model that produced it.  The report must be
reconciled with the actual source, elaborated type, and compiler output.  The
leader likewise receives mathematical route summaries rather than raw proof
logs, so tactical detail does not displace the global proof architecture.
Persistent state files allow either tier to reconstruct its state after
context compaction; the repository, rather than a conversation transcript, is
the authoritative record.

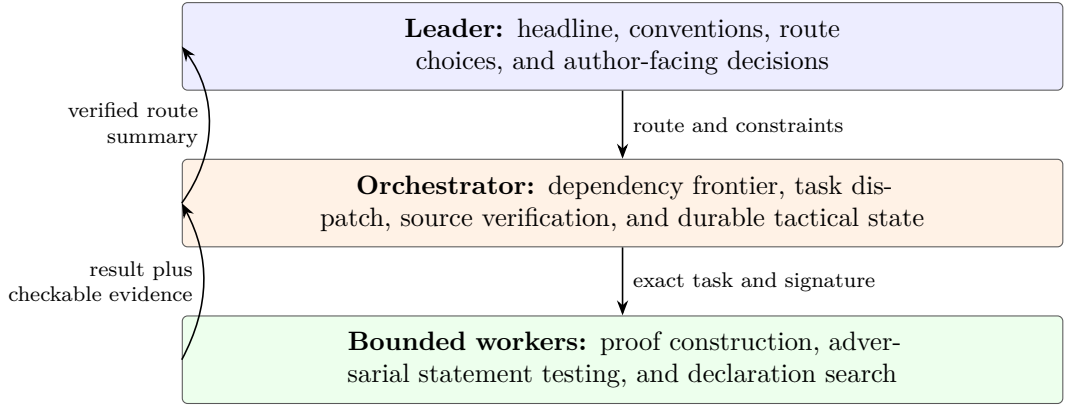
\begin{figure}[htbp]
\centering
\begin{tikzpicture}[
  node distance=9mm,
  box/.style={draw=black!65,rounded corners=2pt,align=center,
    font=\small,text width=.68\textwidth,inner sep=6pt},
  arrow/.style={-{Stealth[length=2mm]},semithick}
]
\node[box,fill=blue!7] (leader) {\textbf{Leader:} headline, conventions,
  route choices, and author-facing decisions};
\node[box,fill=orange!10,below=of leader] (orch) {\textbf{Orchestrator:}
  dependency frontier, task dispatch, source verification, and durable tactical state};
\node[box,fill=green!7,below=of orch] (workers) {\textbf{Bounded workers:}
  proof construction, adversarial statement testing, and declaration search};
\draw[arrow] (leader) -- node[right,font=\scriptsize] {route and constraints} (orch);
\draw[arrow] (orch) -- node[right,font=\scriptsize] {exact task and signature} (workers);
\draw[arrow] (workers.west) to[bend right=28]
  node[left,font=\scriptsize,align=right] {result plus\\checkable evidence} (orch.west);
\draw[arrow] (orch.west) to[bend right=36]
  node[left,font=\scriptsize,align=right] {verified route\\summary} (leader.west);
\end{tikzpicture}
\caption{Separation of mathematical direction, tactical coordination, and
bounded proof work.  Evidence flows upward only after comparison with source
and Lean output.}
\label{fig:auto-formalization-three-tier}
\end{figure}

\subsection{Consumer-driven top-down descent}

The method begins with the desired headline signature rather than with a large
collection of preparatory lemmas.  The current node is first attacked directly.
If a prerequisite is missing, the parent proof is drafted with a typed hole;
Lean's proof state is then used to expose the exact proposition the parent
consumes.  That proposition becomes a named child declaration in its intended
source file, and the parent is compiled from the child before descent
continues.  During this transient stage the child may have a \texttt{sorry}
body.  The placeholder marks the current frontier: it shows that the parent
\emph{assembles} from the child, but it does not show that the child is true.

This distinction gives two separate certificates.  A compiling glue proof
certifies that the children jointly imply the parent.  Only a proof term for
each child certifies the mathematical leaves.  The frontier therefore moves
downward until every declaration reachable from the headline is either an
existing checked theorem or a newly completed proof.  Temporary placeholders
outside the final transitive dependency closure are a repository-quality
issue, but they are not axioms of the headline theorem.

Child statements are kept \emph{consumer-minimal}: their strength and domain
are read from the proof that uses them, rather than chosen for aesthetic
generality.  This reduces several recurrent failure modes---for example,
demanding a pointwise estimate when only an integrated estimate is available,
quantifying over a larger time interval than the hypotheses control, or
asserting separate regularity when only joint regularity has been established.
Before a new child is introduced, the workflow searches both the project and
\texttt{mathlib}; formalizing an already available theorem under a new name is
neither useful nor harmless in a large dependency graph.

The short-time proof illustrates this descent.  The headline reduces to a
jointly regular Ricci--DeTurck solution, a smooth family of conjugating
diffeomorphisms, pullback naturality, and a theorem closing the PDE at $t=0$.
The Ricci--DeTurck node reduces in turn to the spectral maximal-regularity
engine, nonlinear estimates, metric realization, all-order regularity, and the
geometric representation identity.  These are mathematical dependencies, not
chapters invented after the code was complete: their formal signatures were
fixed by how the parent proof uses them.

\subsection{Adversarial statement review}

Proving a formally stated theorem and checking that it is the intended theorem
are different tasks.  A refuter is therefore assigned the null hypothesis that
a proposed statement is false, vacuous, or mis-scoped.  It searches for
degenerate witnesses, unconstrained families, endpoint errors, missing
nonzero hypotheses, and conclusions stronger than the available estimates.
When possible, a claimed counterexample is expressed as runnable Lean source.
The orchestrator reruns it before any statement is changed.

Refutation is most valuable at the birth of a load-bearing signature.  It
cannot replace proof: failure to find a counterexample is not evidence of
truth.  Nor does it replace author review.  In this project, particularly
sensitive interfaces included one-sided time derivatives, chart-to-intrinsic
regularity transfers, positivity of metric perturbations, and limit packages
whose fields could accidentally restate a desired conclusion.  The final
semantic audit of the short-time theorem is given in
Section~\ref{sec:artifact-verification}.

\subsection{Mechanical acceptance gates}

The current protocol treats agent messages as proposals and uses Lean and the
repository as the acceptance layer.  Depending on the task, the orchestrator
checks the edited source, the exact elaborated declaration, a narrow module
build, the project build, and the transitive axiom report.  A refuter's claimed
counterexample is rerun; a researcher's claimed declaration is opened at its
source and its type checked.  Source diffs are inspected to detect accidental
changes outside the assigned task.

For a completed \texttt{/prove} headline, the documented terminal condition is
that the reachable placeholder frontier and the in-scope source sweep are
empty, a fresh warning-free build succeeds, the checkpoint working tree is
clean, and \texttt{\#print axioms} for the headline contains exactly
\[
  [\texttt{propext},\ \texttt{Classical.choice},\ \texttt{Quot.sound}].
\]
The role of these three standard axioms, and the difference between this
theorem-specific result and a repository-wide placeholder scan, are explained
in Section~\ref{sec:artifact-verification}.  In particular, the terminal check
does not infer mathematical meaning from a green build.

\subsection{Human responsibility and scope of automation}

The authors remain responsible for the target theorem, definitions,
conventions, route choices, imported results, and the claim that each formal
signature expresses the mathematics described in this paper.  They also
review which declarations are suitable public interfaces.  This matters here:
the proof-producing agents successfully closed a specialized short-time
endpoint, but some internal PDE capstones still have signatures too entangled
with implementation data to be presented as general theorems.  Kernel
acceptance cannot make such an interface mathematically well designed.

Automated assistance was used for source search, decomposition, proof attempts,
counterexample search, refactoring suggestions, and manuscript drafting.  The
release identifies the source revision discussed here, and
Subsection~\ref{subsec:automated-assistance-disclosure} identifies the
scope and nature of that assistance.  No
language-model output is part of the trusted proof object.  Once elaboration is
complete, Lean's kernel checks the resulting term using the same rules whether
the term was written by a human, generated by an agent, or assembled from both.

With $1,945,081$ tracked Lean lines in the source release, exhaustive human line-by-line
inspection is neither a realistic nor a meaningful acceptance criterion.
Assurance is therefore layered: Lean's kernel checks every proof term;
theorem-specific axiom reports expose transitive trust assumptions; dependency
and placeholder checks identify unfinished or unsafe boundaries; and the
identified source revision makes those checks repeatable.
Human review remains indispensable at a different level.  Authors must check
that formal definitions match the intended geometry, theorem statements have
the advertised quantifiers and hypotheses, conventions are coherent, and
vendored or generated sources have correct provenance.  The kernel can certify
a proof of the formal statement, but it cannot detect that the formal statement
is the wrong mathematical claim.
\section{Formal artifact, kernel verification, and trust boundary}
\label{sec:artifact-verification}

Formalization status is a claim about a particular declaration at a particular
source revision.  A successful build, a theorem-specific axiom report, a
repository-wide placeholder scan, and a human comparison between formal and
informal statements answer different questions.  This section records the
evidence supplied for the one-metric, class-uniform, maximal-flow, and Hamilton
endpoints and examines the main ways circular reasoning could be concealed by
an otherwise valid Lean term.  The source revision is identified in
Subsection~\ref{subsec:source-release}.

\subsection{Verified short-time and Hamilton endpoints}
\label{subsec:verified-short-time-endpoint}

At the current authoritative checkout, focused Lean elaboration and axiom
queries succeed for the advertised dependency chain.  For each of
\leanref{ricci_flow_short_time_existence},
\leanref{ricci_flow_uniform_existence},
\leanref{hamilton_finite_time_flow_exists_on_closed_open},
\leanref{hamilton_admits_constant_positive_sectional_curvature}, and
\leanref{hamilton_positive_ricci}, the corresponding command
\begin{center}
  \texttt{\#print axioms DECLARATION}
\end{center}
returns exactly
\begin{equation}
  [\texttt{propext},\ \texttt{Classical.choice},\ \texttt{Quot.sound}].
  \label{eq:short-time-axioms}
\end{equation}
Equation~\eqref{eq:short-time-axioms} is theorem-specific: each query reports
the axioms used transitively by the proof term of the named declaration.  The
uniform theorem changes the quantifier order by choosing one lifetime for an
entire controlled metric class; its clean report is therefore independent
evidence that this former frontier has actually been discharged, rather than
hidden behind the one-metric theorem.  The endpoint reports contain neither
\texttt{sorryAx} nor a project-defined axiom.

The three reported constants are standard parts of the classical foundation
used throughout \texttt{mathlib}.  Proposition extensionality
(\texttt{propext}) identifies logically equivalent propositions;
\texttt{Classical.choice} supplies a choice function from a proof of
nonemptiness; and \texttt{Quot.sound} identifies related representatives in a
quotient.  Their appearance means that the proof is classical and uses Lean's
quotient foundation.  It does not signal an unfinished mathematical lemma,
and it should not be paraphrased as ``axiom-free'' or ``constructive.''

A green editor marker or successful build gives a related but weaker piece of
evidence.  It shows that Lean elaborated the source and that the kernel accepted
the resulting terms relative to the imported environment.  It does not by
itself reveal whether the theorem depends on a placeholder or custom axiom;
that is why the transitive \texttt{\#print axioms} output is recorded
separately.  Conversely, the axiom report should always be attached to the
exact elaborated signature, since a clean proof of the wrong statement is not
evidence for the intended theorem.

\subsection{Theorem closure, import closure, and repository state}

The theorem-specific axiom report should not be conflated with two broader
questions about the environment in which the theorem was compiled.  A
declaration can have a clean transitive axiom set even when its imported
environment contains declarations that its proof never uses; still less does
that report classify unrelated files elsewhere in the repository.  Conversely,
a source-tree scan cannot establish either the precise dependency closure or
the mathematical adequacy of an advertised theorem.  These scopes must
therefore be checked separately rather than inferred from directory layout or
textual proximity.

Three levels of audit should consequently not be conflated:

\begin{enumerate}[leftmargin=2em]
  \item A \emph{theorem-specific audit} records the exact type and transitive
  axioms of one advertised endpoint.
  \item An \emph{import-closure audit} inspects all declarations made available
  while compiling that endpoint, whether or not its proof term uses them.
  \item A \emph{repository audit} searches every project source for
  placeholders, declared axioms, unsafe escape hatches, and generated or
  vendored code that requires separate provenance.
\end{enumerate}

Equation~\eqref{eq:short-time-axioms} answers the first question for the
short-time theorem.  It is not, by itself, a certification of the other two
scopes.

\subsection{Kernel verification and semantic circularity}
\label{subsec:kernel-circularity}

Lean rules out ordinary proof-term cycles.  A nonrecursive declaration can
refer only to constants already admitted to the environment; mutually
recursive declarations are accepted only through Lean's checked recursion
mechanisms.  Recursive definitions must pass the required termination or
productivity checks, and a theorem cannot be proved by an arbitrary
nonterminating self-call.  A hidden \texttt{sorry} is represented by
an axiom and appears transitively as \texttt{sorryAx}.  In these syntactic and
type-theoretic senses, kernel checking detects the usual forms of circular
proof construction.

It cannot detect \emph{semantic} circularity caused by a poorly designed
statement.  Lean correctly proves $P\to P$ by returning its hypothesis;
likewise, a structure supplied as an input may already contain a field
equivalent to the conclusion later projected from it.  The kernel also cannot
decide that a formal definition of ``Ricci flow,'' ``strictly parabolic,'' or
``Cheeger--Gromov limit'' is weaker than the mathematical concept intended by
the authors.  These are theorem-design and interpretation questions, and they
require a human audit of signatures and proof data.

We performed that semantic audit adversarially for the short-time chain.  The
headline assumes only the manifold structure, compactness and
boundarylessness, positive dimension, and the initial smooth metric.  It does
not assume a Ricci-flow family, a DeTurck solution, an existence theorem, or a
regularity conclusion.  The all-orders forcing condition consumed by the
specialized nonlinear capstone is proved for the actual symmetrized DeTurck
nonlinearity by \leanref{deTurckRicci_forcingBootstrap_symm}; it is not promoted
to a headline hypothesis.  The forcing-ball retraction and the radial
metric-cone scaling are both proved inactive at the constructed fixed point.
The metric representation theorem then identifies the abstract tensor
solution with the genuine DeTurck right-hand side.  Gauge removal uses an ODE
for diffeomorphisms and Ricci pullback naturality, and the $t=0$ equation is
proved from interior evolution and continuity by
\leanref{ricci_flow_pde_at_zero}.  We found no field or hypothesis in this
chain that merely packages the desired Ricci-flow conclusion back into the
proof.

The audit also corrected a tempting but false description of the proof route.
A generic locally-Lipschitz existence theorem is present in the wider
development, but it is not called by the final short-time chain.  The actual
producer uses the symmetrized, radially controlled Sobolev nonlinearity and the
forcing-space fixed point described in
Subsection~\ref{subsec:short-time-nonlinear}.  Similarly, the outer definition
of \leanref{IsStrictlyParabolicMetricRHS} packages a
\leanref{HasPrincipalSymbol} witness.  That inner specification contains the
pointwise negative-isotropic identity and positivity for nonzero covectors,
but not a separate uniform-neighborhood estimate.  The analytic solution is
obtained by explicit estimates, not by projecting existence from the outer
predicate.

\subsection{Automated assistance disclosure}
\label{subsec:automated-assistance-disclosure}

The workflow described in Section~\ref{sec:auto-formalization-workflow} used
language-model agents for declaration search, proof decomposition, Lean proof
attempts, adversarial statement testing, debugging and refactoring proposals.
Language-model assistance was also used in drafting and revising the
manuscript, including the interleaved explanations of formal signatures.  All
such material remains subject to author review against the source and the
mathematics.

The project record distinguishes code generation, search, review, and prose
assistance.
This disclosure concerns provenance and research practice; it does not add an
item to the proof's trusted computing base.  The authors retain responsibility
for the statements, proofs, status classifications, citations, and exposition.

\section{Conclusion}

This paper presents the mathematical architecture and completed
Lean proof of Hamilton's three-manifold theorem.  The
one-metric short-time Ricci-flow declaration is a closed,
theorem-specifically audited milestone; its proof develops a spectral Sobolev
and maximal-regularity engine, solves the Ricci--DeTurck equation, and removes
the gauge through a jointly smooth diffeomorphism family.  The development also
contains substantial differential-geometric, evolution, maximum-principle,
three-dimensional algebra, and pinching layers.

Companion formalization projects supply the no-local-collapsing and
Cheeger--Gromov--Hamilton compactness components used by the
Hamilton route.  Their detailed constructions and verification will be
reported separately.  In the present paper they enter only through the closed
interfaces needed to obtain and analyze the blow-up limit.

The class-uniform short-time existence theorem
\leanref{ricci_flow_uniform_existence} is a quantitative strengthening of the
metricwise theorem: it chooses a common lifetime before the initial metric
varies in a controlled $C^3$ class.  Its proof is now closed, and the resulting
maximal-flow, noncollapsing, compactness, and limit-transfer chain closes
\leanref{hamilton_positive_ricci}.  Fresh axiom queries for every major handoff report
only \texttt{propext}, \texttt{Classical.choice}, and \texttt{Quot.sound}.
Subsection~\ref{subsec:source-release} identifies the release against which
these statements and links were checked.

\section*{Acknowledgments}

This formalization is built on the Lean~4 proof assistant and the
\texttt{mathlib} library \cite{deMouraUllrichLean4,mathlib}, and it would not exist
without the sustained work of the \texttt{mathlib} community.  Beyond that
general debt, the development directly consumes several identifiable bodies of
work, and we wish to acknowledge their authors specifically; published
accounts and expository materials on the relevant differential-geometric
foundations include
\cite{GouezelHigherOrderCalculus,BordgCavalleriDGLean,RothgangDGMathlib,
MassotVanDoornNashSphereEversion}.  The module-level attributions below reflect
the source history at the release identified in
Subsection~\ref{subsec:source-release}.  The manifold,
charted-space, and smooth-map infrastructure on which every geometric object
here is built is due in large part to S\'ebastien Gou\"ezel, Floris van Doorn,
Heather Macbeth, Yury Kudryashov, and Patrick Massot; the tangent- and
vector-bundle layers to Floris van Doorn, Heather Macbeth, Nicol\`o Cavalleri,
S\'ebastien Gou\"ezel, Patrick Massot, and Michael Rothgang; and the bundled
covariant-derivative interface that our Levi-Civita construction instantiates
to Patrick Massot, Michael Rothgang, and Heather Macbeth.  The
Riemannian-metric modules underlying our smooth-metric type are due to
S\'ebastien Gou\"ezel.  The spectral theorem for compact self-adjoint
operators at the core of our spectral Sobolev engine is due to Heather
Macbeth, with the surrounding inner-product-space and adjoint theory due to
Fr\'ed\'eric Dupuis, Zhouhang Zhou, S\'ebastien Gou\"ezel, and others.  The
Picard--Lindel\"of and integral-curve theory consumed by the gauge-removal
step is due to Winston Yin and Yury Kudryashov.  The Bochner-integration and
$L^p$/$L^2$ layers used throughout the analytic development are due to R\'emy
Degenne, Zhouhang Zhou, Yury Kudryashov, and S\'ebastien Gou\"ezel, and the
multilinear-map and operator-norm infrastructure behind the tensor model
fibers to S\'ebastien Gou\"ezel, Sophie Morel, Yury Kudryashov, Jan-David
Salchow, and Jean Lo.  For the calculus and topology foundations we thank,
among many others, Jeremy Avigad, Johannes H\"olzl, Mario Carneiro, Gabriel
Ebner, and Anatole Dedecker.
We are also deeply grateful to the many contributors named above who
generously answered our questions, clarified the design and intended use of
their formalizations, and offered technical guidance through personal
correspondence; their direct help substantially shaped this development.
These lists are necessarily incomplete, and we
thank the many further \texttt{mathlib} contributors, maintainers, and
reviewers whose work we use on every page.

We thank Richard Bamler, Kevin Buzzard, Michael Douglas,
Bogdan Georgiev, Daniel Halpern-Leistner, Ayush Khaitan, Robert Koirala, Alex Kontorovich,
Will Li, Zilu Ma, Pietro Monticone, Michael Rothgang, and Xin Zhou for helpful discussions, technical assistance, and encouragement
in connection with this Lean formalization project.  We thank Xiaodong Cao,
Peng Lu, Henry Shin, and Gang Tian for their encouragement.  We also thank Sam Buss, Alex Cloninger,
Michael Freedman,
Kiran Kedlaya,
Rayan Saab,
and Terence Tao for early and influential
discussions about Lean and the importance of formalized mathematics.

We are grateful to Yueqing Feng, Jesse Han, Jared Lichtman, and Auguste Poiroux for their contributions to
an earlier Ricci-flow formalization effort.  We also thank Jack Lee for discussions about differential
geometry in Lean and for generously making the \LaTeX{} sources of his
textbooks available for related formalization work.

We are especially grateful to Jack McCarthy for his foundational contributions
during the early stages of this project.  His substantial work on tensor,
multilinear-algebra, and bundle infrastructure has become a cornerstone of the
present development.

We are also especially grateful to Scott Armstrong and Julia Kempe.  The project
vendors their formalization of interior
De~Giorgi--Nash--Moser theory \cite{ArmstrongKempeDGNM}, with import-path
adaptations, in the project's \texttt{External/} directory; it supplies the
Euclidean Sobolev substrate of the project-local Rellich--Kondrachov step in the
short-time existence engine.  The vendored source is based on upstream revision
\href{https://github.com/scottnarmstrong/DeGiorgi/commit/4c1b307}{\nolinkurl{4c1b307}};
its provenance and license are recorded in \rffile{External/README.md}, and its
local modifications in \rffile{External/DeGiorgi/MODIFICATIONS.md}.
Finally, we thank the Lean developers and the Lean Focused Research
Organization for the Lean proof assistant itself.
\clearpage

\clearpage
\appendix
\section{Background formulas and auxiliary derivations}

This appendix records textbook background and auxiliary derivations used by the
main text.  It is explanatory rather than a theorem-to-Lean status ledger: a
displayed formula here should not be read as a claim that the project contains a
separate final Lean theorem with exactly that statement.
Table~\ref{tab:status-at-a-glance} and
Section~\ref{sec:artifact-verification} give the authoritative manuscript-level
status summary.  The main-text proof paragraphs elaborate the mathematical
scope of individual results checked at the source revision identified by the
release.

Links below point to the relevant public \texttt{mathlib} documentation from their first occurrences.

\subsection{Basic Riemannian geometry}\label{subsec:BasicRiemGeom}

The fundamental hierarchy of structures in Riemannian geometry is
\[
g \;\longrightarrow\; \nabla \;\longrightarrow\; \Rm ,
\]
where the \href{https://leanprover-community.github.io/mathlib4_docs/Mathlib/Geometry/Manifold/VectorBundle/Riemannian.html}{Riemannian metric} determines the Levi-Civita \href{https://leanprover-community.github.io/mathlib4_docs/Mathlib/Geometry/Manifold/VectorBundle/CovariantDerivative/Basic.html}{connection} and the connection
determines the curvature tensor.

\subsubsection{Levi-Civita, Riemann, Ricci, and scalar curvature}

The Levi-Civita connection
$\nabla: \mathfrak{X}(M) \times \mathfrak{X}(M) \to \mathfrak{X}(M)$
is the unique linear connection on the tangent bundle $TM$ that is both
torsion-free and metric-compatible.  The project-local construction and its
relationship to the surrounding library API are described in the main text.
The \textbf{Koszul formula} for
$\nabla$ is
\begin{align}
 2g\left(  \nabla_{U}V,W\right)   &  =U\left(  g\left(  V,W\right)  \right)
+V\left(  g\left(  U,W\right)  \right)  -W\left(  g\left(  U,V\right)  \right)
\nonumber \\
&   -g\left( U,  \left[  V,W\right] \right) -g\left( V, \left[  U,W\right]
\right)    + g\left( W, \left[  U,V\right]  \right)   .   \label{eq:koszul-formula}
\end{align}
The standard derivation expands the first line of
\eqref{eq:koszul-formula}, uses metric compatibility, and cancels the
remaining terms with the torsion-free condition.

\begin{definition}[Christoffel symbols]
\label{def:Christoffel_symbols}
    Let $(M,g)$ be a Riemannian manifold and let $(U,\mathbf{x})$ be a
    local coordinate chart, with $\mathbf{x}=(x^1,\dots,x^n)$.  Let
    $\{\partial_i\}_{i=1}^n$ denote the induced coordinate frame on
    $TM|_U$, where $\partial_i:=\frac{\partial}{\partial x^i}$.

    If $\nabla$ is the Levi-Civita connection of $g$, the
    \textbf{Christoffel symbols} are the functions
    $\Gamma_{ij}^k:U\to\mathbb R$ defined by expanding the covariant
    derivative in this coordinate frame:
    \begin{equation}
        \nabla_{\partial_i} \partial_j = \Gamma_{ij}^k \partial_k,
    \end{equation}
    where the Einstein summation convention is implied over the index $k$.
\end{definition}

The Christoffel symbols can be expressed in terms of the first partial
derivatives of the metric components:
\begin{equation}\label{eq:ChristoffelSymbols0th}
        \Gamma_{ij}^k = \frac{1}{2} g^{kl} \left( \partial_i g_{jl} + \partial_j g_{il} - \partial_l g_{ij} \right) ;
    \end{equation}
this follows from \eqref{eq:koszul-formula} and the fact that $[\partial_i,\partial_j]=0$.

The Riemann curvature tensor is defined by
\begin{equation}\label{eq:riemann-curvature-definition}
   R(X,Y)Z := \nabla_X (\nabla_Y Z) - \nabla_Y (\nabla_X Z) - \nabla_{[X,Y]} Z .
\end{equation}

The components of the Riemann curvature tensor are defined by
\begin{equation}
    R_{ijk}^l \partial_l = R(\partial_i,\partial_j)\partial_k .
\end{equation}
Using
\begin{equation*}
    \nabla_{\partial_i}(\nabla_{\partial_j}\partial_k) = \nabla_{\partial_i}(\Gamma_{jk}^\ell \partial_\ell)
    = \partial_i \Gamma_{jk}^{\ell}
+ \Gamma_{jk}^p \nabla_{\partial_i}\partial_p ,
\end{equation*}
we compute that
    \begin{equation}\label{eq:riemann-curvature-components}
        R_{ijk}^{\ell} = \partial_i \Gamma_{jk}^{\ell} - \partial_j \Gamma_{ik}^{\ell} + \Gamma_{jk}^{p}\Gamma_{ip}^{\ell} - \Gamma_{ik}^{p}\Gamma_{jp}^{\ell}.
    \end{equation}

As a $(0,4)$-tensor, $\Rm$ is defined by
\begin{equation}
    \Rm(X,Y,Z,W) := \langle R(X,Y)Z,W \rangle .
\end{equation}
Its components are defined by
\begin{equation}
    R_{ijkl} :=\Rm(\partial_i,\partial_j,\partial_k,\partial_l) .
\end{equation}
Raising and lowering the final index gives
\begin{equation}
    R_{ijkl}=R_{ijk}^m g_{ml} \quad \text{and} \quad R_{ijk}^l = R_{ijkm} g^{ml} .
\end{equation}

The dimension-three Riemann-from-Ricci identity, together with its Lean-facing
status and curvature convention, is stated once in
Lemma~\ref{lem:3d-curvature-identities} and is not repeated here.

\subsection{Tensor calculus}\label{subsec:TensorCalculus}

\subsubsection{Tensors}

Many geometric quantities appearing in Riemannian geometry---such as the metric,
curvature, and their covariant derivatives---are naturally expressed as tensors.

\begin{definition}[$(r,s)$-tensor]
\label{def:rs_tensor}
    Let $M$ be a smooth manifold and let $p\in M$.  Let $T_pM$ be
    the tangent space and let $T_p^*M$ be the cotangent space at $p$.

    An \textbf{$(r,s)$-tensor} at $p$ is a multilinear map
    \begin{equation}
        T: \underbrace{T_p^* M \times \dots \times T_p^* M}_{r \text{ times}} \times \underbrace{T_p M \times \dots \times T_p M}_{s \text{ times}} \to \mathbb{R}.
    \end{equation}
    Thus $T$ takes $r$ covectors and $s$ vectors as input, returns a real
    number, and is linear in each argument.

    A \textbf{smooth $(r,s)$-tensor field} is a smooth assignment of such a
    multilinear map to every point $p\in M$.  Equivalently, it can be viewed
    as a $C^\infty(M)$-multilinear map from $r$ smooth $1$-forms and $s$
    smooth vector fields to smooth functions:
    \begin{equation}
        T: \Omega^1(M)^r \times \mathfrak{X}(M)^s \to C^\infty(M),
    \end{equation}
where $\Omega^1(M)$ denotes the space of smooth $1$-forms and
$\mathfrak{X}(M)$ the space of smooth vector fields on $M$.
\end{definition}

\subsubsection{Covariant differentiation of tensors}

Since the Levi-Civita connection is defined on vector fields, we first extend
the notion of covariant differentiation to covariant objects, beginning with
1-forms.

\begin{definition}[Covariant derivative of a $1$-form]
\label{def:cov_dev_1form}
    Let $\nabla$ denote the Levi-Civita connection on $(M,g)$.  Let
    $\omega\in\Omega^1(M)$ be a smooth $1$-form and let
    $X,Y\in\mathfrak{X}(M)$.  The \textbf{covariant derivative} of
    $\omega$ in the direction $X$, denoted $\nabla_X\omega$, is the
    $1$-form defined by
    \begin{equation}
        (\nabla_X \omega)(Y) = X(\omega(Y)) - \omega(\nabla_X Y).
    \end{equation}
Equivalently, this is the product rule for the scalar function
$\omega(Y)$, with the correction term accounting for the variation of the
input vector field $Y$.
\end{definition}

The same rule extends to any smooth $(r,s)$-tensor field $T$.  Since
$\nabla_X$ already acts on functions, vector fields, and $1$-forms, the
general formula is dictated by the Leibniz rule:

\begin{definition}[Covariant derivative of an $(r,s)$-tensor]
\label{def:cov_dev_general}
    We extend the covariant derivative to a smooth $(r,s)$-tensor field $T$
    by requiring that $\nabla_X$ act as a derivation with respect to tensor
    contraction.  Explicitly, for any $X\in\mathfrak{X}(M)$, $r$
    $1$-forms $\omega^1,\dots,\omega^r$, and $s$ vector fields
    $Y_1,\dots,Y_s$, the tensor $\nabla_XT$ is defined by
    \begin{align}
        (\nabla_X T)(\omega^1, \dots, \omega^r, Y_1, \dots, Y_s) &= X \big( T(\omega^1, \dots, \omega^r, Y_1, \dots, Y_s) \big) \nonumber \\
        &\quad - \sum_{k=1}^r T(\omega^1, \dots, \nabla_X \omega^k, \dots, \omega^r, Y_1, \dots, Y_s) \nonumber \\
        &\quad - \sum_{l=1}^s T(\omega^1, \dots, \omega^r, Y_1, \dots, \nabla_X Y_l, \dots, Y_s).
    \end{align}
This is the corresponding product rule for tensor evaluation: each input
slot contributes a correction term.

    The \textbf{total covariant derivative} $\nabla T$ is the $(r,s+1)$-tensor
    field obtained by viewing the differentiation direction as an additional
    vector input.  Covector inputs precede vector inputs, and within the vector
    inputs the differentiation slot comes first:
    \begin{equation}
        (\nabla T)(\omega^1,\dots,\omega^r;X,Y_1,\dots,Y_s)
        :=(\nabla_XT)(\omega^1,\dots,\omega^r;Y_1,\dots,Y_s).
    \end{equation}
\end{definition}

\begin{remark}
    Observe that if $T$ is a $(0,1)$-tensor (a 1-form), the general definition agrees with Definition \ref{def:cov_dev_1form}. In this case ($r=0, s=1$), the first summation is empty, and the formula reduces to
    \begin{equation}
        (\nabla_X T)(Y_1) = X(T(Y_1)) - T(\nabla_X Y_1),
    \end{equation}
    which is exactly the definition of the covariant derivative of a 1-form.
\end{remark}

Since estimates for higher derivatives of the metric in local coordinates
are obtained by controlling covariant derivatives of curvature, we introduce
iterated covariant derivatives of tensor fields.

\begin{definition}[Higher covariant derivatives]
\label{def:high_order_cov_dev}
    We define the $k$-th covariant derivative of a smooth $(r,s)$-tensor field $T$ recursively.

    For $k=0$, we set $\nabla^{(0)} T = T$. For $k \geq 1$, the $k$-th covariant derivative $\nabla^{(k)} T$ is the $(r, s+k)$-tensor field defined as the total covariant derivative of the $(k-1)$-th derivative:
    \begin{equation}
        \nabla^{(k)} T := \nabla (\nabla^{(k-1)} T).
    \end{equation}
    In terms of arguments, if $T$ is a $(0,s)$-tensor, this recursion implies that
    \begin{equation}
        (\nabla^{(k)} T)(X_1, \dots, X_k, Y_1, \dots, Y_s) = (\nabla_{X_1} (\nabla^{(k-1)} T))(X_2, \dots, X_k, Y_1, \dots, Y_s)
    \end{equation}
    for vector fields $X_1, \dots, X_k, Y_1, \dots, Y_s$.
    Note that under this convention, the \emph{first} argument $X_1$ corresponds to the \emph{outermost} (most recent) differentiation operation.
\end{definition}

\begin{remark}
We compute that
\begin{align*}
    (\nabla^{(2)}Z)(X,Y) & = (\nabla_X (\nabla Z)) (Y) \\
    & = \nabla_X( (\nabla Z)(Y)) - (\nabla Z) (\nabla_X Y) \\
    & = \nabla_X(\nabla_Y Z) - \nabla_{\nabla_X Y} Z
\end{align*}
for vector fields $X,Y,Z$.
Therefore we may rewrite the definition
\eqref{eq:riemann-curvature-definition} of the Riemann curvature tensor as
\begin{equation}
    R(X,Y)Z = (\nabla^{(2)}Z)(X,Y) - (\nabla^{(2)}Z)(Y,X) .
\end{equation}
\end{remark}

\subsubsection{Index notation for tensors}

A choice of local coordinates determines a coordinate basis for the
tangent bundle, allowing tensor fields to be expressed through their
components in index notation.

\begin{definition}[Tensor components and covariant derivatives]
\label{def:tensor_components}
    Let $\alpha$ be a smooth $(0,s)$-tensor field. The \textbf{components} of $\alpha$ in the local frame are the smooth functions defined by:
    \begin{equation}
        \alpha_{i_1 \cdots i_s} := \alpha (\partial_{i_1}, \ldots, \partial_{i_s}).
    \end{equation}
    The $k$-th covariant derivative $\nabla^{(k)} \alpha$ is a $(0, k+s)$-tensor field. We denote its components by the symbol $\nabla_{j_1}\cdots \nabla_{j_k} \alpha_{i_1 \cdots i_s}$, defined as the evaluation of the tensor on the coordinate basis:
    \begin{equation}
        \nabla_{j_1}\cdots \nabla_{j_k} \alpha_{i_1 \cdots i_s} := (\nabla^{(k)}\alpha) (\partial_{j_1}, \ldots, \partial_{j_k}, \partial_{i_1}, \ldots, \partial_{i_s}).
    \end{equation}
\end{definition}

\subsubsection{Ricci identities}

The non-commutativity of the covariant derivative is measured precisely by the Riemann curvature tensor. While this is often defined for vector fields, it extends naturally to all tensor fields. We first establish the action of curvature on 1-forms, then extend it to arbitrary tensors via the derivation principle.

\begin{lemma}[Action of curvature on $1$-forms]
\label{lem:curvature_on_1forms}
    Let $\omega \in \Omega^1(M)$ be a smooth $1$-form and let $X, Y, Z \in \mathfrak{X}(M)$. Then
    \begin{equation}\label{eq:ricci-identity-one-form}
        (\nabla_X \nabla_Y \omega)(Z) - (\nabla_Y \nabla_X \omega)(Z) - (\nabla_{[X,Y]} \omega)(Z) = - \omega(R(X,Y)Z).
    \end{equation}
\end{lemma}

\begin{proof}
By the definition of the dual connection, expanding the two second derivatives
and cancelling the first-derivative terms gives
\[
([\nabla_X,\nabla_Y]\omega)(Z)
=[X,Y](\omega(Z))
-\omega(\nabla_X\nabla_YZ-\nabla_Y\nabla_XZ).
\]
On the other hand,
\[
(\nabla_{[X,Y]}\omega)(Z)
=[X,Y](\omega(Z))-\omega(\nabla_{[X,Y]}Z).
\]
Subtracting and using the definition of $R(X,Y)Z$ proves the formula.
\end{proof}

\begin{remark}
We may rephrase equation \eqref{eq:ricci-identity-one-form} as
\begin{equation}
    (\nabla^{(2)} \omega)(X, Y, Z)
- (\nabla^{(2)} \omega)(Y, X, Z) = - \omega(R(X,Y)Z).
\end{equation}
Equation \eqref{eq:ricci-identity-one-form} says, in local coordinates, that
\begin{equation}\label{eq:ricci-identity-one-form-coordinates}
    \nabla_i \nabla_j \omega_k - \nabla_j \nabla_i \omega_k = - R_{ijk}^m \omega_m = - R_{ijk\ell}g^{\ell m} \omega_m .
\end{equation}
\end{remark}

Generalizing \eqref{eq:ricci-identity-one-form-coordinates} to covariant
tensors, we have:

\begin{theorem}[Ricci identity for $(0,s)$-tensors]
\label{thm:ricci_identity}
    Let $\alpha$ be a smooth $(0,s)$-tensor field. In local coordinates, the commutator of second covariant derivatives is given by:
    \begin{equation}
        \nabla_{i}\nabla_{j}\alpha_{k_{1}\cdots k_{s}} - \nabla_{j}\nabla_{i}\alpha_{k_{1}\cdots k_{s}}
        = - \sum_{q=1}^{s} \sum_{m=1}^n R_{ijk_{q}}^{m} \alpha_{k_{1}\cdots k_{q-1} m k_{q+1}\cdots k_{s}}.
    \end{equation}
\end{theorem}

\begin{proof}
Set
\[
\mathcal R(X,Y):=[\nabla_X,\nabla_Y]-\nabla_{[X,Y]}.
\]
Because the connection obeys the Leibniz rule, $\mathcal R(X,Y)$ acts as a
derivation on the tensor algebra.  Lemma~\ref{lem:curvature_on_1forms} therefore
gives the invariant formula
\[
(\mathcal R(X,Y)\alpha)(Z_1,\dots,Z_s)
=-\sum_{q=1}^s
\alpha(Z_1,\dots,R(X,Y)Z_q,\dots,Z_s).
\]
Taking $X=\partial_i$, $Y=\partial_j$, and $Z_q=\partial_{k_q}$, using
$[\partial_i,\partial_j]=0$ and
$R(\partial_i,\partial_j)\partial_{k_q}
=R_{ijk_q}{}^m\partial_m$, yields the stated coordinate identity.
\end{proof}

\begin{remark}[Ricci identity for $(r,s)$-tensors]
    More generally, if $\beta $ is any $\left(  r,s\right)  $-tensor field, one has the
commutator
\begin{align}
\left[  \nabla_{i},\nabla_{j}\right]  \beta_{k_{1}\cdots k_{s}}^{l_{1}%
\cdots l_{r}}  &  := \nabla_{i}\nabla_{j}\beta_{k_{1}\cdots k_{s}}^{l
_{1}\cdots l_{r}}-\nabla_{j}\nabla_{i}\beta_{k_{1}\cdots k_{s}}^{l_{1}\cdots l_{r}} \nonumber \\
&  =\sum_{p=1}^{r} \sum_{m=1}^n R_{ijm}^{l_{p}}\beta_{k_{1}\cdots k_{s}}^{l_{1}\cdots
l_{p-1}ml_{p+1}\cdots l_{r}}-\sum_{q=1}^{s} \sum_{m=1}^n R_{ijk_{q}}^{m}%
\beta_{k_{1}\cdots k_{q-1}mk_{q+1}\cdots k_{s}}^{l_{1}\cdots l_{r}}. \label{eq:ricci-identity-mixed-tensor}
\end{align}
\end{remark}

For example, for a $(1,1)$-tensor $\beta$, we have
\begin{equation}\label{eq:ricci-identity-one-one-tensor}
    \nabla_i\nabla_j \beta_k^l - \nabla_j \nabla_i \beta_k^l =
    R_{ijm}^l\beta_k^m - R_{ijk}^m \beta_m^l.
\end{equation}

\subsubsection{Inner products and norms of tensors}

\begin{definition}[Tensor inner product]
\label{def:tensor_inner_product}
    Let $(M, g)$ be a Riemannian manifold and let $\alpha, \beta$ be $(0,s)$-tensor fields.
    Their \textbf{inner product} is the scalar function $\langle \alpha, \beta \rangle \in C^\infty(M)$ defined at each point $p$ by contracting with respect to a local orthonormal basis $\{e_k\}_{k=1}^n$:
    \begin{equation}
        \langle \alpha , \beta \rangle := \sum_{k_1, \dots, k_s=1}^{n} \alpha (e_{k_1}, \dots, e_{k_s}) \beta (e_{k_1}, \dots, e_{k_s}).
    \end{equation}
    In general local coordinates, this corresponds to the full contraction:
    \begin{equation}
        \langle \alpha , \beta \rangle = g^{i_1 j_1} \dots g^{i_s j_s} \alpha_{i_1 \dots i_s} \beta_{j_1 \dots j_s}.
    \end{equation}
    The norm of a tensor is defined by
\begin{equation}
    |\alpha|^2 := \langle \alpha, \alpha \rangle ,
\end{equation}
so $|\alpha| = \sqrt{\langle \alpha, \alpha \rangle}$.
\end{definition}

For example, if $T$ is a $(0,s)$-tensor and $k\geq 0$, then
\begin{equation}
    |\nabla^{(k)}T|^2 = g^{i_1 j_1} \cdots g^{i_{s+k} j_{s+k}} \nabla_{i_1}\cdots \nabla_{i_k}T_{i_{k+1} \cdots i_{k+s}} \nabla_{j_1} \cdots \nabla_{j_k}T_{j_{k+1} \cdots j_{k+s}} .
\end{equation}

\subsection{The Laplace operator}

For the purposes of this appendix, the rough Laplacian is recorded as the
metric trace of the second covariant derivative.  The main text explains how
the project realizes such identities through invariant, coordinate, or local
frame interfaces when they are used in Lean.

\subsubsection{The rough Laplacian and the divergence}

The Laplacian is the fundamental second-order differential operator in
geometric analysis, obtained by taking the metric trace of the second
covariant derivative.

\begin{definition}[Rough Laplacian]
\label{def:rough_laplacian}
    Let $(M, g)$ be a Riemannian manifold and let $\alpha$ be a $(0,s)$-tensor field.
    The \textbf{rough Laplacian} $\Delta \alpha$ is defined by the metric trace of the second covariant derivative over its first two arguments:
    \begin{equation}
        \Delta \alpha := \operatorname{tr}_{1,2} \big( \nabla^{(2)} \alpha \big).
    \end{equation}
    Explicitly, if $\{e_k\}_{k=1}^n$ is a local orthonormal frame, this trace is given by
    \begin{equation}
        \operatorname{tr}_{1,2} \big( \nabla^{(2)} \alpha \big) (\cdot, \dots, \cdot) := \sum_{k=1}^n (\nabla^{(2)} \alpha)(e_k, e_k, \cdot, \dots, \cdot).
    \end{equation}
    In local coordinates, this corresponds to the contraction:
    \begin{equation}
        (\Delta \alpha)_{i_1 \dots i_s} = g^{jk} \nabla_j \nabla_k \alpha_{i_1 \dots i_s}.
    \end{equation}
\end{definition}

The divergence is a fundamental first-order operator in geometric
analysis, obtained by taking the metric trace of the covariant derivative.

\begin{definition}[Divergence]
\label{def:divergence}
    Let $(M, g)$ be a Riemannian manifold and let $\alpha$ be a $(0,s)$-tensor field with $s \ge 1$.
    The \textbf{divergence} of $\alpha$ is the $(0,s-1)$-tensor field defined by the metric trace of the covariant derivative $\nabla \alpha$ over its first two arguments:
    \begin{equation}
        \Div(\alpha) := \operatorname{tr}_{1,2} (\nabla \alpha).
    \end{equation}
    In local coordinates, this corresponds to contracting the derivative index with the first index of $\alpha$:
    \begin{equation}
        \big(\Div(\alpha) \big)_{i_2 \dots i_s} = g^{jk} \nabla_j \alpha_{k i_2 \dots i_s}.
    \end{equation}
With respect to an orthonormal frame $\{e_i\}_{i=1}^n$,
\begin{equation}
  \Div(\alpha) = \sum_{i=1}^n (\nabla_{e_i} \alpha ) (e_i,\cdot,\ldots,\cdot) .
\end{equation}
\end{definition}

\begin{remark}[Factorization of the Laplacian]
\label{rem:laplacian_div_grad}
    It follows directly from the definitions that the Laplacian factors as the divergence of the covariant derivative:
    \begin{equation}
        \Delta \alpha = \Div(\nabla \alpha).
    \end{equation}
    Note that this composition is well-typed: if $\alpha$ is a $(0,s)$-tensor, then $\nabla \alpha$ is a $(0,s+1)$-tensor (satisfying the rank requirement for divergence), and the resulting divergence maps it back to a $(0,s)$-tensor.
\end{remark}

\bibliographystyle{amsalpha}
\bibliography{references}

\end{document}